\documentclass[twoside,openright,12pt]{amsbook}
\usepackage{geometry} 
\usepackage[T1]{fontenc}
\usepackage{lmodern}
\usepackage{indentfirst}
\usepackage{mathrsfs}
\usepackage{pgfplots}
\usepackage{bm}
\usepackage{pgfplotstable}
\usepackage{tikz-3dplot}
\usepackage{tikz-cd}
\usepackage{wrapfig}

\usepackage{amssymb,latexsym,amsmath,amsthm} 
\usepackage{graphicx}
\usepackage{epstopdf}
\usepackage{enumerate}
\usepackage{color}
\usepackage{hyperref}

\usepackage{tikz}
\usetikzlibrary{matrix,arrows,decorations.pathmorphing}

\numberwithin{section}{chapter}
\numberwithin{equation}{chapter}

\newtheorem{theorem}{Theorem}[chapter]
\newtheorem{lemma}[theorem]{Lemma}
\newtheorem{corollary}[theorem]{Corollary}
\newtheorem{proposition}{Proposition}[chapter]

\theoremstyle{definition}
\newtheorem{definition}{Definition}[chapter]
\newtheorem{example}{Example}[chapter]
\newtheorem{examples}{Examples}[chapter]
\newtheorem{ex}{Exercise}[chapter]

\theoremstyle{remark}
\newtheorem{remark}{Remark}[chapter]

\def\bbC{\mathbb{C}}
\def\bbE{\mathbb{E}}

\def\bbK{\mathbb{K}}
\def\bbN{\mathbb{N}}

\def\bbP{\mathbb{P}}
\def\bbR{\mathbb{R}}
\def\bbQ{\mathbb{Q}}
\def\bbZ{\mathbb{Z}}

\def\bB{\mathbf{B}}

\def\bS{\mathbf{S}}

\def\b1{\mathbb{1}}

\def\cA{\mathcal{A}}
\def\cB{\mathcal{B}}
\def\cC{\mathcal{C}}

\def\cE{\mathcal{E}}
\def\cF{\mathcal{F}}

\def\cI{\mathcal{I}}
\def\cJ{\mathcal{J}}

\def\cM{\mathcal{M}}
\def\cN{\mathcal{N}}
\def\cO{\mathcal{O}}
\def\cP{\mathcal{P}}
\def\cR{\mathcal{R}}

\def\cT{\mathcal{T}}
\def\cU{\mathcal{U}}
\def\cV{\mathcal{V}}
\def\cW{\mathcal{W}}

\def\fB{\mathfrak{B}}
\def\fC{\mathfrak{C}}

\def\fR{\mathfrak{R}}
\def\fT{\mathfrak{T}}

\def\sfA{\mathsf{A}}
\def\sfS{\mathsf{S}}
\def\sfT{\mathsf{T}}

\def\e{\varepsilon}

\def\R{\mathop{R}}

\makeindex

\frontmatter

\title[ADVANCES  IN INFORMATION  GEOMETRY]{EXPLORING INFORMATION GEOMETRY: \\
RECENT ADVANCES\\ AND \\CONNECTIONS TO TOPOLOGICAL FIELD THEORY
\\
\vspace{1cm}
\vspace{2cm}
\vspace{2cm}}

\author{Authors:}
\author{No\'emie C. Combe\\ Philippe G. Combe\\ Hanna K. Nencka}

\email{
}

\address{}

\curraddr{}

\thanks{This research is part of the project No. 2022/47/P/ST1/01177 cofunded by the National Science Centre and the European Union's Horizon 2020 research and innovation programme under the Marie Skłodowska-Curie grant agreement no. 945339 \includegraphics[width=1cm, height=0.5cm]{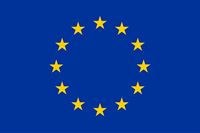}.
For the purpose of Open Access, the author has applied a CC-BY public copyright licence to any Author Accepted Manuscript (AAM) version arising from this submission. The first author extends heartfelt gratitude to the organizers of the Wis\l{}a Baltic Workshop 2023  and  Maria Ulan for the invitation to lecture on information geometry, an opportunity that ultimately led to the creation of this book for students, researchers, and all those curious about this fascinating field. 
~
The first author is also grateful to Maxim Kontsevich for asking stimulating questions on the topic, that inspired deeper exploration and helped illuminate aspects of the subject that had remained unclear. This textbook is, in part, an effort to clarify these questions and provide a structured introduction to the topic. The first author thanks Katarzyna Pietruska-Pa\l{}uba for a great motivational discussion, which led to the creation of the final shape of this manuscript. We thank Janusz Grabowski for his interest in the subject. Warm thanks go to Matilde Marcolli also. Finally, we would like to express all our gratitude to Yuri I. Manin, whose untimely death has saddened us all. }

\address{University of Warsaw, Department of Mathematics \\ 
(MIMUW)\\
Ulica Banacha 2, 02-097 Warsaw}
\email{n.combe@uw.edu.pl}

\begin{document}

\maketitle
\tableofcontents

\begin{center}
 \includegraphics[scale=0.4]{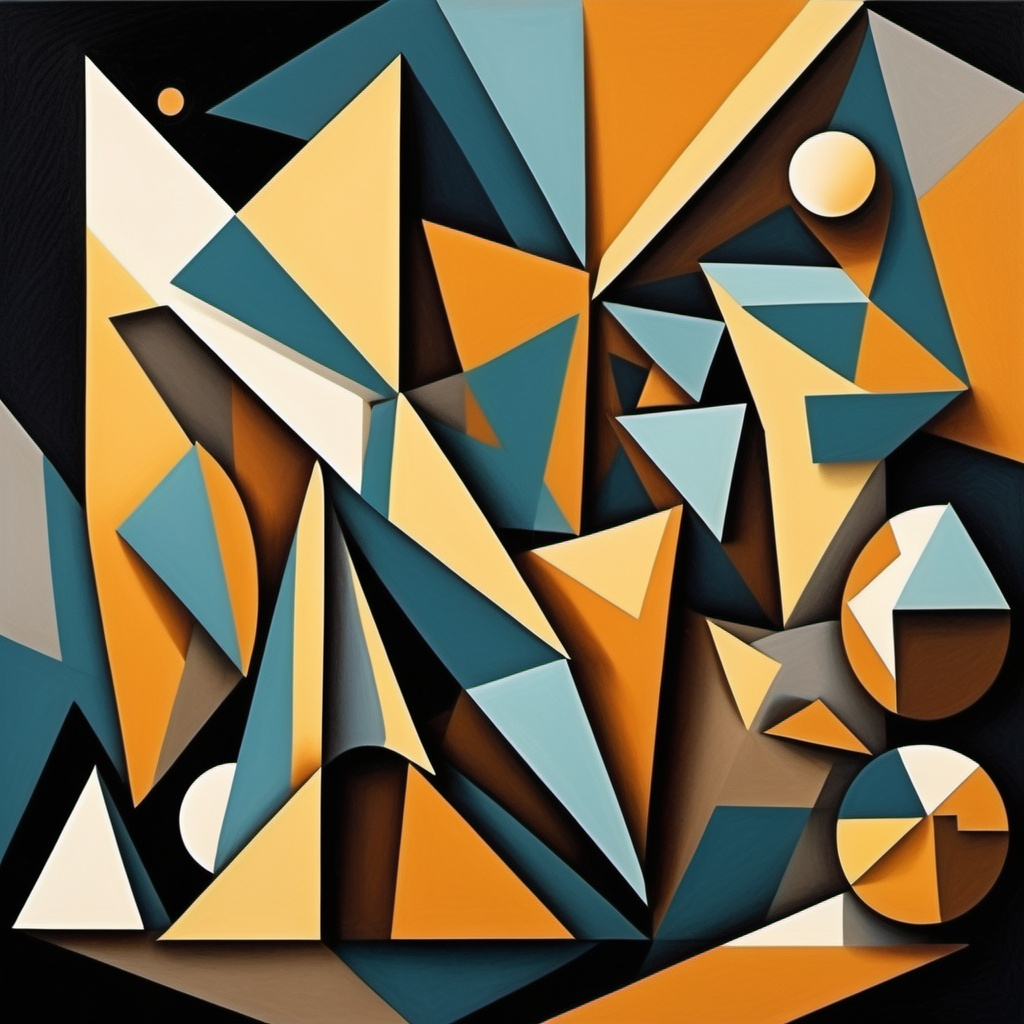}.
\end{center}

\section{Foreword}

Welcome to this Introductory Textbook on Geometry and Information Geometry! Whether you are a student, researcher, or simply someone curious about this fascinating field, this book is designed to provide you with a clear and accessible entry point into the world of geometry and its modern extension into information geometry.

Geometry, in its many forms, has been a cornerstone of mathematics for centuries. With the advent of information theory and machine learning, new geometric structures have emerged that help us better understand complex probabilistic models, optimization, and data analysis. Information geometry is one such powerful framework, blending classical differential geometry with modern applications in probability and statistics.

This book is structured to guide you from the fundamentals of topology and differentiable manifolds to the more advanced concepts of probability geometry and Frobenius manifolds. The progression is designed to be intuitive, allowing you to build your knowledge step by step.

Throughout the chapters, you will find exercises to test your understanding, as well as solutions to some of the more challenging problems to aid your learning. To complement the material, you are encouraged to follow along with the accompanying YouTube video (see the QR code below), which aligns with this textbook and served as an inspiration for its creation.

\begin{center}
 \includegraphics[scale=0.3]{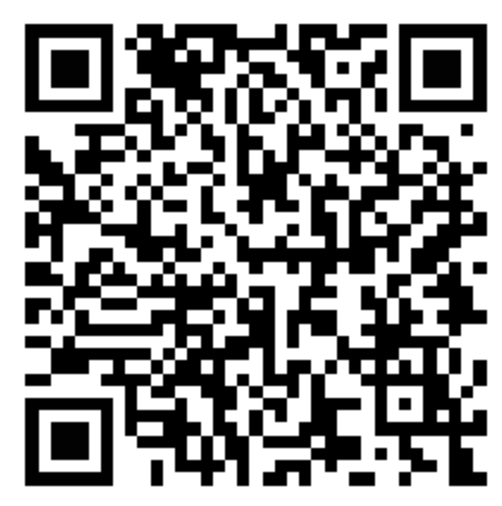}.
\end{center}

We hope this book provides you with both insight and enjoyment as you embark on your journey into geometry and information geometry.

\mainmatter

\begin{abstract}  
This introductory text arises from a lecture given in Göteborg, Sweden, given by the first author and is intended for undergraduate students, as well as for any mathematically inclined reader wishing to explore a synthesis of ideas connecting geometry and statistics. At its core, this work seeks to illustrate the profound and yet natural interplay between differential geometry, probability theory, and the rich algebraic structures encoded in (pre-)Frobenius manifolds.  

\, 

The exposition is structured into three principal parts. The first part provides a concise introduction to differential topology and geometry, emphasizing the role of smooth manifolds, connections, and curvature in the formulation of geometric structures. The second part is devoted to probability, measures, and statistics, where the notion of a probability space is refined into a geometric object, thus paving the way for a deeper mathematical understanding of statistical models. Finally, in the third part, we introduce (pre-)Frobenius manifolds, revealing their surprising connection to exponential families of probability distributions and, discuss more broadly, their role in the geometry of information.  At the end of those three parts the reader will find stimulating exercises.

\, 

By bringing together these seemingly distant disciplines, we aim to highlight the natural emergence of geometric structures in statistical theory. This work does not seek to be exhaustive but rather to provide the reader with a pathway into a domain of mathematics that is still in its formative stages, where many fundamental questions remain open. The text is accessible without requiring advanced prerequisites and should serve as an invitation to further exploration.  
\end{abstract}

\chapter{Introduction}

In what follows, we shall be concerned with a surprising fusion of ideas, namely, that of differential geometry and probability theory (or, more generally, statistics). At first glance, these two domains might appear unrelated, yet their synthesis gives rise to a highly rich and intricate mathematical structure, leading to the emergence of a relatively young field, known as the \emph{geometry of information}. The relevance of this framework extends beyond pure mathematics, with applications in artificial intelligence (such as in language models like ChatGPT) and machine learning. Given the accelerating pace of developments in these areas, it is imperative to pursue the deeper mathematical underpinnings of this theory.  

\,

One of the great advantages of geometry is that it offers an intuitive and often visual means of understanding complex situations. This feature allows us to circumvent difficulties that might otherwise seem insurmountable. More precisely, at the heart of the geometry of information lies a fundamental object: the manifold of probability distributions, where the probability measures belong to a specific class of distributions. 

\,

To make matters precise, recall that a \emph{probability space} consists of a triple $(\Omega, \mathcal{F}, P)$, where $\Omega$ is the sample space (the set of possible outcomes), $\mathcal{F}$ is a $\sigma$-algebra of measurable subsets of $\Omega$, and $P$ is a probability measure assigning to each event in $\mathcal{F}$ a real number in the interval $[0,1]$. This object $(\Omega, \mathcal{F}, P)$ serves as a rigorous mathematical model for phenomena arising in nature. The challenge we face is to endow this structure with additional geometric data, thereby giving rise to a natural geometric space. However, this task is far from trivial.  

\,

A first step is to recognize that the set $\mathcal{F}$ carries a particular mathematical structure: that of a $\sigma$-algebra. The significance of $\sigma$-algebras lies in their closure properties: if $(\Omega, \mathcal{F})$ is a \emph{measurable space}, then $\mathcal{F}$ is closed under countable unions, countable intersections, and complements. However, when seeking to introduce a geometric framework, it is often more convenient to work directly with probability distributions rather than the space $(\Omega, \mathcal{F})$ itself.  

\,

Probability distributions come in a vast array of different families. To clarify our discussion, let us recall the three principal classes of probability distributions:

\begin{enumerate}[1)]
    \item Discrete probability distributions, either with finite or infinite support.
    \item Absolutely continuous probability distributions, whose support may be:
    \begin{enumerate}[a)]
        \item A bounded interval.
        \item An interval of length $2\pi$ (directional distributions).
        \item A semi-infinite interval, $[0, \infty)$.
        \item The entire real line.
    \end{enumerate}
    \item Probability distributions with variable support.
\end{enumerate}

For our purposes, we restrict attention to the second class, focusing in particular on absolutely continuous distributions supported on intervals of length $2\pi$, namely, \emph{wrapped exponential distributions}. 

A remarkable fact, which has emerged through careful analysis, is that probability distributions of this type possess a hidden geometric structure. The apparatus of differential geometry provides the natural language for uncovering and describing this structure. Objects such as connections, parallel transport, curvature, and flatness serve as fundamental tools in the construction of a geometric framework for probability distributions. This allows for the realization of a novel class of geometric spaces: \emph{the manifolds of probability distributions}. 

More concretely, consider a measurable space $(\Omega, \mathcal{F})$, where $\mathcal{F}$ is a $\sigma$-algebra on $\Omega$. If we consider a family of parametrized probability distributions on $(\Omega, \mathcal{F})$, then the space of such distributions naturally inherits the structure of a manifold. 

\,

The study of these manifolds, particularly in the case of exponential families, leads to unexpected connections with \emph{Topological Field Theory} (TFT). The relation between exponential families and TFT is not accidental; rather, it is deeply encoded in the mathematical structures underlying both domains. Indeed, Topological Field Theory is intimately linked to the celebrated \emph{Witten–Dijkgraaf–Verlinde–}
\emph{Verlinde (WDVV) equations}, which in turn play a central role in the theory of \emph{Frobenius manifolds}, developed and studied by Dubrovin, Manin, Kontsevich, and others. 

\,

A fundamental aspect of our approach, which renders manifest the deep relation between the manifold of probability distributions of exponential type and the (pre-)Frobenius manifold structures, lies in the delicate matter of choosing an appropriate system of coordinates. Indeed, the very possibility of perceiving this connection in an explicit and natural manner is contingent upon the identification of a privileged class of coordinates—ones that reflect, in their very definition, the intrinsic geometry of the underlying structures. The art, then, is not merely to introduce coordinates, but to do so in a manner that unveils (rather than obscures) the hidden algebraic and differential properties inherent in the space.  

\,

Thus, in certain cases, we are able to establish a deep and precise connection between the WDVV equations and the geometry of information. This, in turn, suggests that the study of probability distributions, when viewed through a geometric lens, is far richer than initially expected and might hold profound implications for both mathematical physics and information theory.  

\section{Summary and synthesis}
This introductory textbook is structured in three interrelated parts, each designed to build a solid foundation in modern mathematics and lead the reader toward the emerging field of information geometry.

\, 

\section*{\bf Part 1: Manifolds, Topology, and Geometry}
\, 

\subsection*{\bf Chapter 1: Foundations of General Topology}
\,

The book begins with a modern treatment of general topology, reinterpreting Kuratowski’s early ideas \cite{Ku72} in a contemporary framework. Readers are introduced to essential topological concepts—such as open sets, continuity, convergence, and compactness—establishing the language and tools necessary for later discussions.

\, 

\subsection*{\bf Chapter 2: Topological and Modelled Manifolds}~

Building on the groundwork of topology, the text moves into the realm of manifolds. Chapter 2 explains topological manifolds—spaces that locally resemble Euclidean space—while also presenting modelled manifolds based on S. Lang’s influential works \cite{L95,L99}. For additional reference, the text suggests consulting \cite{Mil65} to cover aspects not fully developed in this book. This chapter serves as a bridge, transitioning from abstract topological spaces to concrete geometric structures.

\,

\subsection*{\bf Chapter 3: Differentiability and Gateaux Derivatives}
Differentiability is explored next with a focus on Gateaux derivatives. This notion is particularly useful in infinite-dimensional settings, especially in probability theory where a norm may not be available. The text carefully contrasts Gateaux differentiability with the stronger concept of Fréchet differentiability, ensuring students understand the subtleties and applications of both approaches.
\,

\subsection*{\bf Chapter 4: Fiber Bundles}
Fiber bundles are introduced as indispensable tools in differential geometry. This chapter outlines the local-to-global perspective that fiber bundles offer, demonstrating how complex geometric structures can be understood by piecing together simpler, locally trivial components.

\,

\subsection*{\bf Chapter 5: Connections, Parallel Transport, and Sheafs}
Delving deeper into the geometric framework, Chapter 5 covers connections, parallel transport, and covariant derivatives—key concepts for understanding how geometric data evolves along a manifold. Classical references such as \cite{KoNo96} and \cite{Sik} provide further reading on these fundamental topics. Additionally, an introduction to sheafs is provided, offering a gentle entry point into their role in modern geometry, with \cite{KS90} recommended for a complete reference. This material lays the groundwork for applying geometric techniques within information geometry later in the book.

\,

\section*{\bf Part 2: Probabilities, Statistics, and Related Topics}
\,

\subsection*{\bf Chapter 6: Basics of Probability and Statistics}~

The second part shifts focus to probability theory and statistics. It introduces standard definitions, theorems, and methodologies, employing familiar examples—such as coin flipping—to illustrate key concepts. The chapter also discusses the Radon–Nikodym derivative, linking measure theory with probabilistic reasoning. References such as \cite{Bi,Pa67} for measure theory and \cite{Fe66} for statistics are suggested for further reading.

\,

\subsection*{\bf  Chapters 7 $\&$ 8: Categorical and Geometrical Structures in Probability}~

These chapters are particularly innovative, blending philosophy with mathematics. Beginning with a discussion inspired by Klein’s geometry and Plato’s ideas, the text presents a novel perspective on how categorical structures naturally arise in the study of probability distributions. By considering manifolds of probability measures and the role of Markov kernels, the reader is guided through a modern generalization of classical geometry into a probabilistic and categorical setting. Key references for this section include \cite{Am85,Am97,MoCh89,MoCh91-1,MoCh91-2}.

\,

\section*{\bf Part 3: Frobenius Manifolds and Information Geometry}
\, 

\subsection*{\bf  Chapters 9,10 $\&$ 11 and Beyond: Frobenius Structures in Modern Research}
The final part of the book centers on Frobenius manifolds—geometric structures that elegantly encapsulate the interplay between algebra and geometry, coming from 2D Topological Field Theory. Here, the text explains how Frobenius structures emerge naturally within the framework of information geometry. For foundational references on Frobenius manifolds and related topics, the book cites \cite{Du,Man99}.

\,

Chapter 11 also integrates cutting-edge research from 2020 onward \cite{CoMa,CoCoNen}, demonstrating the latest developments in the field. A notable highlight is the discussion of learning methods pioneered by researchers such as D. Ackley, G. Hinton, and T. Sejnowski \cite{AHS}, showing how deep learning techniques relate to the broader mathematical framework of information geometry.

\,

{\bf  Conclusion}
This textbook provides a concise, accessible, and modern overview of several core areas of mathematics—topology, geometry, probability, and category theory—culminating in the study of information geometry. Through careful exposition and a judicious selection of topics, it equips students with the necessary tools to explore this young and rapidly evolving branch of mathematics, bridging classical theory with modern research and applications.

\part{Manifolds, Topology and Geometry}

\chapter{Foundations in General Topology}
General topology, also known as \emph{point-set topology}, provides the foundational language for modern mathematics by studying the most fundamental properties of sets and their structures. Topology focuses on the intrinsic properties of spaces that remain unchanged under continuous transformations.

A topological space is a set 
$X$ equipped with a topology—a collection of subsets called open sets that satisfy specific axioms and ensuring consistency with notions of continuity and convergence. This framework generalizes many familiar mathematical spaces, including metric spaces and Euclidean spaces.

\, 

Let us mention some key concepts in general topology:

\begin{itemize}
\item Topological spaces: A set equipped with a topology that defines continuity and convergence.
\item Open and closed sets: Fundamental building blocks of topology that generalize intervals in real analysis.
\item Basis and sub-basis: Tools for constructing topologies using (minimal) generating sets.
\item Continuity: A function between topological spaces is continuous if the pre-image of every open set is open.
\item Homeomorphisms: Structure-preserving maps that define when two spaces are topologically equivalent.
\item Compactness: a property that allows to generalize the notion of a closed and bounded subset in an Euclidean space. 
\item Connectedness: A criterion telling whether a space can be separated into disjoint open sets.
Separation axioms: Conditions such as $T_0,T_1,T_2$ (Hausdorff), which control how points can be distinguished in a space.
  \end{itemize}
  
These fundamental concepts serve as the basis for more advanced areas of mathematics, including analysis, geometry, algebraic topology, and functional analysis.

For readers already familiar with general topology, this section can be skipped. However, for those who need a refresher, the following discussion will provide a structured and intuitive approach to these essential topics before moving on to more advanced material.

\section{Topological space}

\begin{definition} [{\bf Topological space}]

 Let $X$ be a set and denote by $\cP(X)$ the \emph{power set}, that is, the set of all its subsets. A topology on $X$ is a distinguished collection of subsets, $\cT \subset \cP(X)$, which we regard as specifying the admissible open sets. This collection satisfies the following axioms:
 \begin{enumerate}
\item The set itself and the empty set belong to the topology: \[\emptyset, X \in  \cT.\]
\item Any arbitrary union of open sets remains open: \[\forall\, \{\cU_{i}\}_{i\in I} \subset  \cT,\quad \bigcup\limits_{i\in I} \cU_{i} \in  \cT.\]
\item Any finite intersection of open sets remains open:  \[\forall \, \cU_{1},\dots,\cU_{N} \in \cT,  \quad \bigcap\limits_{i=1}^{N} \cU_{i} \in  \cT.\]
  \end{enumerate}
   \end{definition}
   
  Thus, a topology provides the language of continuity: it defines what it means for a function to be continuous without reference to distances, relying only on the structure of open sets. 
  
  \, 
  
   The elements of $ \cT$ are called the {\sl open} of the topology; the conditions (1), (2) and (3) form  the axioms of a topology.

\, 

To summarize, we can therefore also state this definition by saying that a topology is a collection of subsets of $X$, called open, which must verify that
\begin{enumerate}
\item the empty set and $X$ are open;
\item any union of open sets is an open;
\item A finite intersection of open sets is an open.
\end{enumerate}

We will sometimes write $(X,\cT)$ to specify that we are considering a set $X$ equipped with its topology $\cT$.

\subsection{Basis}
\begin{definition}[{\bf Basis}]
A basis $\cB$ for a topology $\cT$ is a family of elements of $\cT$ such that every $\cU\in \cT$ is the union of elements of $\cB$. 
\end{definition}
We then refer to this as a topology generated by $\fB$.

\noindent 

An equivalent definition can be given. A basis $\fB$ for a topology $\cT$ is a family of elements of $\cT$ such that for each $x\in X$ and $\cU\in \cT$, with $ x\in \cU$, there exists $\cB\in \fB$ such that $x\in \cB$ and $\cB\subset \cU$.

\begin{example}
A basis for the usual topology of $\bbR$ is provided by the set of open intervals $\{]a,b[ \, \mid\,  a<b, a,b\in \cR\}$.
\end{example}

An example of a topology widely used in practice is the metric topology.

\vspace{5pt}
\begin{definition}[{\bf Metric space}]~\label{D:EspMet} 
A metric space is a set $M$ endowed with a notion of distance, that is, a function $d : M\times M \to \bbR^{+}$ that satisfies:
\begin{enumerate}
\item $d(x,x)\geq 0, \, \forall \, x\in M$.
\item $d(x,y)= 0 \iff x=y, \, \forall \, x,y\in M$.
\item $d(x,y)=d(y,x), \, \forall \, x,y\in M$.
\item $d(x,z)=d(x,y)+d(y,z), \, \forall \, x,y,z\in M$.
\end{enumerate}
The topology of $M$ is generated by the open balls $B_{r}(a)=\{x\in M \mid d(a,x)<r$\}.
\end{definition}
\begin{ex}\label{Ex:1}
The Cantor ternary set $\cC$
 is created by iteratively deleting the open middle third from a set of line segments. One starts by deleting the open middle third $(\frac{1}{3},\frac {2}{3})$ from the interval 
$[0,1]$, leaving two line segments: 
$[0,\frac {1}{3}]\cup [\frac {2}{3},1]$. One continues iterating the a similar procedure for the remaining line segments. The Cantor set is constituted from all points in the interval $[0,1]$ that are not deleted at any step in this infinite process.
Is the Cantor set a metric space?
\end{ex}
\begin{definition}[{\bf Neighborhood}]
\

\begin{enumerate}
\item A neighborhood of a point $x \in X$ is the set $\cN(x)$ containing an open set that contains the point $x$.
\item A neighborhood of a set $A \subset X$ is the set $\cN(A)$ containing an open set that contains $A$.
\end{enumerate}
\end{definition}

\subsection{Glossary of topologies}
\begin{definition}[{\bf Trivial and discrete topology}]

\

\begin{list}{$\triangleright$}{}
\item {\bf trivial topology}: The trivial topology\index{Topology! Trivial} on a non-empty set $X$ consists in taking as open sets the empty set $\emptyset$ and the entire space $X.$
\item {\bf Discrete topology}: The discrete topology\index{Topology! Discrete} on a non-empty set $X$ consists in taking as open sets the elements of $\cP(X)$ the set of subsets of $X.$
\end{list}
\end{definition}

\begin{definition}[{\bf Coarser and finer topology}]

\

\begin{list}{$\triangleright$}{}

\item {\bf Coarser topology} :\index{Topology!coarse} if $\cT_{1}$ and $\cT_{2}$ are two topologies on $X$ such that $$\cT_{1}\subset \cT_{2},$$ $\cT_{1}$ is said to be coarser than $\cT_{2},$
\item {\bf Fine topology} :\index{Topology!Fine} if $\cT_{1}$ and $\cT_{2}$ are two topologies on $X$ Such that $$\cT_{1}\subset \cT_{2},$$ then $\cT_{2}$ is said to be finer than $\cT_{1}$.
\end{list}
\end{definition}

\noindent $\bullet$ The coarsest topology is the trivial topology.

\noindent $\bullet$ The finest topology is the discrete topology.

\begin{definition}[{\bf Induced topology on a subset}]
Let $(X, \cT_{X })$ be a topological space and $A \subset X$ be a subset. We define a topology $ \cT_{A}$ on $A$ by setting:
\[
\cT_{A}=\{U\cap A\mid U\in \cT_{X}\} .
\]
\end{definition}

In other words, we take as open sets of $A$ the intersections of open sets of $X$ with $A$.

\begin{definition}[{\bf Quotient topology}]
Let $(X,\cT)$ be a topological space and $\cR$ be an equivalence relation on $X$. Let the map 
\[
p : X\to X/\cR,\ x\mapsto [x],
\]
 associate an element $x\in X$ with an equivalence class of $X$. The open sets of the quotient topology on $X/\cR$ are the subsets $\cV\subset X/\cR$ such that $\cV=p^{-1}(\cU)$, where $\cU\in \cT$. \end{definition}

\section{Separated space}
In the study of topological spaces, separation properties play a crucial role in understanding the structure of a space and the behavior of functions defined on it.
\begin{definition}[{\bf Hausdorff  space - $\mathbf{T_2}$}]\index{Space!Hausdorff } 

A topological space $(X,\cT)$ is called Hausdorff (separated or $\mathbf{T_2}$) if for any pair of distinct points ${\scriptstyle M}$ and ${\scriptstyle N}$, we can find two open sets $\cU_{M}, \cU_{N}$ with ${\scriptstyle M}\in\cU_{M} , {\scriptstyle N}\in\cU_{N}$ and $\cU_{M}\cap \cU_{N}=\emptyset$.
\end{definition}

\begin{definition}[{\bf Normal space}]\index{Space!normal}
 A topological space $(X,\cT)$ is normal if it is Hausdorff and if for any pair of disjoint closed sets $ F_{1}$ and $F_{2}$, there exist two disjoint open sets $\cU_{1}$ and $\cU_{2}$ such that $F_{1}$ is included in $\cU_{1}$ and $F_{2}$ in $\cU_{2}$.
\end{definition} 

\begin{theorem} A Hausdorff space $X$ is normal if and only if it satisfies the following condition:

\, 

For every closed subset $F\subset X$ and every open subset $\cU$ containing $F$, there exists an intermediate open subset $\cV$ containing $F$ satisfying 
\[F\subset \cV\subset \overline \cV \subset \cU.\] 
\end{theorem}
Notice that here $\overline \cV $ is properly defined in Sec. \ref{S:topProp} under the terminology of \emph{closure} of a set: $\overline{\{p\}}=\{p\}$.  
\begin{example}
An open set $\cU\subset \bbR^{n}$ is a Hausdorff space and moreover a normal space.
\end{example}

\begin{definition}
    A topological space is ${\bf T_1}$ if for any two distinct points $x$ and $y$, there exist open sets $U$ and 
$V$ such that: 

\begin{itemize}
    \item $x\in U$, $y\notin U$,
    \item $y\in V$, $x\notin V$.
\end{itemize}
In other words, each point is closed (its complement is open).
\end{definition}

\begin{ex}\label{Ex:2}
Prove or disprove that any Hausdorff space is a $\mathbf{T_1}$ space. A $\mathbf{T_1}$ space is a space in which any set consisting of one point is closed. 
\end{ex}

$\star$  Warning! There exist ${\bf T_1}$ space that are not ${\bf T_2 }$
\begin{ex}\label{Ex:T_1neqT_2}
Give an example of a ${\bf T_1}$ space that is not ${\bf T_2 }$.
\end{ex}

 \section{Continuity}
\begin{definition}[{\bf Continuity}]\index{Application!Continuity}
Consider two topological spaces $(X,\cT),(X',\cT')$. A map $f:(X,\cT)\to,(X',\cT')$ is continuous if for any open set $\cU'\in \cT'$ its inverse image $g^{-1}(\cU')$ is an open set of the topology $\cT$.
\end{definition}

\begin{theorem}
\

\begin{enumerate}
\item A map $f:(X,\cT)\to(X',\cT')$ is continuous at the point $x_0\in X$ if for every neighborhood $\cN(f(x_0)) \subset X'$  there exists a neighborhood $\cN(x_0 )\subset X)$ of $x\in X$ such that $f(x) \in \cN(f(x_0))$, whenever $x \in\cN(x_0)$.

The map $f$ is continuous on $X$ if it is continuous at all $x\in X$.

\item A map $f$ is continuous if and only if the sequence
$\{f(x_{\beta})\}$ converges to $f(x)$ when the sequence $\{x_{\beta}\}$ converges to $x$.
\end{enumerate}
\end{theorem}

\begin{ex}\label{Ex:cont}
Show that for any pair of continuous functions $f,g:X\to Y$, where $X$ is a topological space $X$ and $Y$ is a Hausdorff space,
the set
    \[\{x\in X \,|\, f(x)=g(x)\} \]
    is closed in $X$. 
\end{ex}

\begin{definition}[{\bf Homeomorphism}]\index{Application!Hom\'eomorphism}
A homeomorphism between two topological spaces $(X,\cT_{X}), (Y,\cT_{Y})$ is a continuous bijective map $h :(X,\cT_{X})\to (Y,\cT_{Y})$ whose inverse map is continuous.
\end{definition}

\begin{definition}[{\bf Open map}]
A continuous map $f: X\to Y$ is called open if the image of any open set of $X$ is an open set of $Y$.
\end{definition}

\begin{definition}[{\bf Closed map}]
A map $f:X\to Y$ between topological spaces is said to be closed if the subset 
\[
\Gamma_f=\{(x,f(x)) \mid x\quad \text{lies in a domain of}\quad  f\}\subset X\times Y,
\] 
The set $\Gamma_f$ is closed in $X\times Y$. The set $\Gamma_f$ is called the \emph{graph} of $f$.
\end{definition}

\section{Topological properties of sets}\label{S:topProp}

\subsection{Open and closed sets}
We will recall some definitions and theorems related to various properties of subsets of a topological space $X$.

\begin{definition}

\

\begin{list}{$\triangleright$}{}

\item {\bf Closed subset} :\index{Set!Subset!Closed}
A set $F \subset X$ is closed if it is the complement of an open set.

\item{\bf Limit point}:\index{Topology! Limit point}
A point $x\in X$ is a limit point of $A \subset X$ if every neighborhood $\cN_{x}$ of $x$ contains at least one point $a\subset A$ different from $x$.

\item {\bf Closure}:\index{Set!Subset!Adherence} \index{Set!Subset!Closure}
The closure $\bar A$ of $A \subset X$ is the union of $A$ and all its limit points. It is the smallest closed set containing $A$.

\item{\bf Dense subset} :\index{Set!Subset!Dense}
The set $A \subset X$ is dense in $ X$ if $\bar A=X$.

\item{\bf Separable} :\index{Set!Separable}
A topological space $X$ is separable if it contains a countable dense subset .

\item {\bf Interior} : \index{Set!Subset!Interior}
The interior $\AA$ of the set $A \subset X$ is the largest open set contained in $A$.

\item {\bf Nowhere dense set} : \index{Set!Subset!Nowhere dense}\label{D:npd}

The set $A \subset X$ is nowhere dense in $X$ if
its adhesion $\bar A$ has an empty interior.
\end{list}
\end{definition}

\begin{definition}[{\bf Baire space}]\label{D:EBaire}
A topological space is called a Baire space if the intersection of any countable family of dense open sets remains dense.

Equivalently, a topological space is Baire if the union of any countable collection of closed sets with empty interior also has an empty interior.
\end{definition}
Complete metric spaces, as well as locally compact Hausdorff spaces, provide natural examples of Baire space. 
\begin{example}
\

\begin{list}{$\triangleright$}{}
\item The set of integers is nowhere dense in the set of real numbers equipped with the usual topology.

\item The set of real numbers whose decimal expansion only includes the digits $0$ or $1$ is nowhere dense in the set of real numbers.
\end{list}
\end{example}

\begin{ex}\label{Ex:3}
    Is the set of rational numbers equipped with the subspace topology a Baire space? Write a proof. 
\end{ex}

\begin{ex}\label{Ex:4} 
Is the set of real numbers equipped with standard topology a Baire space? Write a proof. 
\end{ex}
\begin{theorem}[{\bf Baire's Category Theorem}]
\

\begin{enumerate}
\item Every complete metric space is a Baire space.
\item Every locally compact topological space is a Baire space.
\end{enumerate}
\end{theorem}

\begin{theorem}
\

\begin{enumerate}
\item A subset $A \subset X$ is open if and only if it is a neighborhood of each of its points.

\item A subset $A \subset X$ is closed if and only if it contains all its limit points.
\end{enumerate}
\end{theorem}
 
\subsection{Compactness}
\begin{definition}
\

\begin{list}{$\triangleright$}{}
\item {\bf (Open) Covering}:\index{Set!Covering} A family $\cU$ of open sets $U _{\alpha}\subset X$ is an open covering \index{Set!Covering!Open} if $\displaystyle \bigcup_{\alpha} U_{\alpha}=\cU$.

\item {\bf (Open) Subcovering}: An open subcovering of a covering\index{Set!Covering!Subcovering} $\cU$ is a subset of $\cU$ which is itself an open covering.

\item {\bf Compact }:\index{Set!Covering!Compact} A subset $A \subset X$ is compact if it is Hausdorff
and each covering has a finite subcovering.

The condition of being Hausdorff  can be dropped  in a large number of situations.
\item {\bf Relatively compact}: \index{Set!Covering!Relatively compact}A subset $A \subset X$ is {\bf relatively compact} if its closure $\bar A$ is compact.

\item {\bf Locally compact}: \index{Set!Subset!Locally compact} A subset is locally compact if every point has a compact neighborhood.

Notice that, Euclidean spaces are locally compact but not compact.

\item {\bf Paracompact }: \index{Space!Paracompact} A Hausdorff space is paracompact if every covering $\cU=\{ U _{\alpha}\}$ has a locally finite covering $\cV=\{V _i\}$ such that, for every $V_i$ there exists a $U _{\beta}$ containing it $(V_i\subset U _{\alpha})$.

\item {\bf Compactification}:\index{Topology! Compactification} A compactification of a topological space $X$ is a pair $(h,K)$ where $K$ is a compact space and $h$ is a homeomorphism of $X$ onto a dense subspace of $K$.
\end{list}
\end{definition}

\begin{theorem}
\

\begin{enumerate}
\item A compact subspace of a Hausdorff space is necessarily closed.

\item Every closed subspace of a compact space is compact.

\item {\bf Bolzano-Weierstrass theorem}:\index{Topology!Theorem!Bolzano-Weierstrass} A Hausdorff space is compact if and only if every sequence has a convergent subsequence.

\item {\bf Heine--Borel theorem}\index{Topology!Theorem!Hein-Borel} The compact subsets of $\bbR^n$ are the closed bounded subsets of $\bbR^n$.
\end{enumerate}
\end{theorem}
The Heine-Borel theorem concerns finite-dimensional spaces, it is not necessarily true for an arbitrary topological space. In particular, a closed bounded set with a non-empty interior of an infinite-dimensional normed vector space is  never compact, for the norm topology.
 
\begin{theorem} 
Let $K$ be a compact space. The image of $K$ under a continuous map $f$ is also compact.
\end{theorem}

\begin{corollary} Given a compact space $K$, any continuous function on $K$ attains on $K$ a minimum and a maximum value.
\end{corollary}
\begin{ex}\label{Ex:5}
 Prove that  the Cantor set is compact.
\end{ex}
\begin{ex}\label{Ex:6}
 Consider the set $K$ of all functions $ f: [0, 1] \to [0, 1]$ equipped with the Lipschitz condition i.e. $|f(x) - f(y)| \leq |x - y|$ for all $x, y \in [0,1]$. Consider on $K$ the metric induced by the uniform distance 
\[d(f,g)=\sup_{x\in [0,1]}|f(x)-g(x)|.\] Prove that the space $K$ is compact.   
\end{ex}

 \subsection{Connectedness}
\begin{definition} [{\bf Connectedness}]\index{Space!Connectivity} A topological space $X$ is said to be connected if it cannot be described as the union of two disjoint (non-empty) open sets. Otherwise $X$ is said to be non-connected.

\, 

A subspace of a topological space is said to be connected if it is connected for the induced topology.
\end{definition}
\begin{ex}\label{Ex:7}
Is the General Linear group $GL_n(\bbR)$
(i.e. the space of square matrices of size $n\times n$ with non-null determinant) a connected space? Provide a proof.    
\end{ex}
\begin{theorem}
For a topological space $X$ the following conditions are equivalent:
\begin{enumerate}
\item $X$ is connected;
\item $X$ cannot be the union of two non-empty closed sets;
\item The subsets of $X$ that are both open and closed are $X$ and the empty set $\emptyset$;
\item $X$ cannot be the union of two non-empty separable sets;
\item The continuous functions from $X\to \{0,1\}$ are the constant functions.
\end{enumerate}
\end{theorem}

\begin{definition} [{\bf Locally connected}]\index{Space!Locally connected}
A topological space $X$ is said to be locally connected if every neighborhood of every point $x \in X$ contains a connected neighborhood.
\end{definition}

\begin{examples}
~\begin{enumerate}
\item The Euclidean line, the Euclidean plane, the cube $I^{n}$ (without boundary) are connected and locally connected.
\item Any isolated point is a point of local connectivity.
\item The curve $y=\sin\frac{1}{x}$ on the interval $(0,1]$, is connected but not locally connected. This is because the set contains the point $(0,0)$ but it is not possible to connect the function to the origin.
\item Let $\bbR^{2}$ be equipped with the standard topology and let $K=\{\frac{1}{n} \mid n\in \bbN\}$. We call the comb space the set
\[C= (\{0\}\times[0,1])\cup (K\times[0,1])\cup ([0,1]\times\{0\})\]
considered as a subspace of $\bbR^{2}$ equipped with the induced topology. The comb is a connected space that is not locally connected.

\end{enumerate}

\end{examples}

\noindent $\star$ Warning! A space can be connected without being locally connected.

\begin{ex}\label{Ex:8}
 Is the curve $y=\sin\frac{1}{x}$ on the interval $(0,1]$ locally connected?
\end{ex}

\begin{ex}\label{Ex:9}
    Compute how many connected components has the real quartic surface: 
    \[f(u,v)=3.2u-v^2-\frac{1}{20}=0\]
     where $u=x^2+y^2+z^2,v=x^2y^2+y^2z^2+x^2z^2$.

\, 

Do the same for: 
\[f(u,v)=8u-v^2-\frac{1}{20}=0\]
\end{ex}

\begin{definition}[{\bf Covering space}]
The covering of a topological space $X$ is a pair $(\tilde X,f)$, where $\tilde X$ is a connected and locally connected space and $f$ is a continuous map from $\tilde X$ onto $X$, such that for every  neighborhood of $x$, $x\in \cN(x)$, the restriction of $f$ to each connected component $C_{\alpha}$ of $f^{-1}(\cN(x))$ is a homeomorphism from $C_{\alpha}$ to $\cN(x)$.
\end{definition}

\begin{definition}[{\bf Simply connected}]
A topological space $X$ is simply connected if $X$ is connected and locally connected and any covering $(\tilde X,f)$ is isomorphic to the trivial covering $(X,Id)$ where $Id$ is the identity map.
\end{definition}
  
 \begin{example}
\

\begin{list}{$\triangleright$}{}
\item $\bbR$ is simply connected (roughly speaking it has no "holes" and every loop is shrunk to a point).
\item The circle $S^{1}=\bbR/\bbZ$ is not simply connected.
\end{list}
\end{example}

A universal covering space is the "simplest" simply connected space that covers a given space. Rigorously, it is defined as follows.
\begin{definition}[{\bf Universal covering}]
$(\tilde X,f)$ is a universal covering of the space $X$ if it is a covering and if $\tilde X$ is simply connected.
\end{definition}
\begin{example}
  A covering of $S^{1}$ is given by the pair $(\bbR,\pi)$, where $\pi$ is the canonical projection given by \[\pi(t)=\exp{2\pi\imath t}\] and $t$ is a real parameter.    
\end{example}

\begin{ex}\label{Ex:10}
Does the torus $S^1\times S^1$ have a universal covering? If yes, give it explicitly. 
\, 

Demonstrate that the Kähler torus $T^n=\bbC^n/\Lambda$ in $\bbC^n$ where $\Lambda$ is a lattice in $\bbC^n$ (meaning that it is a discrete subgroup isomorphic to $\mathbb{Z}^{2n})$  has a universal covering space.
\end{ex}

\begin{definition}[{\bf Locally simply connected}]
A topological space $X$ is locally simply connected if every point $x\in X$ has at least one simply connected neighborhood.
\end{definition}
\begin{example}
    The sphere $S^2$ is locally simply connected, since every point on the sphere has a simply connected neighborhood. For every point on $S^2$ one can choose a small open neighborhood around that point, homeomorphic to an open disk in $\bbR^2$. 
\end{example}

\begin{ex}\label{Ex:11}
We consider a space which consists of infinitely many circles of decreasing radius, all tangent to a single point. This is called a Hawaiian earring. 
Prove that this space is not locally simply connected.
\end{ex}

 \begin{definition}[{\bf Isomorphism of covering}]
Two coverings $(\tilde X_{1},f_{1})$ and $(\tilde X_{2},f_{2})$ are isomorphic if
\begin{enumerate}
\item $\varphi: \tilde X_{1}\to \tilde X_{2}$ is a homeomorphism.
\item $f_{2}= f_{1}\circ \varphi$.
\end{enumerate}
\end{definition}

\begin{definition}[{\bf Fundamental group}]
Let $X$ be a space that admits a universal covering $(\tilde X,f)$. The group of homeomorphisms $\varphi$ of $\tilde X$ onto itself such that $f\circ \varphi = f$ is called the fundamental group of $\tilde X$.

Since two universal coverings are isomorphic, so are the fundamental groups. The corresponding abstract group is called the fundamental group of $X$.

\end{definition}
\begin{ex}\label{Ex:12}
Compute the fundamental groups of the following topological spaces:
\begin{itemize}
    \item Sphere $S^n$
    \item Torus $T^2$
    \item Real projective space $\bbR P^n$
    \item Hawaiian earring
\end{itemize}
\end{ex}
\begin{theorem}
\

\begin{enumerate}
\item A connected and locally connected space has a universal covering.
\item Two universal coverings are isomorphic.
\item Any covering of a universal covering is trivial. \end{enumerate}
\end{theorem}

\begin{definition}[{\bf Path-connected}]
\

\begin{enumerate}
\item A topological space $X$ is path-connected connected, if for any two points $a$ and $b$ in $X$, there exists a continuous path connecting those points; that is, there exists a continuous function $\gamma : [0,1] \to X$ such that $\gamma(0)=a$ and $\gamma(1)=b$.
\item $X$ is locally path-connected (or pathwise connected) if for each point $x\in X$ and each neighborhood $\cN(x)$ there exists a neighborhood $\cU(x)\subset \cN(x)$ that is path-connected.
\end{enumerate}
\end{definition}

\begin{theorem}
If a topological space is path-connected (locally path-connected) then it is connected (locally connected).
\end{theorem}

\, 

$\star$ Warning! The converse statement is false.

\, 

 \begin{example}
Consider the topologist's sine curve, which is constituted from the graph of the function  $y=sin(\frac{1}{x})$, where $x\in (0,1]$ and the vertical line segment $\{0\}\times [-1,1]$. 
The space is connected however it is not path-connected. 
Indeed, there is no continuous path connecting a point on the vertical segment to a point on the sine curve. In particular, any neighborhood around any point on $\{0\}\times [-1,1]$ contains infinitely many disconnected components of the sine curve. 
 \end{example}     
                       
   \begin{ex}\label{Ex:13}
    Provide an other example where the space is connected but not path-connected.  
\end{ex}

\section{Homotopy}
The notion of homotopy is a purely topological notion which allows to consider classes of topological objects up to some (homotopical) relation. 

To give an everyday life example, up to homotopy, a mug is equivalent to a doughnut (because they are both homotopy equivalent to a torus). 

\begin{definition}[{\bf Homotopic paths}]
Two paths $\gamma_{1}, \gamma_{2}:[0,1]\to X$ are homotopic if there exists a continuous map $F:[0,1]\times [0,1] \to X$ such that $F(t,0)=\gamma_{1}(t)$ and
$F(t,1)=\gamma_{2}(t)$. \end{definition}
\begin{example}
In the figure below, we draw a pair of homotopic planar curves. These curves are the level curves of the real and imaginary parts of a pair of complex polynomials.   
\begin{center}
\includegraphics[scale=0.5]{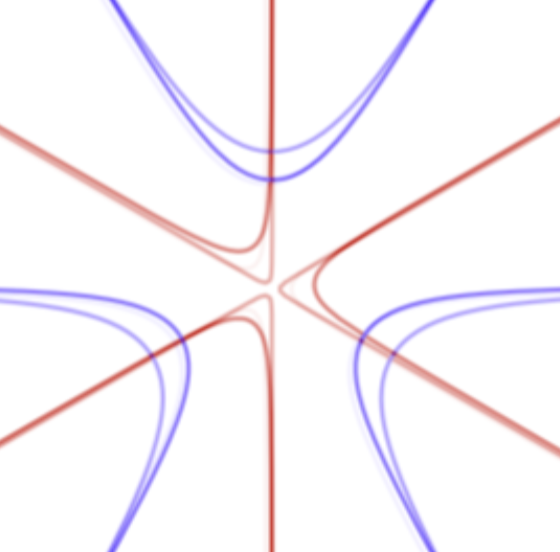}
\end{center}
Can you tell the degree of those curves?  Give (approximately) the equations of those polynomials.
\end{example}
\begin{definition}[{\bf Homotopic map}]

\

\begin{enumerate}
\item Two continuous maps $f,g:X\to Y$ are homotopic if there exists a continuous map $F:X\times [0,1] \to Y$ such that $F(x,0)=f(x)$ and
$F(x,1)=g(x)$ for all $x\in X$.

\item If $A$ is a subset of $X$, two continuous maps $f,g:X\to Y$ are homotopic with respect to $A$ if there exists a continuous map $F:X\times [0,1] \to Y$ such that $F(x,0)=f(x)$ and
$F(x,1)=g(x)$ for all $x\in X$ and $F(a,t)=f(a)$ for all $a\in A$ and $t\in [0,1]$.
\end{enumerate}
\end{definition}

\begin{theorem}
If $X$ is path-connected and locally path-connected, it is simply connected if every path $\gamma$ in $X$ is homotopic to the constant map.
\end{theorem} 

The notion of homotopy between two functions allows to define an equivalence relation between topological spaces.

\begin{definition}[{\bf Homotopically equivalent spaces}]
Two spaces $X$ and $Y$ are said to be homotopically equivalent (or `` of the same homotopy type '') if there exist two continuous maps $f:X\to Y$ and $g:Y\to X$ such that :
\begin{list}{$\bullet$}{}
\item $g\circ f $ is homotopic to $Id_{X}$, the identity on $X$ ;
\item $f\circ g $ is homotopic to $Id_{Y}$, the identity on $Y$ .
\end{list}
\end{definition}

\begin{definition}[{\bf Contractible space}]\label{D:contractible}
A space is called contractible if it is homotopically equivalent to a point.
\end{definition}
This definition is equivalent to saying that its identity map is homotopic to a constant map.

\begin{example}
The space $\bbR^{n}$ is contractible.
\end{example}

\begin{theorem}
Two homeomorphic topological spaces are homotopically equivalent
\end{theorem}

$\star$ Warning! The converse is false, as shown by the following examples:

\begin{list}{$\triangleright$}{}
\item Except  from the point itself, contractible spaces are not homeomorphic to the point.
\item A circle is homotopically equivalent to $\bbC^{\star}$, that is, a plane without a point, but it is not homeomorphic to it. A circle without two points is not connected, while a plane without three points is.
\item The real interval, the disk in the plane or the ball in $\bbR^{3}$, of radius 1 (open or closed) are all  contractible, therefore homotopically equivalent. However,  pairwise they are non-homeomorphic (by analogous arguments).
\end{list}

\section{Disjoint sum topology}

\subsection{Disjoint sum}
\begin{definition}[{\bf Disjoint sum}]\label{somdisj}

The disjoint sum of an indexed family of sets $\{E_i\}_{i\in I}$ is defined
by:
\begin{equation}
\coprod_{i\in I}E_{i}=\bigcup_{i\in I} \{(x,i) \mid\, i\in I, \,x\in E_i.\}
\end{equation}
\end{definition}

\noindent {\bf Note}: This definition of the disjoint sum allows to take into account the case where the $E_{i}\cap E_{j}\ne \emptyset$. Indeed, the elements of the disjoint sum are ordered pairs $ (x,i)$. Here $i\in I$ serves as an auxiliary index that indicates from which $E_i$ the element $x$ came from.  
Each of the $E_i$ sets  are canonically isomorphic to the set $\tilde E_i=\{(x,i): x\in E_i\}$. 
So, if we define $\varphi_j:E_j \to \coprod_{i\in I}E_i$ by $\varphi_j(x)=(x,j), $ where $j\in I$ then $ \varphi_{j}$ is a bijection from $E_j$ onto $\tilde E_j=\{(x,i)\mid x\in E_j\}$ and the $\tilde E_j$ are  disjoint, even if the $E_i$ are not.  
\[
\coprod_{i\in I}E_{i}=\bigcup_{i\in I} \{(x,i) \mid\, i\in I, \,x\in E_{i}\}=\bigsqcup_{i\in I} \{(x,i) \mid\, i\in I, \,x\in E_{i}\}.
\] 
Usually to state that the union is a disjoint union we replace $\bigcup$ by $ \bigsqcup$.

\subsection{Disjoint sum topology}

 \begin{definition}[{\bf Disjoint sum topology}] The disjoint sum topology on $\coprod_{i\in I}E_{i}$ is defined by 
 \[
 \cT= \left\{\cU\subset \coprod_{i\in I}E_{i} \mid \varphi_{i}^{-1}(\cU)\in \cT_{i} \ \forall \ i\in I. \right\}.
 \]
  \end{definition} 

\begin{proposition}[Fundamental properties of the disjoint sum topology] 

\,

\begin{enumerate}
 \item The set $ \left\{\cU\subset \coprod_{i\in I}E_{i} \mid \varphi_{i}^{-1}(\cU)\in \mathcal{T}_{i} \ \forall \ i\in I \right\}$ defines a topology.

\item[]

\item The maps $\varphi_{j}: X_{j}\to \coprod_{i\in I}E_{i}$ are continuous, open and closed; they are in fact homeomorphisms on
\[\varphi_{j}(X_{j})=\left\{(x,j)\in \coprod_{i\in I}E_{i}\mid x\in E_{j}\right\}.\]

\item[]
\item Let $Y$ be a topological space and $f: \coprod_{i\in I}E_{i}\to Y$ be a map. Then
\[ f\quad \text{is continuous }\quad  \Leftrightarrow f\circ \varphi_{j}:E_{j}\to Y \text{ is continuous }, \forall\, j\in I.\]

\item[]
\item The disjoint sum topology is the finest among those for which the $\varphi_{j},$ for all $\ j\in I$ are continuous.
\end{enumerate}
\end{proposition}

\begin{proof}
\begin{enumerate}
\item It is clear that $\emptyset$ and $\coprod_{i\in I}E_{i}$ belong to $\cT$.

\, 

Axioms (2) and (3) are a consequence of
the fact that the $\cT_{i}$ are topologies and of the following two set-theoretic properties of the operation $\varphi_{j}^{-1}$:

\vspace{5pt}
\begin{enumerate}[a)]
\item if we have a family $\{\cU_{\lambda}\}_{\lambda\in\Lambda}$ of subsets of $\coprod_{i\in I}E_{i} $, then:
\[
\varphi_{j}^{-1}\left(\bigcup_{\lambda\in\Lambda}\cU_{\lambda}\right)=\bigcup_{\lambda\in\Lambda} \varphi_{j}^{-1}(\cU_{\lambda}).
\]

 \vspace{3pt} 
\item if $\cU_{1},\dots,\cU_{N} \subset \coprod_{i\in I}E_{i}$, \[ \varphi_{j}^{-1}(\cU_{1} \cap \dots \cap \cU_{N})=\varphi_{j}^{-1}(\cU_{1} )\cap\dots \cap\varphi_{j}^{-1}(\cU_{N}).\] 
\end{enumerate}

 \vspace{5pt}
 \item By the above statement, the $\varphi_{i},$ where $ i \in I$ are continuous. If $\cU \subset E_{i}$ is open, let us see that $\varphi_{i}(\cU)$ is open. Indeed : 
 \[
 \varphi_{j}^{-1}(\varphi_{i})=\begin{cases} \cU& \text{ if } i=j\\ \emptyset & \text{ otherwise }\end{cases}
 \]
  and therefore it is an open set for the topology that we defined on $ \coprod_{i\in I}E_{i}$.

If $F\subset E_{j}$ is closed, then $\left(A=\coprod_{i\in I}E_i\right)/\varphi_j(F)$. This implies 
\[\varphi_i^{-1}(A)=\begin{cases} E_j/F & \text{ if } i=j\\ E_i & otherwise
\end{cases}\]
and therefore $A$ is open for the topology that we have defined on $\coprod_{i\in I}E_i$, which makes $\varphi_j(F)$ closed. It follows from previous arguments that the $\varphi_{j}$ form homeomorphisms on their image.

\vspace{5pt}
\item If $f: \coprod\limits_{i\in I}E_i\to Y$ is continuous for all $j \in I$, then the composition $f \circ \varphi_j$ is also continuous.

Suppose that $\varphi_j\circ f$ is continuous forall $j\in I$ and that $V \subset Y$ is an open set. Then $ \varphi_j^{-1}(f^{-1}(V))$ is open in
$E_j $, for all $ j \in I$, which means that $f^{-1}(V)$ is open in $ \coprod_{i\in I}E_i$. So, $f$ is continuous.

\vspace{5pt}
\item Let $\cT'$ be a topology on $\coprod_{i\in I}E_i$ for which the $\varphi_j$ are continuous for all$j \in I$. If $\cU\in \cT'$ we therefore have that $\varphi_j^{-1}(\cU')$ is open in $ E_j$, for all $j \in I$, which means that $U' \in \cT$. We thus have $\cT' \subset \cT$, that is to say that $\cT'$ is less fine than $\cT$.
\end{enumerate}
\end{proof}

\section{More Exercises}

\begin{enumerate}

\item Prove or disprove that every Hausdorff space $\mathbf{T_1}$ that is a space in which every set consisting of one point is closed.\\

\item Give an example of $\mathbf{T_1}$ space which is not Hausdorff space.\\

\item For each pair $f,g$ of continuous maps of a topological space $X$ in a Hausdorff space $Y$ the set 
$\{ X\in X: f(x)=g(x)\}$ is closed in space $X$
\end{enumerate}

\section{Hints/ short answers for exercises}

Ex. \ref{Ex:1} $\&$ \ref{Ex:5} The Cantor set is a metric space. 
The Cantor set is an uncountable set with Lebesgue measure zero. As the complement of a union of open sets, it is a closed subset of the real numbers and, consequently, a complete metric space. 
Furthermore, since it is totally bounded, the Heine–Borel theorem ensures that it is compact.
\, 

Ex.\ref{Ex:2}
Let $p$ be a given point of the space $X$. By hypothesis $X$ is a ${\bf T}_2$ space. 
Therefore, it follows that every point $x\neq p$ belongs to an open set $G_x$, that does not contain $p$ and thus that \[X\setminus\{p\}=\cup_{x\neq p}G_x.\] This implies that $X\setminus\{p\}$ is open, in other words $\{p\}$ is closed.

\, 

Ex. \ref{Ex:T_1neqT_2}  Consider the following topological space:
\[
X = \bigg\{ 0 \bigg\} \cup \bigg\{ \frac{1}{n} \;\bigg|\; n \in \mathbb{N}^* \bigg\}.
\]
We equip $X$ with a topology defined as follows:
\begin{itemize}
    \item The open sets that do not contain the point $\{1\}$ are precisely those inherited from the standard topology on $\mathbb{R}$. For example $\{\frac{1}{2},\frac{1}{3}\}$ is open.
    \item Any open set that contains the point $\{1\}$ must necessarily be of the form $U=X\setminus F$, where $F$ is a finite subset of $X$ not containing 1. For example, $U=X\setminus \{\frac{1}{2},\frac{1}{3}\}$.
\end{itemize}

Let us check that $X$ is a ${\bf T_1}$ Space. To show that $X$ is ${\bf T_1}$, we need to prove that every singleton set $\{x\}$ is closed. 
Consider:
\begin{itemize}
    \item For $x=0$: the space $X\setminus\{0\}$ is open. 
    \item For $x=\frac{1}{k}$: the space $X\setminus\{\frac{1}{k}\}$ is open because it can be written as $X\setminus F$, where $F=\{\frac{1}{k}\}$ is finite. 
\end{itemize}
Therefore, every singleton is closed, and $X$ is ${\bf T_1}$.

Let us check if $X$ is ${\bf T_2}$. 
\begin{itemize}
    \item Any open set $U$ containing 0 must be of the form $X\setminus F$, where $F$ does not contain 0. Since $X\setminus F$ is infinite, 
$U$ must contain infinitely many points of the form $\frac{1}{n}$.
\item Any open set $V$ containing 1 must also be of the form $X\setminus F$, where $F$ is finite and does not contain 1. This set $V$ will also contain infinitely many points of the form $\frac{1}{n}$. Thus $U$ and $V$ cannot be disjoint because they both contain infinitely many points $\frac{1}{n}$. Therefore, 0 and 1 cannot be separated by disjoint open sets and $X$ is not ${\bf T_2}$.
\end{itemize}

To summarize,
\begin{itemize}
    \item the space $X$ is ${\bf T_1}$, because every singleton is closed.
    \item The space $X$ is not ${\bf T_2}$,
 because the points 0 and 1 cannot be separated by disjoint open sets.
 \end{itemize}
 
\, 

Ex.\ref{Ex:cont} is is enough to show that the set $$A=\{x\in X: f(x)\neq g(x)\}$$ is open. 
For every $x\in A$ there exists in $Y$ two open sets $U$ and $V$ such that $f(x)\in U$ and $g(x)\in V$ and their intersection is empty. 
The set given by $f^{-1}(U)\cap g^{-1}(V)$ is the neighborhood of the point $x$ and contained in $A$. Therefore, $A$ is an open set. 

Ex.\ref{Ex:3} Let us consider the example of a non-Baire space: the set of rational numbers with the subspace Topology.

The set of rational numbers is countable. In particular, we can write 
\[\bbQ =\cup_{q\in \bbQ}\{q\}.\] Each singleton $\{q\}$ is closed in $\bbQ$ (subspace topology) but has empty interior: thus, $\{q\}$ is nowhere dense. Therefore, $\bbQ$ is a countable union of nowhere dense sets and $\bbQ$ cannot be a Baire space. 

\, 

Ex.\ref{Ex:4} The real numbers $\bbR$, equipped with the standard Euclidean topology, forms a Baire space. 

A space is Baire if the intersection of countably many dense open subsets is dense.
By the Baire Category Theorem, every complete metric space is a Baire space. Therefore, the conclusion follows.

\,

Ex.\ref{Ex:6} This is an application of the Arzela--Ascoli theorem.
\, 

Ex.\ref{Ex:7} No. This is due to the fact that $GL_n(\bbR)$ splits into two disjoint subsets: the set of matrices with strictly positive determinant and the set of matrices with strictly negative determinant. 

\, 

Ex.\ref{Ex:8} The curve $y=\sin\frac{1}{x}$ on the interval $(0,1]$ is connected but not locally connected nor arc-connected. 
As $x\to 0^+$, $\lim_{x\to 0^+}\frac{1}{x}$ grows without bound, causing $\sin\frac{1}{x}$ to oscillate infinitely between the values -1 and 1. By definition, a space is locally connected if every neighborhood of any point contains a connected open neighborhood.
In the case of $y=\sin\frac{1}{x}$, no matter how small a neighborhood $U$ of $(0,y)$ is chosen, the graph within 
$U$ consists of infinitely many disconnected components (due to the oscillations). Thus, there is no connected open neighborhood around $(0,y)$.

\,

Ex.\ref{Ex:9} The first quartic has eight non-compact connected components, while the second has only one. This can be shown by leveraging the fact that these quartics are invariant under the Coxeter group $CB_3$, which corresponds to the symmetries of the cube. By decomposing space into Coxeter chambers, the quartic can be analyzed within a single fundamental domain. Applying an appropriate change of variables, $x_i^2\mapsto X_i$ where $i\in \{1,2,3\}$ within this domain reduces the problem to studying degree-2 surfaces, defined in a cone contained in a positive octant (for instance the cone delimited by the equations 
$\{x_1=0\},\{x_2=0\},\{x_3=0\}$ and $\{x_i=x_j\}$ where $i\neq j$. 
Since the classification of quadrics is well understood, this approach simplifies the analysis. See \cite{C18} for an entire classification of such quartics. 

\, 

Ex.\ref{Ex:10} The universal covering space of the torus 
 is the Euclidean plane 
$\bbR^2$, with the covering map: $ \pi:\bbR^2\to T^2,$ 
$ \pi(x,y)=(\exp{2\pi\imath x},\exp{2\pi\imath y})$.

\, 

The Kähler torus (that is a torus equipped with a Kähler structure) admits a universal covering. In fact, the universal covering of a Kähler torus is the same as the universal covering of a standard torus, because the Kähler structure does not affect the underlying topology.
The map $\pi$ projects each point $z\in \bbC^n$ to its equivalence class 
$z+\Lambda$ in the quotient space $\bbC^n/\Lambda$.
This is a smooth, holomorphic map that respects the complex and Kähler structures.
The space $\bbC^n$ is simply connected because it is contractible.

\, 

Ex.\ref{Ex:11} The Hawaiian earring is a classic example of a space that is not locally simply connected. It consists of infinitely many circles of decreasing radius, all tangent to a single point. While each circle is itself simply connected, the entire space fails to be locally simply connected at the point of tangency because any neighborhood of that point contains infinitely many circles, making it impossible to find a simply connected neighborhood.

\, 

Ex.\ref{Ex:12} Fundamental groups. 
\begin{itemize}
    \item $\pi_1(S^1)\cong \bbZ$
    \item $\pi_1(S^n)\cong 0$, if $n\geq 2$
    \item $\pi_1(S^1\times S^1)\cong \bbZ\times \bbZ$
    \item $\pi_1(\bbR P^1)\cong \bbZ$
    \item $\pi_1(\bbR P^n)\cong \bbZ/2\bbZ$
    \item Hawaiian earring: uncountable. 
\end{itemize}

\, 

Ex.\ref{Ex:13} There are many examples. We can take the example of the deleted comb space. The comb space is a subspace of $\bbR^2$ which looks like a comb. A comb space is defined by the set: 
\[\{0\} \times [0,1] ) \cup (K \times [0,1]) \cup ([0,1] \times \{0\}\]
where $K=\{\frac{1}{n}\, |\, n\in\bbN^*\}$.
However, the aim of this exercise is to create your own example, using your imagination.

\chapter{Topological and Modelled manifolds}
\begin{center}
    \begin{minipage}{0.6\textwidth} 
  \centering
  \includegraphics[scale=0.5]{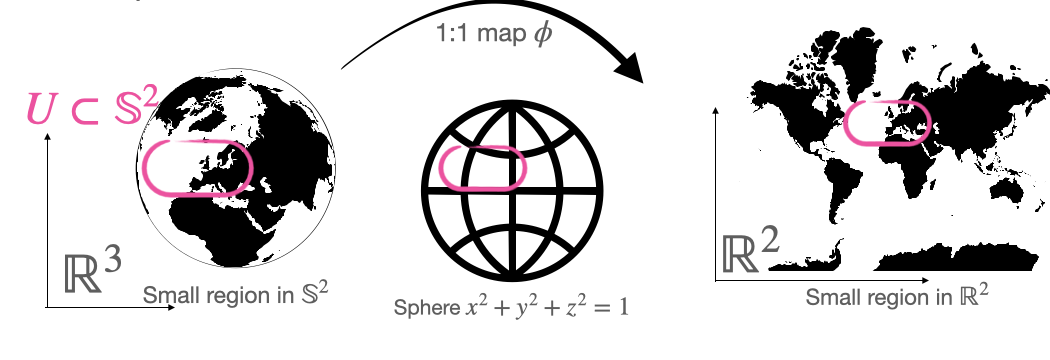}
\end{minipage}
\end{center}

A topological manifold is a fundamental object in mathematics, serving as an abstraction of spaces that locally resemble Euclidean space but  exhibit additional properties. Manifolds provide the natural setting for geometry and topology, forming the foundation for different disciplines in mathematics.

Formally, an $n$-dimensional topological manifold is a Hausdorff topological space that is locally homeomorphic to $\bbR^n$. This means that around every point, there exists a neighborhood that behaves like an open subset of the Euclidean space, allowing us to use an Euclidean space type of intuition, while still permitting the existence of global structures (namely curvature, holes, or any other topological feature).

\, 

Topological manifolds can be studied with some extra structures:

\begin{itemize}
    \item Differentiable manifolds, where smooth functions and derivatives can be defined, leading to differential geometry.
    \item Riemannian manifolds, which are equipped with a Riemannian metric.
    \item Lie groups: these are topological spaces that also possess a group structure, playing a crucial role in symmetry and mathematical physics.
\end{itemize}
The simplest examples of topological manifolds include familiar spaces such as the circle $S^1$, the sphere $S^n$ the torus $T^n$, as well as more sophisticated objects like projective spaces.
\section{Definitions}

\, 

\begin{definition}[Topological Manifolds]\label{D:topman}\index{Manifold!Topological Manifold}
Let $\cM$ be a set, and let $(\cE_i)_{i\in I}$ be a family of Hausdorff topological vector spaces. Consider a collection $\cA$ of pairs $\{(\cU_i, \varphi_i)\}_{i \in I}$, indexed by some set $I$, where each $\cU_i$ is a subset of $\cM$ and each $\varphi_i$ is a mapping associated to it, subject to the following axioms:

\begin{itemize}
    \item[$\mathbf{(\cM1)}$] The sets $\cU_i$ form an open cover of $\cM$:
    \[
    \cM = \bigcup_{i \in I} \cU_i, \quad \cU_i \subset \cM.
    \]
    
    \item[$\mathbf{(\cM2)}$] Each $\varphi_i$ is a bijection between $\cU_i$ and an open subset of $\cE_i$:
    \[
    \varphi_i: \cU_i \to \varphi_i(\cU_i) \subset \cE_i.
    \]

    \item[$\mathbf{(\cM3)}$] For each pair of indices $i, j \in I$, the transition maps
    \[
    \varphi_j \circ \varphi_i^{-1}: \varphi_i(\cU_i \cap \cU_j) \to \varphi_j(\cU_i \cap \cU_j)
    \]
    define homeomorphisms between the respective open subsets of $\cE_i$ and $\cE_j$.

    \item[$\mathbf{(\cM4)}$] Unless otherwise specified, we assume that each $\cE_i$ is a locally convex Hausdorff topological vector space.
\end{itemize}

The data $(\cM, \cA)$ thus defines a \textit{topological manifold} modeled on the spaces $\cE_i$.
\end{definition}




\begin{figure}[h]
\includegraphics[scale=0.45]{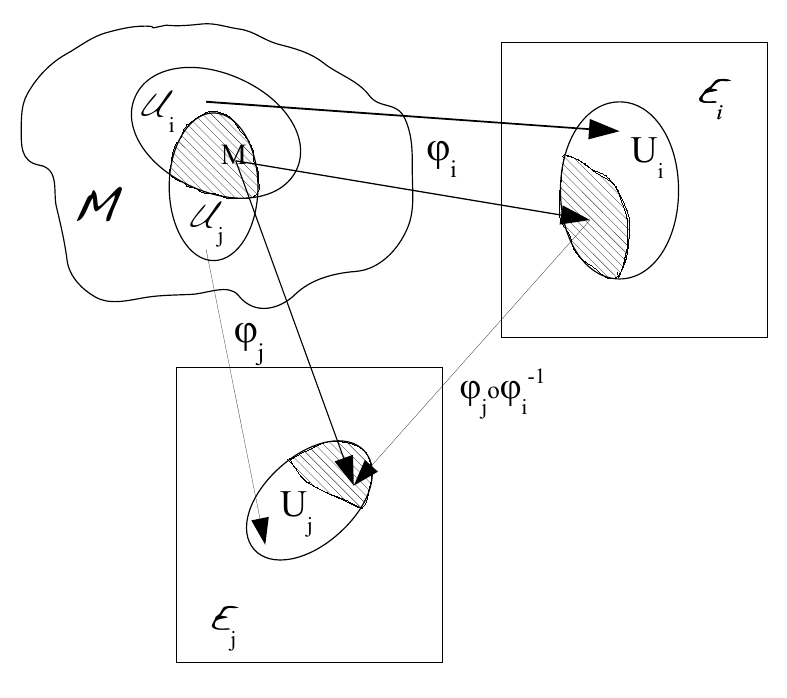} 
\caption{Charts and atlas.}\label{F:chart1}
\end{figure}

\section{Charts, Atlas}

\begin{definition}[Chart]\index{Manifold!Chart}
Let $\cM$ be a topological manifold. Each pair $(\mathcal{U}_i, \varphi_i) \in \mathcal{A}$, where $\mathcal{U}_i$ is an open subset of $\cM$ and $\varphi_i$ is a mapping of $\mathcal{U}_i$ onto an open subset of a topological vector space, is called a \textit{chart} (or \textit{coordinate system}) of $\cM$.

If a point \textsc{m} $\in \mathcal{M}$ lies in $\mathcal{U}_i$, we say that $(\mathcal{U}_i, \varphi_i)$ is a chart at \textsc{m}.
\end{definition}

\begin{definition}[Atlas]\index{Manifold!Atlas}
An \textit{atlas} on a manifold $\cM$ is a collection of charts $\{(\mathcal{U}_i, \varphi_i)\}_{i \in I}$ such that:
\begin{itemize}
    \item[$\mathbf{(A1)}$] The sets $\mathcal{U}_i$ cover $\cM$:
    \[
    \cM = \bigcup_{i \in I} \mathcal{U}_i.
    \]
    \item[$\mathbf{(A2)}$] The transition maps
    \[
    \varphi_j \circ \varphi_i^{-1}: \varphi_i(\mathcal{U}_i \cap \mathcal{U}_j) \to \varphi_j(\mathcal{U}_i \cap \mathcal{U}_j)
    \]
    are homeomorphisms for all $i, j$ for which $\mathcal{U}_i \cap \mathcal{U}_j \neq \emptyset$.
\end{itemize}
\end{definition}

It is natural to assume that the topology on $\cM$ is given a priori. In many cases, one requires that each chart $\varphi_i$ be a homeomorphism onto its image. As a first consequence of this definition, we obtain the following structural result:
\begin{proposition} One can give to $\cM$ a topology in a unique way such that each
$\cU_i$ is open, and the $\varphi_i$ are topological isomorphisms.
\end{proposition}







\begin{remark}
The condition that $\cM$ is  Hausdorff,  is not necessary.
This condition plays no role in the formal development of manifold.  

However, in practical applications,  $\cM$ is Hausdorff.
\end{remark}
\begin{ex}\label{Ex:smfd}
    Prove that the Special Linear group $SL_n(\R)$ is a smooth submanifold of $\bbR^{n\times n}$.
\end{ex}
\begin{ex}\label{Ex:sing}
Is the set $\{{x,y}\in \bbR^2\, |\, xy=0\}$ a submanifold of $\bbR^2$ ?
\end{ex}

\section{Modeled Manifolds}
\subsection*{Hausdorff Condition and Separation Axioms}
$\star$ In what follows, we impose the Hausdorff condition on all manifolds under consideration. Furthermore, any construction performed subsequently—such as products, tangent bundles, or fibered structures—will be required to yield spaces that remain Hausdorff. This ensures a well-behaved topological framework in which geometric operations preserve separation properties.

\subsection{\(\mathcal{E}\)-modeled manifolds}

In the formulation given above, particularly in condition $\cM2$, no global assumption was imposed on the structure of the topological vector spaces $\mathcal{E}_i$ used as local models. In particular, we did not require that all $\mathcal{E}_i$ be the same, nor that there exist continuous linear isomorphisms between them.

However, one of the most natural and frequently encountered cases in practice is that in which there exists a fixed topological vector space $\mathcal{E}$ such that each $\mathcal{E}_i$ is isomorphic to $\mathcal{E}$, allowing a uniform model for the local structure of $\mathcal{M}$. In such a setting, one can regard $\mathcal{M}$ as being locally modeled on a single space $\mathcal{E}$, which provides a more rigid but also more manageable framework for various constructions.





\vspace{3pt}
This remark leads us to the notion of modeled manifold:
\begin{figure}[h]
\includegraphics[scale=0.5]{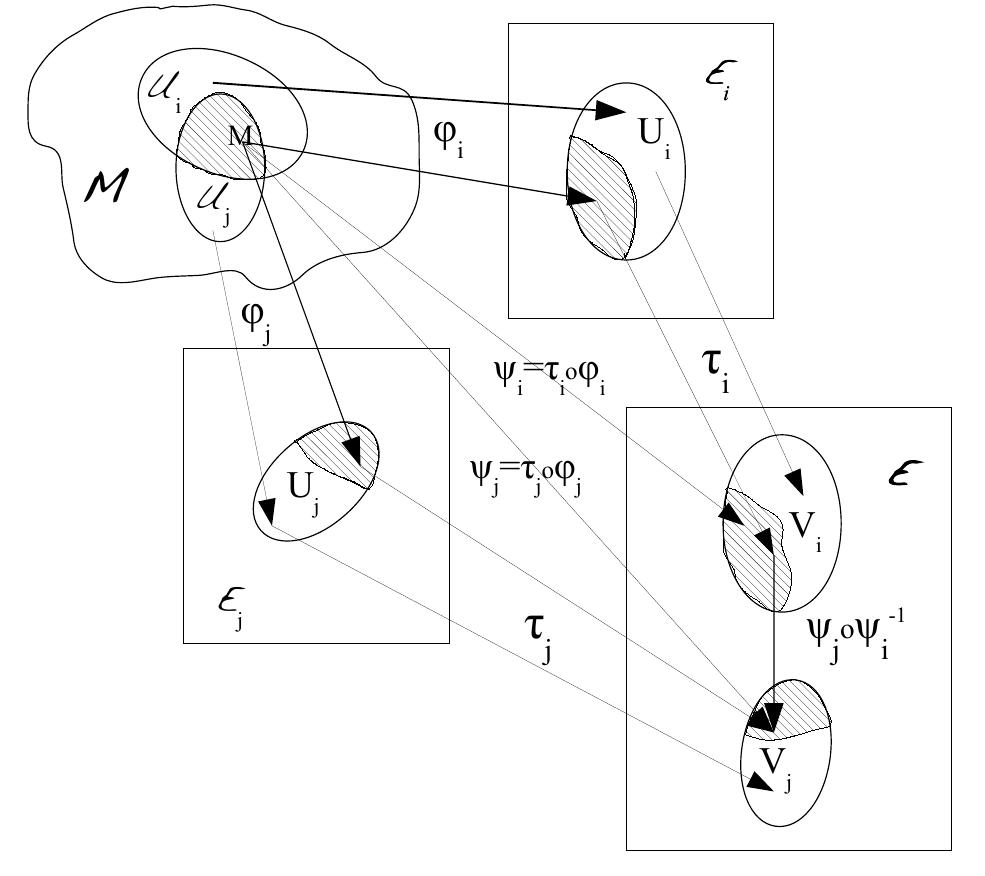}
\caption{Modeled manifold }\label{F:Chart2}
\end{figure}

\begin{definition}[Modeled manifold]\index{Manifold!Modeled manifold}

\noindent
Let \(\mathcal{M}\) be a topological space. An \(\mathcal{E}\)-\textit{atlas} on \(\mathcal{M}\) consists of a collection of charts \(\{(U_i, \varphi_i)\}_{i \in I}\), where:

\begin{enumerate}
    \item \(\{U_i\}_{i \in I}\) is an open cover of \(\mathcal{M}\),
    \item Each \(\varphi_i: U_i \to \mathcal{E}_i\) is a homeomorphism onto its image, where \(\mathcal{E}_i\) is a topological space,
    \item All the spaces \(\mathcal{E}_i\) are homeomorphic to a fixed topological space \(\mathcal{E}\), i.e., there exist homeomorphisms \(h_i: \mathcal{E}_i \to \mathcal{E}\).
\end{enumerate}

If an \(\mathcal{E}\)-atlas is given, we say that \(\mathcal{M}\) is an \(\mathcal{E}\)-\textit{modeled manifold} (or simply an \(\mathcal{E}\)-\textit{manifold}).

\end{definition}

\noindent

If the transition maps (i.e. maps of the type \(\varphi_j \circ \varphi_i^{-1}\)) satisfy additional compatibility conditions (such as differentiability, smoothness, or analyticity), then  the \(\mathcal{E}\)-manifold is endowed with some extra structure.

\vspace{5pt}

In the following, we primarily focus on modeled manifolds whose model space is a vector space that admits a well-defined differentiable structure. This is the case for manifolds modeled on Banach spaces or Hilbert spaces, where a differentiable structure can be naturally defined. However, we are also interested in a generalization of normed spaces, namely locally convex topological vector spaces. 

\, 
\begin{example}
A Hilbert manifold is a manifold modeled on a Hilbert space, which is a complete inner product space. 
Hilbert manifolds generalize finite-dimensional smooth manifolds to infinite dimensions.

\begin{itemize}
    \item {\bf Model Space:} Assume $H$ is a separable Hilbert space. One can take for instance the space of square integrable functions $L^2(\R)$.
    \item {\bf Charts:} A Hilbert manifold $M$ is a topological space equipped with an atlas of charts where $U_a\subseteq M$ is an open set and $\phi_a:U_a\to H$ is a homeomorphism onto an open subset of $H$.
    \item {\bf Transition Maps:} Given a pair of charts $(U_a,\phi_a)$ and $(U_b,\phi_b)$, the transition map 
    \[\phi_b\circ \phi_a^{-1}:\phi_a(U_a\cap U_b)\to \phi_b(U_a\cap U_b) \]
    is a smooth map between open subsets of $H$.
    \item {\bf Local Structure:} At every point $p\in M$, the tangent space $T_pM$ is isomorphic to $H$, the Hilbert space. 
\end{itemize}
\end{example}

\,

\begin{ex}\label{Ex:LoopSpace}
Consider a smooth finite-dimensional manifold $M$. The loop space $LM$ is the space of all smooth maps:

\[\gamma:S^1\to M,\]
where $S^1$ is the unit circle. Is the loop space $LM$ a modeled manifold? 
\end{ex}

\begin{ex}\label{Ex:BanSpace}
Let $M$ and $N$ be finite dimensional smooth manifolds. Let $C^k(M,N)$ be the space of $C^k$-differentiable maps from $M$ to $N$. Prove that  this space can be given the structure of a Banach manifold modeled on a Banach space.
\end{ex}


\subsection{Recollections on Locally Convex  Spaces}

\, 

Let us recall that a {\bf locally convex  space} $X$ is a vector space over $\bbK$, a (sub)field of the complex numbers (it can be $\bbC$ itself or $\bbR$ for instance).

A locally convex space is defined either in terms of convex sets or equivalently in terms of seminorms.
In fact, a topological vector space $X$ is said to be locally convex if it verifies one of the following two equivalent properties. We start with the convex set definition.

\begin{definition}[{\bf Convex sets definition}]
A topological vector space $X$ is said to be locally convex if there exists  a neighborhood basis (that is, a local base) at the origin, consisting of balanced convex sets. 

We elaborate on those two notions of convexity and of balanced sets. 

\begin{enumerate}
\item A subset $C\in  X$ is called \emph{convex} if for all $ x,y\in C$, and $0\leq t\leq 1$ we have 
\[tx+(1-t)y\in C.\] 
\item A convex subset $C\in  X$ is called 
\emph{balanced} if for all $x\in C$ and scalars $s\in \bbK$, the condition $|s|\leq 1$ implies that $ sx\in C$.
\item A convex subset $C\in  X$ is called a \emph{cone} (we assume here that the underlying field is ordered) if for all  $x\in C$ and 
$ t\geq 0$ we have $tx\in C$.
\item A convex subset $C\in  X$ is called  \emph{absorbent}  if for every $ x\in X$ there exists $r>0$ such that 
\[ x\in t\cdot C\] for all $t\in \bbK$, and  $|t|>r$. The set 
$C$ can be scaled out by any "large" value to absorb every point in the space.
\end{enumerate} 
\end{definition}
In any topological vector space, every neighborhood of the origin is absorbent.

\, 

A second possible viewpoint on this notion can be achieved using seminorms. A \emph{seminorm} on $X$ is a function \[p: X \to \bbR\] such that:

\begin{enumerate}
\item {\bf Non-negativity:} $p$ is nonnegative or positive semidefinite i.e. \[p(x)\geq 0,\quad \text{for all}\quad x\in X.\]

\item {\bf Scaling property}: $p$ is positive homogeneous or positive scalable: \[p(sx)=|s|\cdot p(x),\quad \text{for every scalar} \quad s \, \text{and}\, x\in X.\] So, in particular, $p(0)=0$,

\item {\bf Subadditivity}: $ p$ is subadditive and it satisfies the triangle inequality: \[p(x+y)\leq p(x)+p(y), \quad \text{for all}\quad x,y\in X.\]
\end{enumerate}

\begin{definition}{\bf (Seminorms definition)} A topological vector space $X$ over a field $\bbC$ or $\bbR$ is said to be locally convex if there exists a family $\cP$ of \emph{seminorms} on $X$.
Let $\{p_i\}_{i\in I}$ be a family of semi-norms on 
$X$, where $I$ is an index set. Each semi-norm $p_i:X\to \bbR$ satisfies the properties of a semi-norm (non-negativity, absolute homogeneity, and subadditivity).
\end{definition}
\,

\begin{remark}
    Although the definition in terms of a neighborhood base gives a better geometric picture and intuition the definition in terms of seminorms is easier to work with, in practice.

\end{remark}
    
\begin{ex}\label{Ex:seminorm}
Consider the space of continuous functions $C(\bbR)$ on $\bbR$. Verify whether the family of functions $\{p_n\}_{n\in \bbN}$,
given by: \[p_n(f)=\sup_{x\in [-n,n]}|f(x)|\] forms a family of seminorms.  
Is $C(\bbR)$ locally convex? Do the $\{p_n(f)\}$ generate a locally convex topology?
\end{ex}
\subsection{Bridging those two definitions}

\,
 As mentioned earlier, both definitions are equivalent.  We outline a sketch of the proof showing this equivalence. 

\, 

$\star$ The equivalence of those two definitions follows from a construction known as the \emph{Minkowski functional} or Minkowski gauge. The key feature of seminorms which ensures the convexity of their $\varepsilon$-balls is the {\it triangle inequality}.

\vspace{3pt}

For an absorbing set $C$ such that: if $x\in C$ then $ t\cdot x\in C$, whenever $ 0\leq t\leq 1$, let us define the Minkowski functional of $C$ to be \[\mu _{C}(x)=\inf\{r>0 : x \in r\, C\}.\]

\vspace{3pt}

From this definition, it follows that $\mu _{C}$ is a seminorm if $C$ is balanced and convex (it is also absorbent by assumption). 

\, 

Conversely, given a family of seminorms, the sets
\[\left\{x:p_{\alpha _{1}}(x)<\varepsilon _{1},\ldots ,p_{\alpha _{n}}(x)<\varepsilon _{n}\right\}\] form a base of convex absorbent balanced sets.

\vspace{5pt}
Let us mention some interesting properties.
\begin{itemize}
\item If $p$ is positive definite (which states that if $p(x)=0$ then 
 $x=0$) it implies that $p$ is a norm. 

\,

$\star$ While in general seminorms {\it need not} be norms, there is an analogue of this criterion for families of seminorms and separatedness, defined below.\\

\item If $X$ is a vector space and $\cP$ is a family of seminorms on $X$ then a subset $\mathcal{Q} \subset \cP$ is called a base of seminorms for $\cP$ if for all $p\in \cP$ there exists a $q\in \mathcal{Q} $ and a real number $r>0$ such that $p\leq rq$.\\

\item If $X$ is a real vector space and $C$ a convex subspace, if $x\in C$, then $C$ contains the line segment between $x$ and $-x$.\\
 
\item If $X$ is a complex vector space and $C$ a convex subspace,  for any $ x\in C$ convex space, contains the disk with $x$ on its boundary, centered at the origin, in the one-dimensional complex subspace generated by $x$.
\end{itemize}

\noindent
\subsection{Relations to the Banach Space Implicit Function Theorem}
A \textit{locally convex topological vector space} is a topological vector space in which every neighborhood of \(0\) contains an open neighborhood \(U\) of \(0\) such that, for all \(x, y \in U\) and \(0 \leq t \leq 1\),
\[
t\cdot x + (1-t)\cdot y \in U.
\]
This property turns out to be essential in the context of the implicit function theorem, for Banach spaces.

\vspace{0.5cm}

\noindent
The latter result, known as the \textbf{Banach Space Implicit Function Theorem}, can be stated as follows:

\, 

Let \(X\), \(Y\) and \(Z\) be Banach spaces and let \(\mathcal{U}\) be an open subset of \(X \times Y\). Suppose that \(f : \mathcal{U} \to Z\) is a continuously differentiable (\(C^1\)) mapping such that \(f(a, b) = 0\) for some \((a, b) \in \mathcal{U}\) and the partial derivative \(D_y f(a, b)\) is a linear isomorphism from \(Y\) onto \(Z\).

Then, there exists:

\begin{itemize}
    \item an open set \(\mathcal{W} \subset X\) with \(a \in \mathcal{W}\);
    \item an open set \(\mathcal{V} \subset \mathcal{U} \subset X \times Y\) with \((a, b) \in \mathcal{V}\);
    \item a continuously differentiable (\(C^1\)) mapping \(g : \mathcal{W} \to Y\),
\end{itemize}

such that
\[
(x, y) \in \mathcal{V},\, f(x, y) = 0 \iff x \in \mathcal{W},\, y = g(x).
\]





In other words, the implicit equation $f(x,y)=0$ has for $x\in W$ a solution $y=g(x)$ of class $C^1$
such that $(x,y)\in \cV$. 

This solution is unique in an open set $\cW'\subset \cW$.

\medskip

$\star{263C}$ Warning! It is important to note that \emph{not} every locally convex topological vector space admits a differentiable structure in a meaningful way. Nevertheless, a particular class of locally convex topological vector spaces, known as \emph{Fréchet spaces}, plays a fundamental role in various areas of mathematics, such as statistics and the theory of partial differential equations.

\, 

A \emph{Fréchet space} $X$ is defined as a locally convex, metrizable, and complete topological vector space, meaning that every Cauchy sequence in $X$ converges to a point in $X$. 

 For general normed vector space $X,Y$,  the Fr\'echet directional derivative  of a function  $ f:\cU\to Y$, where $ \cU$ is an open subset of $X$, exists at $ x\in \cU$  if there exists a bounded linear operator $A:X\to Y$ such that
\[
 \lim _{\|h\|\to 0}{\frac {\|f(x+h)-f(x)-Ah\|_{Y}}{\|h\|_{X}}}=0.
 \]
The Fréchet derivative in finite-dimensional spaces is the usual derivative.  it is represented in coordinates by the Jacobian matrix.

\section{Exercises}

Ex. \ref{Ex:LoopSpace} The loop space $LM$ is modeled on the Hilbert space $H=L^2(S^1,\bbR^n)$, where $n=dim(M)$. 
Each loop $\gamma$ can be locally approximated by functions in $L^2(S^1,\bbR^n)$.
For a fixed loop $\gamma\in LM$, a chart around $\gamma$ can be constructed using the exponential map on $M$. For a small neighborhood $U$ of $\gamma$, the chart maps $U$ to an open subset of $H$.

\, 

Ex. \ref{Ex:BanSpace} This space can be given the structure of a Banach manifold, modeled on a Banach space of sections of a vector bundle. The model space is the Banach space $C^k(M,\bbR^n)$ where $n=dim(N)$. For a fixed map $f\in C^k(M,N)$, a chart around $f$ can be constructed using the exponential map on $N$. Specifically, for a small neighborhood $U$ of $f$, the chart maps $U$ to an open subset of $C^k(M,\bbR^n)$.

\, 

Ex. \ref{Ex:seminorm}. Yes, each $p_n$ forms a seminorm. The family $\{p_n\}_{n\in \bbN}$ generates a locally convex topology on $C(\bbR)$. In this topology, a sequence of functions $f_k$ converges to $f$ if and only if $p_n(f_k-f)\to 0$, for every $n\in \bbN$.  

\, 

Ex. \ref{Ex:smfd}.
This exercise is slightly more advanced since this statement can be shown using the Regular Value Theorem, which we have not considered in this book. We will use this exercise as a possibility to mention and illustrate this theorem! 

\, 
\begin{enumerate}
    \item \begin{itemize}
    \item A first step, is to identify the space of all $n\times n$ matrices to the Euclidean space $\bbR^{n^2}$. 
\item The determinant function $\det$ is a smooth (infinitely differentiable) function since it is a polynomial in the entries of the matrix.
\item By definition,  the special linear group is a set of matrices of size $n\times n$ where the determinant is equal to 1. In other words, this is the preimage of the value 1, under the determinant function.
\end{itemize}
\item The Regular Value Theorem states that if \[f:M\to N\] is a smooth map between two smooth manifolds and $y\in N$ is a regular value of $f$, then $f^{-1}(y)$ is a smooth submanifolds of $M$. In our case, we have $M=Mat_n(\bbR)=\bbR^{n^2}$, $N=\bbR$ and $f=\det$. One needs to prove t hat 1 is a regular value of $\det$. This can be done by showing that for every matrix $A\in Sl_n(\bbR)$ the derivative $d(\det)_A$ is surjective. 
\item We now compute the derivative of the determinant. The derivative of the determinant function at a matrix $A\in Mat_n(\bbR)$ is given by the following formula: 
\[d(\det)_A:Mat_n(\bbR)\to \bbR,\quad d(\det)_A(H)=tr(adj(A)\cdot H),\]
where $adj$ stands for the classical adjoint of the square matrix $A$ (it is defined as the transpose of the cofactor matrix of $A$) and $tr $is the trace operator. If $A\in Sl_n(\bbR)$ then the formula becomes
\[d(\det)_A(H)=tr(A^{-1}\cdot H),\] since $adj(A)=A^{-1}$ if $A\in Sl_n(\bbR)$.

\item Finally, the last step is to show the surjectivity of the derivative.  In other we need to demonstrate that for any real number $c\in \bbR$, there exists a matrix $H\in Mat_n(\bbR)$ such that:
\[tr(A^{-1}H)=c.\]
An easy example of such matrices is given for $H=cA$. 
Indeed, \[tr(A^{-1}H)=tr(A^{-1}(cA))=c\cdot tr(I_n)=c\cdot n.\]
Therefore, $d(\det)_A$ is surjective for all $A\in Sl_n(\bbR)$.
\item Applying the theorem leads to the conclusion. 
\end{enumerate}
\, 

Ex. \ref{Ex:sing}. The equation $f(x,y)=xy=0$ in the Euclidean plane $\bbR^2$ describes a union of the two coordinate axes:
the $x$-axis (corresponding to $y=0$) and $y$-axis corresponding to ($x=0$). This set is not a submanifold of $\bbR^2$ in the usual sense of differential geometry. 
\begin{enumerate}
    \item Observe that there exists a \emph{singular point} at $(0,0)$. The set 
    \[S=\{x,y\in \bbR^2\, |\, xy=0\}\] consists of two connected components (the $x$-axis and the $y$-axis), which intersect precisely at the origin. Intuitively, a smooth submanifold of $\bbR^2$ must locally resemble a smooth curve (or a single point). Away from the origin, the set $S$ consists of a pair of smooth lines, but the origin, the submanifold condition fails. 
    \item To verify this, we compute the partial derivatives of the function $f(x,y)=xy$:
    \[\frac{\partial f(x,y)}{\partial x}=y\quad \text{and} \quad \frac{\partial f(x,y)}{\partial y}=x.\]
    At $(0,0)$, both derivatives  vanish:
    \[\frac{\partial f(x,y)}{\partial x}(0,0)=0 \quad \frac{\partial f(x,y)}{\partial y}(0,0)=0.\]
    This confirms that (0,0) is indeed a singular point.  
\end{enumerate}
This illustrates an example of a singular algebraic variety, where the smooth parts (the axes away from the origin) are manifolds, but the overall structure is not a manifold due to the singularity at the origin.

\chapter{Differentiability and Gateaux Derivatives}\index{Manifold!Differentiable}

In this chapter, we undertake a comprehensive investigation of the notion of differentiability, within an extended framework. This expands beyond elementary calculus to encompass both differentiable functions and the rich structure of differentiable manifolds. Building upon this foundation, we systematically develop the essential constructions of tangent vectors and tangent spaces, alongside their dual counterparts, cotangent spaces. These concepts not only underpin the analytical machinery of differential geometry but also enable far-reaching applications across mathematics—from topology and dynamical systems—and physics, particularly in the geometric formulation of classical mechanics, general relativity, and gauge field theories.

More precisely, the notion of differentiability serves in mathematics in:  
\begin{enumerate}
    \item Local Linear Approximations
    \begin{itemize}
        \item Differentiability allows functions, curves, or maps to be approximated locally by linear objects (e.g., tangent lines, Jacobian matrices). This is used for instance in optimization, where one considers extrema via critical points. This also used in machine learning. 
    \end{itemize}
    \item Smoothness (and other important properties) of manifolds:

    \begin{itemize}
        \item Differentiability defines the notion of "smoothness" of a manifold but also permits to study  vector fields, differential forms and curvature tensors. 
    \end{itemize}
    \item Global/Local analysis: 
    \begin{itemize}
        \item Differentiability connects local properties (e.g., derivatives) to global phenomena for instance using De Rham Cohomology.
        \end{itemize}
    \item Algebraic structures
    \begin{itemize}
        \item
    Differentiability is the gateway to advanced (and more algebraic) frameworks such as 
Jet Bundles; Lie Groups (studying symmetries of differentiable manifolds); Sheaf Theory (formalizing local-to-global properties in algebraic geometry).
\end{itemize}
\end{enumerate}

\section{Gateaux directional derivation: a definition} Inscribing ourselves in the vein of the previous chapter, the notion of Gateaux's directional derivation appears the most natural to start with. We introduce this notion below. 

\,

Let $\cM$ be a manifold  modeled on a topological vector space $\cE$, and let us assume  that a differentiable structure can be defined on $\cE$. As we have shown $\cE$ to be a Hausdorff (locally convex vector space) one can define a differentiable structure via the \emph{Gateaux directional derivation}.

\, 

\subsubsection{Gateaux derivative}
Suppose that $X$ and $Y$ are locally convex topological vector spaces. Assume $U \subset X$  is open and that we have the map $F: X \to Y$. The Gateaux differential $dF(x;\varphi )$ of $F$ at $x \in U$ in the direction $\varphi$ in $X$ is defined as
\begin{equation}\label{E:Gateaux}
 dF ( x ; \varphi ) = \lim_{t \to 0} \frac{F(x +t\, \varphi) - F ( x )}{ t} = \frac{d}{d t} F( x +t\,\varphi )\vert_{t=0}
 \end{equation}
If this limit exists for all $\varphi$, then $ F$ is said to be \emph{Gateaux differentiable} in $x$.

\, 

The limit  is taken relatively to the topology of $Y$.

Indeed:
\begin{enumerate}
\item if $X$ and $Y$ are {\it real} topological vector spaces, then the limit is taken for {\it real} $t$. 
\item If $X$ and $Y$ are complex topological vector spaces, then the limit above in \eqref{E:Gateaux} is usually taken as $t \to 0 $ in the complex plane, as usually in the definition of complex differentiability.
\item In some cases, a \emph{weak limit} is taken instead of a strong limit, which leads to the notion of a \emph{weak Gateaux derivative.}
\end{enumerate}

\, 

This short introduction to Gateaux's derivative guides us towards the more standard terminology of  differentiable manifolds. 

\subsection{Differentiable manifolds}
We start with a definition on differentiable manifolds.

\begin{definition}[Differentiable manifold]\index{Manifold!Differentiable}
A \(\mathcal{C}^k\)-\textit{differentiable manifold} \(\mathcal{M}\) is a topological manifold where the condition  \((\mathcal{M}3)\) is substituted by a new condition:

\begin{itemize}
    \item[\((\mathcal{M}3')\)] for each pair of indices \(i, j \in I\), the transition map between overlapping coordinate charts,
    \[
    \varphi_j \circ \varphi_i^{-1} : \varphi_i(U_i \cap U_j) \to \varphi_j(U_i \cap U_j),
    \]
    is of class \(\mathcal{C}^{k}\), meaning it is \(k\)-times continuously differentiable. Moreover, for every \(i, j \in I\), the set \(\varphi_i(U_i \cap U_j)\) is open in the model space \(\mathcal{E}\).
\end{itemize}


\end{definition}

Moreover, we have the following compatibility criterion. 
\begin{definition}[Compatible atlas]\index{Manifold!Atlas!Compatible}
Two $\cC^{k}$ atlases on $\cM$, modeled on  $\cE$ are said to be compatible if their union is an another such atlas.
\end{definition}
Remark that this notion of compatibility is in fact an equivalence relation. 

\begin{ex}
Prove the remark above.
\end{ex}

\begin{definition}[Admissible atlas]\index{Manifold!Atlas!Admisible}
Given a manifold $\cM$,
all atlases, modeled on $\cE$, and lying within the same equivalence class are said  to be \emph{admissible} on $\cM$.
\end{definition} 

It is enough to have an admissible atlas to define the structure of a manifold.

\,
\subsection{Differentiable mappings}
\, 

\begin{definition}[Differentiable mappings]~\label{DiffMap}\index{Mapping!Differentiable}
Consider a pair of $\cC^k$-differentiable manifolds, denoted  $\cM$ and $\cM'$.
\begin{itemize}
\item A mapping $f: \cM \to
\cM'$ is said to be differentiable of class $\cC^r,$ where $\, r\leq k$, if for every chart $(\cU_i,\varphi_i)$ of $\cM$ and every chart $(\cV_j,\psi_j)$ of  $\cM'$ such that $f(\cU_i) \subset \cV_j,$ the mapping $\psi_j
\circ f \circ \varphi_i^{-1}$ 
 of $\varphi_i(\cU_i)$ into $\psi_j(\cV_j)$ is
differentiable, of class $\cC^r$.

\item[ ]
\item A \emph{diffeomorphism} \index{Mapping!Diffeomorphism} $f$ of class $\cC^r$ is a bijective mapping $f$ such that $f$ and $f^{-1}$ are continuously $\cC^r$-differentiable.
\item[ ]

\item If the transition maps $\varphi_j\circ\varphi_i^{-1}$ are diffeomorphisms of class $\cC^\infty$ then the manifold $\cM$ is said to be smooth.

\item[ ]
\item A $\cC^k$-differentiable function $f:\cM \to \bbR$ on $\cM$ is a mapping of class $\cC^k$ from 
$\cM$ to $\bbR$. 
\item[ ]
\item A function $f$ is \emph{differentiable} at a point $ m$ on the manifold $\cM$ if, for some coordinate chart $(U,\phi)$ containing $ m$, the composition $f\circ \varphi^{-1}$ is differentiable at the image point $\varphi(m)$. 
\end{itemize}
\end{definition}

\noindent

Notice, that these definitions do not depend on the chart.

\,
\begin{remark}
Since the implicit function theorem does not hold for arbitrary locally convex space, this definition has a limited utility if we do not introduce   more information on the topology.
\end{remark}

\section{Manifolds modeled on a normed vector spaces}

Interesting properties are obtained
in the case where the model space $\cE=\fB$ is a Banach vector space a natural  differentiability structure is provided by the Banach derivation~\eqref{E:Bdif}. In this case the implicit function theorem allows to derive interesting properties.

\begin{definition}[Banach manifold]\index{Manifold!Banach}
A $\cC^k$-Banach manifold $\cB$ is a manifold such that condition $\cM 3$ is  replaced by: 

\vspace{3pt}
$\bullet$ $\cM 3''.$ The map
\[
\varphi_j\circ\varphi_i^{-1}:\varphi_i(\mathcal{U}_i\cap \mathcal{U}_j) \to \varphi_j(\mathcal{U}_i\cap
\mathcal{U}_j),
\]
is a $\cC^{k}$-isomorphism on $\fB$ for each pair of indices $i,j$, and for any
$i,j\in I$
$\varphi_i(\cU_i\cap \cU_j)$ is open in $\cB$.
\end {definition}

\begin{definition}[Riemann manifold]\index{Manifold!Riemann}
A Riemannian manifold is a differentiable manifold modeled on a real vector space  $\cE$, the topology of which is given by a scalar product $\langle\cdot|\cdot\rangle$.
\end{definition}
Let us recall what a scalar product is. A \emph{scalar product} on the vector space $\cE$ is a bilinear symmetric form 
\begin{align*}
\langle\cdot,\cdot\rangle:\,  \cE\times \cE \to &  \bbR,  \\
 (x,y)\mapsto&  \langle x, y \rangle  \\
\end{align*}
 where:
\begin{enumerate}
\item  The bi-linearity property is satisfied. For any scalars $\alpha, \beta$ and for any $x,x_1,x_2,y,y_1,y_2\in \cE$ the following holds: 
\begin{align*}
\langle\alpha x_{1} +\beta x_{2}, y\rangle=& \alpha\langle x_{1} ,y \rangle+\beta \langle x_{2},y\rangle\\
\langle x, \alpha y_{1} +\beta y_{2}\rangle=&\alpha\langle x ,y_{1}\rangle +\beta \langle x,y_{2}\rangle.\\
\end{align*}

\item The scalar product is symmetric. For any $x,y \in \cE$: \[\langle x,y\rangle=\langle y,x\rangle.\] 
\item The scalar product is positive definite. For any $x,y\in \cE$: \[\langle x, x \rangle \geq 0, \quad  \langle x, x\rangle =0 \,   \iff \, x=0.\]
\end{enumerate}

\, 

A Riemann manifold is said to be a Banach manifold for the norm: \[\Vert x\Vert=\langle x,x\rangle^{\frac{1}{2}},\, x\in \mathcal{E}.\]

In the case where the model space is a Hilbert space, we speak of \emph{ Hilbert manifold}. Hence a Riemann manifold is a real Hilbert manifold.

\,
\subsection{Real Manifolds,  Coordinates, Charts}
\ 

A particularly interesting class of Riemannian manifolds is the manifold endowed with the model space \( \mathcal{E} = \mathbb{R}^{n} \). The usual topology and differentiability structure on \( \mathbb{R}^{n} \) induce a differentiable structure on $ \mathcal{M} $, making it a real $ n $-dimensional differentiable manifold.

In the following, by an $ n$-dimensional (real) manifold $ \mathcal{M} $, we mean a Hausdorff topological space in which every point has a neighborhood homeomorphic to $\mathbb{R}^n$.

\begin{definition}[$n$-dimensional differentiable manifold]\index{Manifold! Finite manifold!Differentiable structure}

An $n$-dimensional real manifold of class $C^k$ is a manifold modeled on $\bbR^{n}$ and such that the condition \(\cM 3\) is replaced by the following property. For each pair of indices $i,j\in I$, the map:
\[
\varphi_j\circ\varphi_i^{-1}:\varphi_i(\cU_i\cap \cU_j) \to \varphi_j(\cU_i\cap
\cU_j),
\]
is a $\cC^{k}$-isomorphism on $\bbR^{n}$. 

Furthermore, for any pairs of indices
$i,j\in I$ the image of the intersection $\cU_i\cap \cU_j$ given by
$\varphi_i(\cU_i\cap \cU_j)$, under the coordinate map $\varphi_i$, is an open subset of $\bbR^{n}$.
\end {definition}
\begin{figure}[h]
\includegraphics[scale=0.6]{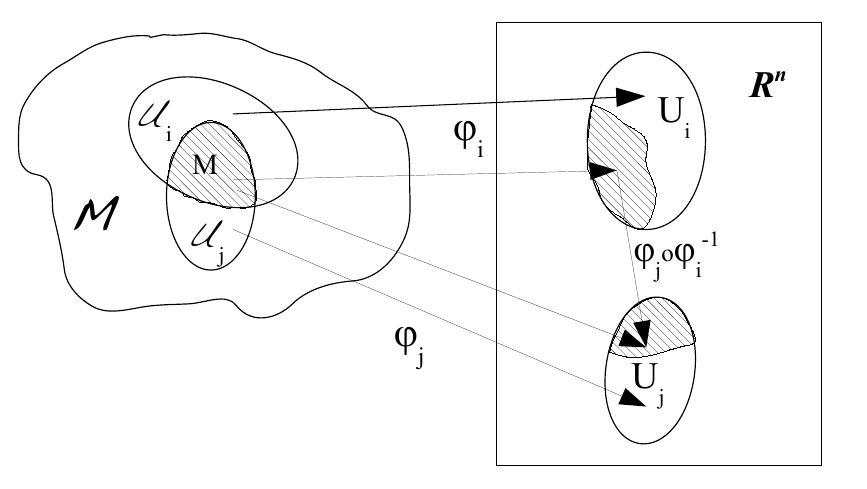} 
\caption{ $n$-dimensional real manifold}\label{F:Chart3}
\end{figure}
A natural way to describe the structure of an 
$n$-dimensional differentiable manifold is through local coordinate systems. These coordinate systems are given by charts, which map open subsets of the manifold to open subsets of $\R^n$, allowing us to analyze the manifold locally as if it were a Euclidean space.
\begin{definition}[Local coordinate system]\index{Manifold! Finite manifold!Local coordinate system}
 Let $(\cU,\varphi)$ be a  chart at the point {\scshape m} $\in \cM$, such that the image $\varphi$({\scshape m})$=x=(x^1,\dots,x^n)\in\bbR^{n}$.
Given a basis $\{e_{1},\dots,e_{n}\}$ of the Euclidean space $\mathbb{R}^{n}$, the coordinates $(x^1,\dots,x^n)$ of the image $\varphi$({\scshape m}) $\in\bbR^{n}$ of {\scshape m}\, $\in U\subset\cM$ are called the \emph{coordinates of {\scshape m} in the chart }$(\cU,\varphi)$. The chart $(\cU,\varphi)$ is also called the local coordinate system.
\end{definition}

\,
Local coordinate systems provide a way to describe differentiable manifolds in terms of open subsets $\bbR^{n}$, allowing us to define smooth functions and analyze local geometric properties. 
In particular, this enables the construction of tangent spaces, which capture the local linear structure of the manifold at each point. The dual spaces to these, known as cotangent spaces, naturally arise when considering differential forms and gradients of functions, playing a fundamental role in differential geometry and analysis on manifolds.
\section{Tangent and cotagent spaces}\index{Tangent vector} 
To grasp the essence of the tangent and cotangent spaces, we begin with the most intuitive objects on a manifold: curves. A curve provides a way to move infinitesimally along the manifold, and by examining how functions change along such curves, we arrive at the concept of directional derivatives. These, in turn, will guide us toward a precise definition of tangent vectors. Finally, this will naturally lead us to the construction of a tangent space, the fundamental linear structure underlying the local geometry of the manifold.
\subsection{Curves on a manifold $\&$ Directional derivatives}

\begin{definition}[Curve]\index{Manifold!Curve}
A curve $\gamma:[a,b]\to \cM$ on a Manifold $\cM$ is a $C^p$-map, where $p\geq 1$ mapping an interval $\mathcal{I}=[a,b]\subset \bbR$ into $\cM$. 
\end{definition}

The curve is said to be smooth if it is a $C^\infty$ map.

\, 
\begin{definition}[Tangent vector  to a curve]\index{Tangent vector!Tangent vector to a curve }
Let $\gamma:t\mapsto \gamma(t)$ be a curve of class $C^1$ such that:
\begin{align*}
    &\gamma(t_0)={\scriptstyle M}_0 \in \cM\\
    &x_0:=\varphi({\scriptstyle M}_0) \in\bbR^n.
\end{align*}
Then, the tangent vector $u_0$ at $\gamma$ in ${\scriptstyle M}_0$ is defined by
\begin{equation}
u_0 =\left. \frac{d \gamma}{dt}\right|_{t_0}.
\end{equation}
 \end{definition}

\begin{definition}[Directional  derivative]\label{D:dirder}\index{Derivative!Directional}

Let \( \mathcal{A}({\scriptstyle M}) \) be the family of \( C^1 \)-functions defined on a neighborhood of \( {\scriptstyle M} \in \mathcal{M} \). Assume \( f \in \mathcal{A}({\scriptstyle M}) \).  

Consider a curve \( \gamma \) of class \( C^1 \),  
\[
\gamma: t \mapsto \gamma(t),
\]
such that at some fixed parameter value \( t_0 \), we have  
\[
\gamma(t_0) = {\scriptstyle M}_0.
\]  
Then, the \emph{directional derivative} of \( f \) in the direction of the curve \( \gamma \) at \( t_0 \) is defined as

\begin{equation}\label{E:dirder}
D_{u_{0}}f(t_{0})=\left. \frac{d\,f\circ \gamma}{dt}\right|_{t_0}, \quad u_{0}=\left.\frac{d\gamma}{dt}\right|_{t_{0}},\quad \text{where}\quad f  \in \cA({\scriptstyle M}_0).
\end{equation}
 \end{definition}
 
 More  generally:
\begin{definition}[Derivation]\index{Derivative!Derivation}
A derivation at a point ${\scriptstyle M}$ on $\cM$ is a linear functional \[D:\cA({\scriptstyle M}) \to \bbR,\] where the Leibniz rule holds:
\begin{align}\label{E:Leibniz}\index{Derivative!Leibniz rule}
D(\alpha f + \beta g)({\scriptstyle M})&=\alpha Df ({\scriptstyle M})+ \beta Dg({\scriptstyle M}), \quad \alpha,\, \beta \in \bbR,\\
D(fg)({\scriptstyle M})&= [(Df)g +fDg]({\scriptstyle M}), \quad {\scriptstyle M}\in \cM.
\end{align} 
\end{definition}

\section{Tangent vector - Tangent space}\index{Tangent space}
\ 

One can associate, via the formula~\eqref{E:dirder}, at each point ${\scriptstyle M}\in \cM$ a tangent vector $u$.

\begin{definition}[Tangent vector]\index{Tangent vector!Differentiable manifold}
A \emph{tangent vector} \( X_{\scriptstyle M} \) at a point \( {\scriptstyle M} \in \mathcal{M} \) on a differentiable manifold \( \mathcal{M} \) is a linear map  
\[
X_{\scriptstyle M} : \mathcal{A}({\scriptstyle M}) \to \mathbb{R},
\]  
where \( \mathcal{A}({\scriptstyle M}) \) denotes the space of functions defined and differentiable in some neighborhood of \( {\scriptstyle M} \in \mathcal{M} \).  

This map satisfies the Leibniz rule: for any \( f, g \in \mathcal{A}({\scriptstyle M}) \),  

\begin{equation}\label{E:Leibniz}
\begin{aligned}
X_M(\alpha f + \beta g) &= \alpha X_M (f) + \beta X_M (g), \quad \alpha, \beta \in \mathbb{R},\\
X_M (f g) &=  f(M) X_M (g) + g(M) X_M (f).
\end{aligned} 
\end{equation}
\end{definition}

\,

Notice that two differentiable functions $f_{1}$ and $f_{2}$ which coincide on a neighborhood of  ${\scriptstyle M}$ have the same  tangent vector. Therefore, a tangent vector at ${\scriptstyle M}$ is the same for all functions belonging to   the class of differential functions, coinciding on a neighborhood of ${\scriptstyle M}$.

\, 

The class of functions which coincide on a neighborhood of  ${\scriptstyle M}$ is called a ``germ'' of $f$. A germ forms an algebra under the sum and the pointwise product. A tangent vector (also called contravariant vector) is a derivation on the algebra of germs of differentiable functions at ${\scriptstyle M}$.

\, 
\begin{definition}[Tangent space]
The space of all tangent vectors at \( {\scriptstyle M} \in \mathcal{M} \), equipped with the addition and scalar multiplication defined by  
\begin{equation}
(\alpha X_{\scriptstyle M} + \beta Y_{\scriptstyle M})(f) = \alpha X_{\scriptstyle M} (f) + \beta Y_{\scriptstyle M} (f),
\quad \alpha, \beta \in \mathbb{R}, \quad X_{\scriptstyle M}, Y_{\scriptstyle M} \in \mathcal{T}_{\scriptstyle M}(\mathcal{M}),
\end{equation}
forms a real vector space, called the \emph{tangent space} at \( {\scriptstyle M} \), and denoted by \( \mathcal{T}_{\scriptstyle M}(\mathcal{M}) \).  

\end{definition}

Let $\gamma\in\mathcal{A}({\scriptstyle M})$ be a curve such that $\gamma(t_0)={\scriptstyle M}_0$ and let $u_{0}=\frac{d\gamma}{dt} \vert_{t_{0}}$ be the tangent vector at ${\scriptstyle M}_0$ along $\gamma$. 
A tangent vector $X_{{\scriptstyle M}_0}$ at ${\scriptstyle M}_0\in \cM$  in the direction $u_{0}$ can be written

\begin{equation}\label{E:tdveccov}
X_{{\scriptstyle M}_0}(f)= D_{u_{0}}f=\left.\frac{d\,f\circ \gamma}{dt}\right|_{t_0},  \quad \forall \,f\in \cA({\scriptstyle M}),\quad u_{0=}\left.\frac{d\gamma}{dt}\right|_{t_{0}}.
\end{equation}

\,

 In the case of an $n$-dimensional real manifold $\cM$ \index{Tangent space! finite dimensional}. Let $(\cU,\varphi)$ be a chart at the point ${\scriptstyle M}\in \cM$ and let  $x=(x^1,\dots,x^n) \in \varphi(\cU)\subset \bbR^n$ be a local coordinate system. A tangent vector at $X_{{\scriptstyle M}_0}$ is the set $\{ X_{{\scriptstyle M}_0}(\varphi^{i}) \}_{i=1}^{n}$, called the local coordinates of the tangent vector, where:
 \[ 
 \varphi^i= \pi^i\circ \varphi \quad \text{ and } \quad \pi^i(x^1,\dots,x^n)=x^i ,
  \] 
and where $\pi^i$ is the projection operateur in $\bbR^{n}$. 
  \begin{itemize}
\item For each $i, \, \partial/\partial x^i|_x ,\ x=\varphi({\scriptstyle M})$ satisfies the equation~\eqref{E:Leibniz}

\item[]
\item The
$\{\partial/\partial x^i|_x\}_{i=1}^n$ forms a basis, called the \emph{natural basis}  of $\cT_x( \cM)$. (Note that this  basis associated to the local coordinate system is not an  orthogonal basis).
\end{itemize}

Moreover, for all $C^1$-functions  on a neighborhood of ${\scriptstyle M}\in \cM$,
\begin{equation}\label{E:tgvecnd}
X_{\scriptstyle M}(f) =\left. \sum_{i=1}^n X^i_{\scriptstyle M}\,\frac{\partial f\circ \varphi^{-1}}{\partial x^i}\right\vert_x \ x={\varphi({\scriptstyle M})},\quad X^i_{\scriptstyle M}= X_{\scriptstyle M}(\varphi^i) .
 \end{equation}

The equation~\eqref{E:tgvecnd} can be written in the abbreviated form:

\begin{equation}\label{E:tgvecnd2}
 X_{\scriptstyle M}= \left. \sum_{i=1}^n X_{\scriptstyle M}^i\,\frac{\partial}{\partial x^i}\right|_x  x={\varphi({\scriptstyle M})},\quad X^i_{\scriptstyle M}= X_{\scriptstyle M}(\varphi^i).
 \end{equation}
 
 \begin{remark}
 Any other basis (or frame\index{Frame}) can be obtained from the local coordinate basis. Let $\{e_{i}\}_{i=1}^{n}$ be a basis of $T_{\scriptstyle M}( \cM)$ with: 
  \[
  e_{k}= \sum_{i=1}^n \Phi_{k}^{i}\frac{\partial}{\partial x^i}= \sum_{i=1}^n\Phi_{k}^{i} \partial_{i},
  \]
 where $\Phi_{k}^{i}$ is the matrix corresponding to an invertible linear transformation. 
 We have
  \[X_{\scriptstyle M}= \sum_{k=1}^{n}\chi^{k}_{\scriptstyle M} e_{k}, \text{ with } \chi^{k}_{\scriptstyle M}=\Phi^{k}_{i}X^{i}_{\scriptstyle M}, \]
 the $X^{k}_v$ being the component of $X_{\scriptstyle M}$ in the basis $\{e_{i}\}_{i=1}^{n}$.
 \end{remark}

Tangent and cotangent spaces are dual vector spaces, linked by a canonical pairing. We discuss now the notion of cotangent vector spaces. 
\,

 
\subsection{Cotangent vector, Cotangent space}\index{Cotangent space}
\ 

\,
\subsubsection{\bf Cotangent vectors $\&$ Differential 1-forms}
\begin{definition}[Cotangent vector space]
Let  $\cM$ be a manifold. The dual space $\cT^\star_{\scriptstyle M}(\cM)$ to the tangent
vector space $\cT_{\scriptstyle M}(\cM)$ in ${\scriptstyle M}\in \cM$  is the space of linear forms on  $\cT_{\scriptstyle M}(\cM)$. It is a vector
space  called the \emph{cotangent vector space} to  $\cM$ at ${\scriptstyle M}$. 

\, 

The elements of
$\cT^\star_{\scriptstyle M}(\cM)$ are called cotangent vectors, or covariant vectors, or covectors,\index{Covector} or differential 1-forms. 
\end{definition}

\noindent
Let $\omega_{\scriptstyle M} \in \cT^\star_{\scriptstyle M}(\cM)$ be a differential 1-form and consider $X_{\scriptstyle M} \in \cT_{\scriptstyle M}(\cM)$. Then:
\begin{equation}
\left\{ \begin{aligned} &\omega_{\scriptstyle M}(X_{\scriptstyle M} ) \in \bbR,\\ 
 &X_{\scriptstyle M}(\omega_{\scriptstyle M})= \omega_{\scriptstyle M}(X_{\scriptstyle M}).
\end{aligned}\right. 
\end{equation}

\subsubsection{\bf Finite dimensional real manifold}\index{Cotangent space!Finite dimensional}
In the case of  an $n$-dimensional real manifold,
a useful canonical isomorphism between a space and its dual can be chosen as follows.

\, 

Let
$(e_1,e_2,\dots,e_n)$ be a basis in  $\cT_{\scriptstyle M}(\cM)$.  We may construct its dual
$(\varepsilon^1,\varepsilon^2,\dots,\varepsilon^n)$ by
\begin{equation}
\varepsilon^i(X_{\scriptstyle M})= X^i_{\scriptstyle M}, 
\end{equation}
where $X^i_{\scriptstyle M}$ is the component of $X_{\scriptstyle M}$ in the basis $\{e_j\}$ and 
\begin{equation}
\varepsilon^i(e_j)= \delta^i_j.
\end{equation}

If we have chosen the natural basis $\{e_{i}=\partial/\partial x^i\}_{i=1}^{n}$ then the dual basis is denoted by  $\{\epsilon^{i}=dx_{i}\}_{i=1}^{n}$ with
\begin{equation}
dx^i\left(\frac{\partial}{\partial x^j}\right) =\frac{\partial}{\partial x^j}(dx^i)=\delta^i_j;
\end{equation}
a tangent vector $X_{\scriptstyle M}\in \cT_{\scriptstyle M}$ takes the form
\[
X_{\scriptstyle M} = \sum_{i=1}^n X^{i}_{\scriptstyle M}\frac{\partial}{\partial x^{i}},\quad X^{i}_{\scriptstyle M}=dx_{i}(X_{\scriptstyle M});
\]
whereas a cotangent vector $\omega_{\scriptstyle M} \in \cT_{\scriptstyle M}^\star(\cM)$ in the dual basis is given by
\[
\omega_{\scriptstyle M} = \sum_{i=1}^n \omega_{\scriptstyle M}{}_i\, dx^{i},\qquad  \omega_{\scriptstyle M}{}_i=\omega\left(\frac{\partial}{\partial x^{i}}\right)=\frac{\partial \omega}{\partial x^{i}}.
\]

For a real manifold of finite dimension, we have the following equality:
\begin{equation}
\cT^{\star \star}_{\scriptstyle M}(\cM)= \cT_{\scriptstyle M}(\cM).
\end{equation}

\subsection{Implications of a chart change}

Let $ \cM$ a $n-$dimensional real manifold. A chart $(\cU,\varphi)$ with a local coordinate system $\varphi:{\scriptstyle M}\to \varphi({\scriptstyle M})=x=(x^1,\dots,x^n)\in \bbR^n$  at $M \in \cM$ being given, the local coordinate basis on the tangent space is given by   \[\{e_{i}=\partial/\partial x^{i}\}_{i=1}^{n}\]  and its dual basis is \[\{\epsilon^{i}=dx^{i}\}_{i=1}^{n}.\]  

\, 

Had we chosen a different chart, say  $(\cU',\varphi')$ with $\varphi': {\scriptstyle M}\to  x'=({x'}^1,\dots,{x'}^n)$ at ${\scriptstyle M}\in \cM$,  a local coordinate basis can be given by $ \{e'_{i}=\partial/\partial {x'}^{1}\}_{i=1}^{n}$ as well as its dual basis $ \{{\epsilon'}^{i}=d{x'}^{i}\}_{i=1}^{n}$. We can certainly express one local coordinate basis in terms of the other on the open set $\varphi(\cU)\cap\varphi'(\cU')$.

\, 

Let us break this statement down. Under a given change of a coordinates, the local coordinates of a tangent vector are described as follows:  
\[\begin{aligned}
 X_{\scriptstyle M}= \sum_{i=1}^n X_{\scriptstyle M}^i \frac{\partial}{\partial x^{i}}=& \sum_{i=1}^n {X'}_{\scriptstyle M}^i \frac{\partial}{\partial {x'}^{i}},\\
X_{\scriptstyle M}^i=\varphi^i({\scriptstyle M}), \quad  {X'}_{\scriptstyle M}^i={\varphi'}^i({\scriptstyle M}),\quad 
&\frac{\partial}{\partial x^i}= \sum_{j=1}^n\left. \frac{\partial {x'}^j}{\partial {x}^i}\right|_{\varphi({\scriptstyle M})} \frac{\partial}{\partial {x'}^{i}},
\end{aligned}
 \]
So, the vector transformation law is given by: 
\begin{equation}
{X'}_{\scriptstyle M}^i= \sum_{j=1}^n\left. \frac{\partial {x'}^i}{\partial {x}^j}\right|_{\varphi({\scriptstyle M})} X_{\scriptstyle M}^j.
\end{equation}

\, 

Similarly, for the covector we have
\[ 
\omega_{\scriptstyle M}= \sum_{i=1}^n \omega_idx^{i}= \sum_{i=1}^n \omega'_i d{x'}^{i},\qquad  \omega_i= \omega\left(\frac{\partial}{\partial x^{i}}\right)\ \omega'_i= \omega\left(\frac{\partial}{\partial {x'}^{i}}\right)\]
\begin{equation}
\begin{aligned}\label{E:chartchange}
dx^{i}&= \sum_{j=1}^n\left. \frac{\partial {x}^i}{\partial {x'}^j}\right|_{\varphi({\scriptstyle M})} d{x'}^{j}\\
\omega'_{i}&= \sum_{j=1}^n\left. \frac{\partial {x}^j}{\partial {x'}^i}\right|_{\varphi({\scriptstyle M}} \omega_{j}.
\end{aligned}
\end{equation}

Having explored the implications of a changing the chart for vectors and covectors, a logical next step is to establish a rigorous definition of a function's differential.

\subsection{Differential of a function}\index{Differential}
\begin{definition}
Let $f$ be a differentiable function.
The differential $df|_{\scriptstyle M}$ of $f$ on $\cM$ in the neighborhood of a point ${\scriptstyle M}$ is defined by the 1-form 
\begin{equation}\label{difffunc}
df|_{\scriptstyle M}(X_{\scriptstyle M})=X_{\scriptstyle M}(f).
\end{equation}
\end{definition}

In the natural basis
\[
df|_{\scriptstyle M}(X_{\scriptstyle M})= \left. \sum_i X_{\scriptstyle M}^{i}\frac{\partial f}{\partial x^{i}} \right\vert _{\scriptstyle M}.
\]

More generally, let $\{e_{i}\}_{i=1}^{n}$ be a basis in $\cT_{\scriptstyle M}(\cM)$ and let $\{\epsilon_{i}\}_{i=1}^{n}$ be its dual basis. The value of
 $df|_{\scriptstyle M}\in \cT^{\star}_{_{M}}(\cM)$ at $X_{\scriptstyle M}\in\cT_{\scriptstyle M}(\cM)$ is given by 
 
 \[\begin{aligned}   
df|_{\scriptstyle M}(X_{\scriptstyle M})=&X_{\scriptstyle M}(f)=\\
\sum_i X^{i}e_{i}(f)=& \sum_i e_{i}(f)\epsilon^{i}(X_{\scriptstyle M}).
\end{aligned}\]

Hence,
\begin{equation} \label{E:difcoord}
df|_{\scriptstyle M}= \sum_i e_{i}(f)\epsilon^{i}.
\end{equation}

\,
Having previously established the fundamental geometric notions of vectors, differential forms (linear maps on vectors), along with differentials of functions (a type of 1-form), we naturally arrive to the notion of tensors, which encode multilinear maps between these objects.

\section{Tensor at a point of a manifold}\label{S:tensor}\index{Tensor}

For simplicity, in this section, we assume that the manifold is finite dimensional.

\,

\subsection{Multilinear maps $\&$ Tensors of type $(r,s)$}\index{Manifold!Tensor at a point}
\begin{definition}
Given a point ${\scriptstyle M}\in \cM$, a tensor $T^{(r,s)}_{\scriptscriptstyle M}$ of type $(r,s)$ is a multi-linear form on the space
$\underbrace{ \cT_{\scriptstyle M}\times \dots \times \cT_{\scriptstyle M}}_{r}\times \underbrace{\cT_{\scriptstyle M}^\star \times\dots \times \cT_{\scriptstyle M}^\star }_{s}= \cT_{\scriptstyle M}^{\otimes r}\times  {\cT_{\scriptstyle M}^\star}^{\otimes s},$

defined by

\begin{align*}
T_{\scriptscriptstyle M}:\underbrace{ \cT_{\scriptstyle M}\times \dots \times \cT_{\scriptstyle M} }_{r}\times \underbrace{\cT_{\scriptstyle M}^\star \times\dots \times \cT_{\scriptstyle M}^\star }_{s}\to & \bbR.\\
(X_{1},\quad, \dots,\quad \, X_{r}, \omega_{1},\quad\cdots,\quad\omega_{s})\mapsto & T_{\scriptscriptstyle M}(X_{1},\dots, X_{r},\omega_{1},\dots,\omega_{s}).\\
\end{align*}

Furthermore, this object adheres to the following properties. 
 \[
\begin{aligned} 
\hspace{.2cm}T_{\scriptscriptstyle M}(X_{1},\dots,&\alpha X_{j}+ \beta Y_{j},\dots, X_{r},\omega_{1},\dots,\omega_{s})\\&=
\alpha T_{\scriptscriptstyle M}(X_{1},\dots, X_{j},\dots, X_{r},\omega_{1},\dots,\omega_{s})+\beta T_{\scriptscriptstyle M}(X_{1},\dots, Y_{j},\dots, X_{r},\omega_{1},\dots,\omega_{s}),\\
\hspace{.2cm}T(X_{1},\dots,&, X_{r},\omega_{1}\dots,\alpha\omega_{k}+\beta \eta_{k},\dots,\omega_{s})\\&=
\alpha T_{\scriptscriptstyle M}(X_{1},\dots, X_{r},\omega_{1},\dots,\omega_{k},\dots,\omega_{s})+\beta T_{\scriptscriptstyle M}(X_{1},\dots,  X_{r},\omega_{1},\dots,\eta_{k},\dots,\omega_{s}),\\
\end{aligned}\]
for scalars $\alpha,\beta \in \bbR$ and all $1\leq j\leq r$ as well as $1\leq k\leq s$.

Moreover, if $T_{\scriptscriptstyle M}$ and $S_{\scriptscriptstyle M}$ are two tensor of the same type $(r,s)$, then
\begin{equation}
\begin{aligned}
(\alpha T_{\scriptscriptstyle M}+ &\beta S_{\scriptscriptstyle M})(X_{1},\dots, X_{r},\omega_{1},\dots,\omega_{s})\\
&=\alpha T_{\scriptscriptstyle M}(X_{1},\dots, X_{r},\omega_{1},\dots,\omega_{s})+ \beta S_{\scriptscriptstyle M}(X_{1},\dots, X_{r},\omega_{1},\dots,\omega_{s} ), \end{aligned}\end{equation}
and tensors of a given type $(r,s)$ span a linear vector space of dimension $n^{r+s}$.
\end{definition}

\,

The following denominations are often used in the literature. 
\begin{itemize}
\item A (covariant) tensor of order $r$ at $\scriptstyle M$ is a tensor on $\cT_{\scriptstyle M}(\cM)^{\otimes r}$. It is of type $(r,0)$\\
\item  A contravariant tensor of order $s$ at $\scriptstyle M$, is a tensor on ${\cT^\star _{\scriptstyle M}(\cM)}^{\otimes s}$. It is of type $(0,s)$.
\end{itemize}

\,
\subsection{Tensor products}
\begin{definition}[Tensor product]\index{Tensor!Tensor product}
Let $T_{\scriptscriptstyle M}$ be a tensor at the point $\scriptstyle M\in \cM$ of type $(r,s)$ and let $S_{\scriptscriptstyle M}$ be a tensor of type $(p,q)$. Then, the tensor $T_{\scriptscriptstyle M}\otimes S_{\scriptscriptstyle M}$ defined by
\begin{equation}\begin{aligned}
T_{\scriptscriptstyle M}&\otimes S_{\scriptscriptstyle M}(X_{1},\dots,X_{r},\dots,X_{r+p},\omega_{1},\dots,\omega_{s},\dots,\omega_{s+q})\\
&=T_{\scriptscriptstyle M}(X_{1},\dots,X_{r},\omega_{1},\dots,\omega_{s})\times  S_{\scriptscriptstyle M}(X_{r+1},\dots,X_{r+p},\dots,\omega_{s+1},\dots,\omega_{s+q}).
\end{aligned}\end{equation}
is called the tensorial product of $T_{\scriptscriptstyle M}$ and $S_{\scriptscriptstyle M}$.
\end{definition}

Admitting this definition, an $(r,s)$-tensor $T$ can be seen as the tensorial product of $r$ covariant vectors and $s$ tangent vectors at $\scriptstyle M \in \cM$ as we can see: 
\[
T(X_{1},\dots,X_{r},\omega_{1},\dots,\omega_{s})= T_{(1)}\otimes\dots\otimes T_{(r)}\otimes T^{(1)}\otimes\dots\otimes T^{(s)}(X_{1},\dots,X_{r},\omega_{1},\dots,\omega_{s}),\]
where the $T_{(i)}$ are covectors and the $T^{(k)}$ are contravectors. 

BY abuse of notation, we have omitted the index $\scriptscriptstyle M$ to have a more convenient formula. However, it is important to remember that the definition is local.

\,

To conclude, the space of tensors of type $(r,s)$ can be identified with the tensorial product space:	
\[ \cT_{\scriptstyle M}^{\star}(\cM)^{\otimes r} \otimes T_{\scriptstyle M}(\cM)^{\otimes s }.\]	

\,


Let us choose a basis of $ T_{_{M}}^{\star}(\cM)^{\otimes r} \otimes T_{_{M}}(\cM)^{\otimes s }$ given by the $n^{(r+s)}$ vectors:

\begin{equation}
 \mathbf{e}^{i_{1}\dots i_{r}}_{\phantom{i_{1}\dots i_{r}}j_{1}\dots j_{s}}= \epsilon^{i_{1}}\otimes\dots\otimes\epsilon^{i_{r}}\otimes e_{j_{1}}\otimes\dots\otimes e_{j_{n}}\end{equation}  
where $\{\epsilon^{i}\}_{i=1}^{n}$ is the dual basis of  $\{e_{j}\}_{j=1}^{s}$.

\vspace{3pt}

If we work in the framework of local coordinate basis the associate basis of the $(r,s)$-tensor space is
\[
\mathbf{e}^{i_{1}\dots i_{r}}_{\phantom{i_{1}\dots i_{r}}j_{1}\dots j_{s}}=dx^{i_{1}}\otimes\dots\otimes dx^{i_{r}}\otimes \partial_{j_{1}}\otimes\dots\otimes \partial_{j_{n}}=\bigotimes_{k=1}^{r}dx^{i_{k}} \bigotimes_{\ell=1}^{s}\partial_{j_{\ell}}, \quad 1\leq i_{_{k}}, j_{_{\ell}}\leq n,
\]
 then the  $(r,s)$-tensor $T$ is given by:

\begin{equation}
T=\sum_{i_{k}=1\atop 1\leq k\leq r}\sum_{j_{\ell}=1\atop 1\leq\ell\leq s}T_{i_{1}\dots i_{r}}^{\phantom{i_{1}\dots i_{r}}j_{1}\dots j_{s}}dx^{i_{1}}\otimes\dots\otimes dx^{i_{r}}\otimes \partial_{j_{1}}\otimes\dots\otimes \partial_{j_{n}}.
\end{equation}

To enhance the clarity, it is advantageous to adopt the Einstein summation convention, where repeated upper and lower indices are implicitly summed over. In the following formula, we illustrate this convention.

\begin{equation}
T=T_{i_{1}\dots i_{r}}^{\phantom{i_{1}\dots i_{r}}j_{1}\dots j_{s}}dx^{i_{1}}\otimes\dots\otimes dx^{i_{r}}\otimes \partial_{j_{1}}\otimes\dots\otimes \partial_{j_{n}},
\end{equation}

where the components are given by,
\begin{equation}
T_{i_{1}\dots i_{r}}^{\phantom{i_{1}\dots i_{r}}j_{1}\dots j_{s}}=\partial_{i_{1}}\otimes\dots\otimes \partial_{i_{r}}\otimes dx^{j_{1}}\otimes\dots\otimes dx^{j_{n}}(T)= \bigotimes_{k=1}^{r}\partial_{i_{k}} \bigotimes_{\ell=1}^{s}dx^{j_{\ell}}(T),
\end{equation}

Under a change of coordinates, an $(r,s)$-tensor transforms via 
$r$ applications of the Jacobian (for covariant indices) and 
$s$ applications of its inverse (for contravariant indices):


\begin{equation}
T_{i'_{1}\dots i'_{r}}^{\phantom{i'_{1}\dots i'_{r}}j'_{1}\dots j'_{s}}=\prod_{k=1}^{r}\frac{\partial x^{i_{k}}}{\partial x^{i'_{k}}} \prod_{\ell=1}^{s}\frac{\partial x^{j'_{\ell}}}{\partial x^{j_{\ell}}}T_{i_{1}\dots i_{r}}^{\phantom{i_{1}\dots i_{r}}j_{1}\dots j_{s}},
\end{equation}

We leverage this framework to study  symmetries of tensors. The study of tensor symmetries within this framework reveals fundamental insights into their invariant properties and multilinear algebraic behaviour.

\subsection{Symmetry properties of tensors}\index{Tensor!Symmetry properties}

Let us first consider a (covariant) tensor of order $r$. Let $S_{r} $ be the group of permutations of the $r$ integers $\{1,2,\dots,r\}$. By definition, it acts on the tangent space at $\scriptstyle M \in \cM$ in the following way:
\[ (\sigma T)(X_{1},\dots,X_{r})= T(X_{\sigma (1)},\dots,X_{\sigma (r)}), \quad \sigma \in S_{r},\ X_{k}\in T_{_{M}}(\mathcal{ M}),\]
or in local coordinates
\[   (\sigma T)_{i_{1}\dots i_{r}}= T_{\sigma(i_{1})\dots \sigma (i_{r})}, \]
so that $T$ has the symmetry defined by $\sigma$.
\begin{definition} 

\ 

$\bullet$   If $\sigma T=T$  then the tensor $T$ is said to be symmetric.

\vspace{3pt}
$\bullet$  If $\sigma T=sign(\sigma)\,T$, with $sign(\sigma)=\pm1$ (depending on  whether the permutation is even or odd)  the tensor $T$ is said to be antisymmetric.
\end{definition}

One can define a symmetrization operator $S$ as well as an antisymmetrization operator $A$ on the (covariant) tensor $T$. This is done below: 

\begin{equation}\begin{aligned}
ST&=\frac{1}{r!}\sum\limits_{\sigma\in S_{r}} \sigma\, T:\,  \text{ a completely symmetric tensor.}\\\\
AT&= \frac{1}{r!}\sum\limits_{\sigma\in S_{r}} sign(\sigma)\,\sigma \,T:\,  \text{ a completely antisymmetric tensor,}
\end{aligned}\end{equation}

which in local coordinates is defined by 

\begin{equation}(AT)_{i_{1}\dots i_{r}}=\frac{1}{r!}\epsilon_{i_{1}\dots i_{r}}^{k_{1}\dots k_{r}}T_{k_{1}\dots k_{r}}\end{equation} where we have used the Einstein summation rule and 
\[ 
\epsilon_{i_{1}\dots i_{r}}^{k_{1}\dots k_{r}}=\begin{cases}
\phantom{+}0& \text{if } ( k_{1}\dots k_{r}) \text{ is not a permutation of } (i_{1}\dots i_{r})\\
+1& \text{if } ( k_{1}\dots k_{r}) \text{ is an even permutation of }(i_{1}\dots i_{r})\\
-1& \text{if } ( k_{1}\dots k_{r}) \text{ is an odd permutation of }(i_{1}\dots i_{r})\\
\end{cases}
\] 
is the Kronecker tensor.

\,

The symmetry properties of a contravariant tensor is defined similarly. 

\begin{remark}
The permutation $\sigma$ permutes only the labels of objects of the same nature. For a $(r,s)$-tensor, the symmetry and antisymmetry properties are defined only for indices of the same nature. 
It is usual to use square brackets $[\cdot]$ for the antisymmetry and   round brackets ($\cdot$) for symmetry. For example, the $(4,3)$-tensor $T_{[abc]d}^{\phantom{[abc]d}(ef)g}$ is antisymmetric in $a,b,$ and symmetric in $e,f$.
\end{remark}

\chapter{Fiber bundles}\index{Bundle}

In the language of modern geometry, fiber bundles serve as a unifying framework, extending classical notions such as vector bundles to a broader setting. They naturally arise in various mathematical and physical contexts, providing a structural backbone for spaces with local product structures.

\,
A principal application appears in Gauge Theory, where principal bundles furnish a natural setting for describing gauge fields in physics, including the Yang-Mills theories. Here, the connection on a bundle encodes the dynamics of fundamental interactions, and curvature expresses field strength in a geometrically intrinsic way.

\, 
In topology, fiber bundles play a fundamental role in the theory of characteristic classes, such as the Chern classes, which assign global topological invariants to vector bundles. These invariants capture deep global properties of manifolds, serving as indispensable tools in the study of topology, geometry, and mathematical physics.

\, 
\section{Definitions and examples}

The notion of bundle have been introduced to generalize topological product.

\begin{definition}[Fiber bundle]\index{Bundle!Fiber bundle}
A fiber bundle is a triple $(\cB,\cM,\pi)$ consisting of two topological spaces
$\cB$ and $\cM$ and a continuous surjective map
\[
\pi:\cB \to \cM,
\]
where $\pi$ is called the projection of the total space $\cB$ onto the base space $\cM$.
\end{definition}

\vspace{3pt}
The topological space $\cF_{\scriptstyle M}= \pi^{-1}({\scriptstyle M}), \, {\scriptstyle M}\in \cM$
 is called the fibre at ${\scriptstyle M}\in \cM$.

\

\begin{example}
The simplest example of bundle is the product bundle defined by $(E_{1}\times E_{2},E_{1},\pi) $ where $E_{i}$ is isomorphic to $\bbR$ and with $\pi(x,y)=x$, for all $x \in E_{1}$ and $y\in E_{2}$. This is illustrated by the following diagram.

\begin{center}
    \begin{tikzpicture}[scale=0.7]
    \draw[thick,->] (0,0) -- (5,0) node[below right] {$E_1$};
    \draw[thick,->] (0,0) -- (0,4) node[above left] {$E_2$};
    \node at (1.3,1.5) { $E_1 \times E_2$};
    \filldraw (3,2) circle (2pt);
    \node[above right] at (3,2) { $(x,y)$};
    \draw[dashed,->] (3,2) -- (3,0.3);
    \filldraw (3,0) circle (2pt);
    \node[below] at (3,0) {$(x,0)$};
    \node at (3.5,1) { $\pi$};
\end{tikzpicture}
\end{center}

So, in this example the total space is given by the topological product $\cB= E_1 \times E_2$ and the base space is $\cM=E_1 $. The projection map is given by $\pi(x,y)=x$, where we project $E_1 \times E_2$ onto $ E_1$. Naturally, we could have chosen to project on $E_2$. 
\end{example}

\,

The need to generalize topological products can be seen on the following examples.

a) The cylinder is obtained by taking the topological product $S^{1}\times \mathcal{I}$, where $\mathcal{I}$ is  a segment. The fiber bundle is defined by the triple $(S^{1}\times \mathcal{I},S^{1},\pi)$, where here we have chosen to define the projection map as follows $\pi:S^{1}\times \mathcal{I} \to S^1$ .

\begin{center}
\begin{tikzpicture} 
     \draw[thick] (-9,0) rectangle (-4,2.1);
         \draw[ultra thick,blue, ->] (-9,0) -- (-9,2.1);
    \draw[ultra thick,red, ->] (-4,0) -- (-4,2.1);
   \draw[ultra thick,black, ->] (-9,0) -- (-4,0);
    \draw[ultra thick, black, ->] (-9,2.1) -- (-4,2.1);
     \node at (-9,-0.2) { $B$}; 
       \node at (-4,-0.2) { $B'$}; 
  \node at (-9,2.4) { $A$};
   \node at (-4,2.4) { $A'$}; 

  \draw[dashed] (2,0) arc[start angle=0,end angle=180,x radius=2cm, y radius=0.8cm];

    \draw[thick] (2,0) arc[start angle=0,end angle=-180,x radius=2cm, y radius=0.8cm];

    \draw[thick] (0,2) ellipse (2cm and 0.8cm);   
      \draw[ultra thick,blue,->] (-2,0) -- (-2,2);
     \draw[ultra thick,red,->] (-2.09,0) -- (-2.09,2);
    \draw[thick] (2,0) -- (2,2.1);
    \node at (0,-1.4) {\Large $S^1$};   
      \node at (2.4,1) {\Large $\mathcal{I}$}; 
 \node at (-2.0,-0.2) { $B\, B'$}; 
  \node at (-2,2.5 ) { $A\, A'$}; 
\end{tikzpicture}
\end{center}

\,

b) The M\"obius band obtained by twisting a sheet and then gluing the opposite edges forms only a local topological product and is {\it no longer a topological product}. It can be done locally: for an open subset $U\subset S^1$, the topological product $U \times \cI$ describes a segment of the  M\"obius band, but does not take under account  the twisting operation.

\begin{figure}[h]
\includegraphics[scale=0.6]{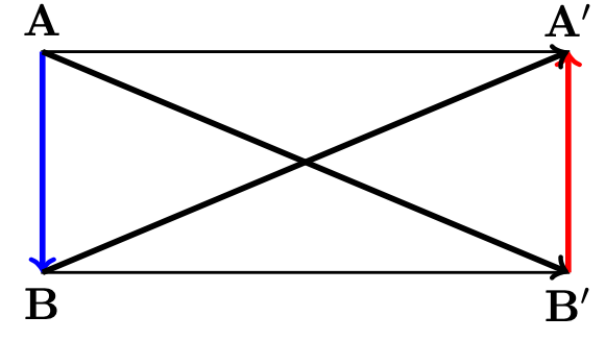}
  \hspace{0.3cm}\includegraphics[scale=0.25]{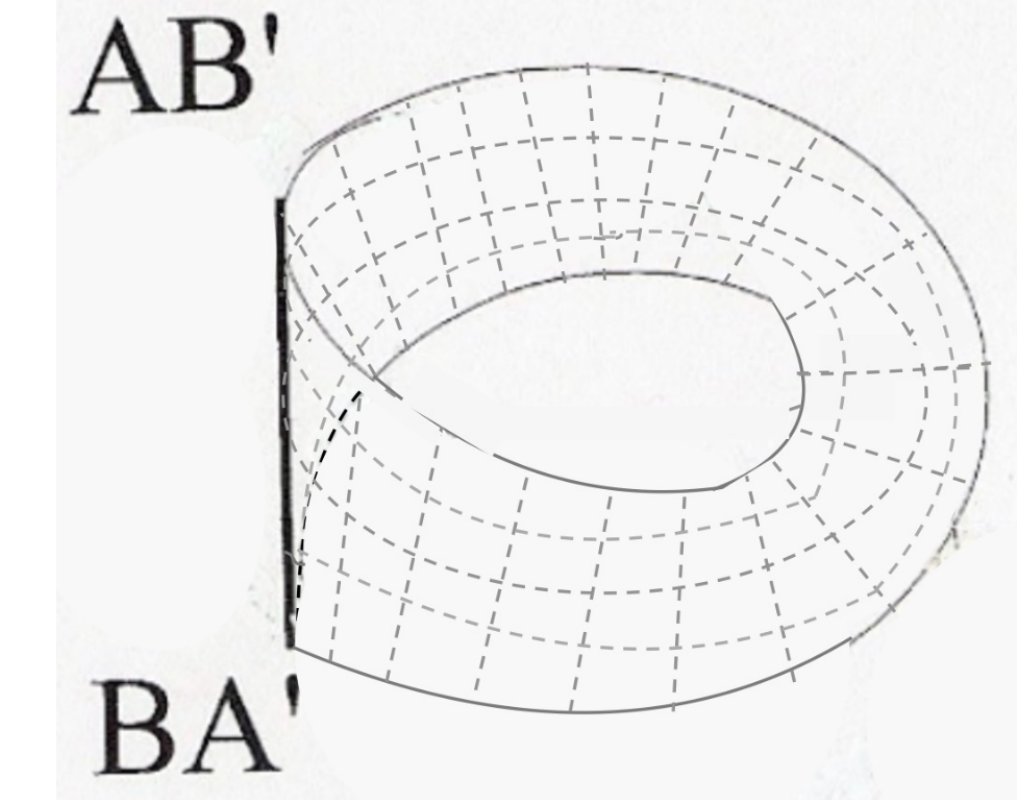}
 \end{figure}

\,

\subsection{Properties}
In the following we restrict ourselves to the case  where the topological space $\cF_{\scriptstyle M}= \pi^{-1}({\scriptstyle M}), \, {\scriptstyle M}\in \cM$ are  homeomorphic to a space $F$, called typical fibre.

\begin{definition}[Trivial bundle]\index{Bundle!Trivial bundle}
A trivial bundle $\cB$ is a fiber bundle homeomorphic to a product bundle $\cM\times \cF$ and
$ \pi $ is just the projection from the product space to $\cM$.
\end{definition}

\begin{definition}[Vector bundle]\index{Bundle!Vector bundle}	
	\noindent If $\cF_M$ is a vector space, $(\cB,\cM,\pi)$ is  called vector bundle.
\end{definition}

\begin{definition}[G-bundle]\index{Bundle!G-bundle}
A fiber bundle $(\cB,\cM,\pi, G)$ is  made of:
\begin{enumerate}
	\item a bundle $(\cB,\cM,\pi)$
together with a typical fiber $\cF$, 

\vspace{3pt}
	\item a topological group $G$ of homeomorphisms of $\cF$ onto itself (called the structural group)
	
	\vspace{3pt}
	\item a covering by open sets $\{U_j\}_{j\in J \subseteq \bbN}$ where  $U_j \in\cM$  such that:

\vspace{3pt}
\begin{enumerate}
\item Locally, the bundle is a trivial bundle.

\vspace{3pt}
\item Let ${\scriptstyle M}\in U_i\cap U_j$, the homeomorphism $ {\psi_{\scriptscriptstyle M}}_i\circ {\psi_{\scriptscriptstyle M}}_{j}^{-1}: \cF \to\cF$, where
${\psi_{\scriptscriptstyle M}}_{i}$ denote a homeomorphism from $\cF_x$ onto $\cF$, is an element of the structural group $G$ for all $i,j \in J$.

\vspace{3pt}
\item The induced mapping $\gamma_{ij}: U_i\cap U_j \to G$ by $\gamma_{ij}(x)=  {\psi_{\scriptscriptstyle M}}_i\circ {\psi_{\scriptscriptstyle M}}_{j}^{-1}$ is continuous.
\end{enumerate}
\end{enumerate}
\end{definition}

\begin{definition}[Principal fiber bundle]\index{Bundle!Principal fiber bundle}
A $G$-bundle $(\cB,\cM,\pi, G)$ in which the typical fiber $\cF$ and the structural groupe $G$ are isomorphic by left translation~\footnote{Let $G$ be a group the action \[L_{g}:G\to G \text{ by } L_{g}(h)=gh\] is called left translation.} is called a principal fiber bundle.
\end{definition}

\begin{definition}[Vector bundle]\index{Bundle!Vector bundle}	
	\noindent If $\cF$ is a vector space, $(\cB,\cM,\pi)$ is  called vector bundle.	
\end{definition}
A vector bundle is a $G$-bundle with $G$ the linear group of $\cF$.

\begin{definition}[Cross-section of a bundle] \index{Bundle!Cross-section}
A cross-section of the bundle $(\cB,\cM,\pi,G)$ is a mapping $\sigma: \cB  \to
\cM$ such that $\sigma \circ \pi = I_\cM$ the identity in $\cM$.
\end{definition}

\begin{theorem}
 A principal fiber bundle $(\cB,\cM,\pi, G)$ is trivial if and only if it is a continuous cross-section.
\end{theorem}

\begin{proof}
First assume that $\cB$ has a cross-section $\sigma: \cM \to \cB$. Then
\[
\pi \sigma(x) =x \text{ for } x\in \cB. 
\]
Given $p\in \cF$, there exists a unique $g_0$ such that $p= L_{g_0}=g_0\sigma(x)$. 
Then, 
\[
 \phi_\sigma : \cB \times G \text{ by } p\mapsto (x,g_0)
 \]
 is a homomorphism which preserves the group structure of the fibers
\[ 
\phi_\sigma(L_{g'} p) =L_{g'}\phi(p) \quad \forall g'\in G, \forall p\in \cB.
\]
In particular, $\phi_\sigma(f(x)=(x,e)$, where $e$ is the identity of $G$.
This shows that the existence of a continuous cross-section implies the triviality.

Conversely, if the principal bundle is trivial i.e. $\cB= \cM \times G$,  then \[f:\cM \to \cM \times G\] defined by $x\mapsto (x,k(x))$ (where $k : X \to G$ is some continuous mapping) forms a cross-section of $\cB$.
\end{proof}

\subsection{Tangent bundles and frame bundles} The notion of tangent spaces can be upgraded to the structure of a tangent bundle. We describe this in the next section below. 
\subsection{\bf Tangent bundle}
\begin{definition}[Tangent  bundle] \index{Bundle!Tangent bundle}
A tangent bundle $(\cT(\cM),\cM,\pi,G)$ is a bundle with fibers being the tangent spaces. 
\end{definition}

In the case of a $n$-dimensional manifold, the tangent-bundle  $(\cT(\cM), \cM, \pi,G)$ can be identified to the natural bundle defined by
\begin{itemize}
	\item a fiber at ${\scriptstyle M}\in \cM$, $\cF_{{\scriptstyle M}}= \cT_{\scriptstyle M}(\cM)$ and typical fiber $\bbR^{n}$,
	\item a projection $\pi:({\scriptstyle M},X_{{\scriptstyle M}}) \mapsto {\scriptstyle M}$, 
		\item the structural group $G=GL(n,\bbR)$ of linear automorphisms on
$\bbR^{n}$.
\end{itemize}

\subsection{Frame bundle}
Let $\cM$ a $n$-dimensional manifold. A frame $\phi_{\scriptscriptstyle M}$ associated with the tangent bundle $\cT_{\scriptstyle M}(\cM)$ of $\cM$ is a set of $n$ linearly independent vectors $\{f_{1},\dots,f_{n}\}$ which can be expressed as a linear combination of a particular basis  $\{e_{1},\dots,e_{n}\}$ of  $\cT_{\scriptstyle M}(\cM)$, such that 
\[f_{i}=a_{i}^{j}e_{j},\quad \text{for}\quad a \in GL(n,\bbR)\quad \text{and}\quad i\in \{1,\cdots, n\}. \]
This implies that there exists a bijection between the set of all frames in  $\cT_{\scriptstyle M}(\cM)$ and the group $GL(n,\bbR)$.

\begin{definition}[Frame bundle]
Let us consider the collection \[\Phi(\cM)=\{({\scriptstyle M},\phi_{\scriptscriptstyle M}),\, {\scriptstyle M}\in \cM\}\] and $\cM$ is equipped with a differentiable structure. Then, the four-tuple 

$(\Phi(\cM),\cM,\pi, GL(n,\bbR))$ 
with typical fiber $GL(n,\bbR)$ and structural group $GL(n,\bbR)$ is called the frame bundle on $\cM$.
\end{definition}

\, 

Furthermore, a frame bundle associated with a vector bundle 
forms  a principal $GL(n,\bbK)$-bundle, where $\bbK$ is a field of characteristic 0 such as $\bbR$ or $\bbC$. The general linear group acts freely and transitively on the frames via basis changes. 

\, 

The topology on the fiber bundle associated to a vector bundle $E$
 is constructed using local trivializations of 
$E$. Each trivialization induces a bijection between the fiber over 
an open set $U_i$ and $U_i\times GL(n,\bbK)$, with the final topology ensuring compatibility across overlapping regions.
The fibers of the frame bundle are $GL(n,\bbK)$-torsors. A \emph{torsor} (or  principal homogeneous space) for a Lie group $G$ is a homogeneous space for $G$ in which the stabilizer subgroup of every point is trivial. A principal homogeneous space for a group $G$ is a non-empty set on which $G$ acts freely and transitively. 

\, 
\begin{ex}
Prove that when $\cM$ is equipped with a Riemannian metric the  structure group reduces to the orthogonal group 
$O(n)$.  
\end{ex}

\begin{ex}
Prove that when $\cM$ is a $2n$-dimensional symplectic manifold then it has a natural 
$Sp(2n,\bbR)$-structure.  
\end{ex}

We remark the following fact, relating the manifold structure and the frame bundle. 

A natural manifold structure can be given if we  notice that the open sets of the typical fiber  $GL(n,\mathbb{R})$ are in bijection with the open sets of $\bbR^{n^{2}}$, and that the structural group of diffeomorphisms  of the typical fiber onto itself is simply $GL(n,\bbR)$.

\section{Vector fields, Tangent bundles, Lie algebras}
Given a differentiable manifold 
$\cM$, we consider at each point ${\scriptstyle M}$ of $\cM$ its corresponding tangent space $\cT_{\scriptstyle M}(\cM)$. A vector field 
$X$  assigns to every  point in ${\scriptstyle M}\in \cM$ a tangent vector $X_{\scriptstyle M}\in \cT_\cM$ in a smooth manner. This means that if $f$ is a smooth function on $
\cM$, the map 
$p\mapsto X_{\scriptstyle M}(f)$ must also be smooth.
 
In more concrete words, given a subset of the Euclidean space $\bbR^n$, a vector field is represented by a vector-valued function. 

\, 

More fundamentally, a smooth vector field $X$ is a linear map  $X:C^{\infty}(\cM)\to C^{\infty}(\cM)$ where $C^{\infty}(\cM)$ is the algebra of smooth functions and where $X$ satisfies the Leibnitz rule. This property reveals the deep algebraic nature of vector fields: they form a Lie algebra under the Lie bracket, defined as
\[[X,Y](f)=X(Y(f))-Y(X(f)),\] with $f\in C^{\infty}(\cM)$ and $X,Y$ are smooth vector fields on $\cM$.

\, 

To make a link between the previous section on tangent bundle, let us highlight that a vector field on $\cM$ can be defined in terms of a section of the tangent bundle. This is stated in a more precise way in the following definition.

\,

\begin{definition}[Vector field]\index{Vector field}
A vector field $X$ on a manifold $\cM$ is a cross-section of the tangent bundle
$\cT(\cM)$.
\end{definition}
Using this approach, we can say that a vector field $X$ associates to each ${\scriptstyle M} \in \cM$ a tangent vector $ X_{\scriptstyle M} \in T_{\scriptstyle M}(\cM)$ by the mapping $X: {\scriptstyle M} \mapsto ({\scriptstyle M}, X_{\scriptstyle M})$.

A vector field $X \in \cT(\cM)$ act on differentiable function $f$ on $\cM$ by
\begin{equation}
(Xf)({\scriptstyle M}) =X_{\scriptstyle M} (f), 
 \end{equation}
or in local coordinates,
\begin{equation}~\label{E:tgvecfield}
X=\sum_{i=1}^n X^i\,\frac{\partial}{\partial x^i},  \quad X^{i}=X(\varphi^{i})
\end{equation}
where $X^i$ are functions such that in a chart $(\cU,\varphi)$ in the neighborhood of $M$ by $ X_{\scriptstyle M}^{i}=(X(\varphi))({\scriptstyle M})$, is called the component of
$X$ with respect to the local coordinates $x^i=\varphi^{i}({\scriptstyle M})$. 

\begin{definition}[$\cC^{r}$-vector field]
By definition, a vector field on a $\cC^{k}$-manifold $\cM$ is  $\cC^{r}$-differentiable if the mapping $\cM \to T(\cM)$ is $\cC^{r}$-differentiable, where $r\leq k-1.$
\end{definition}

Going back to out first definition, we can refine it for $\cC^{k}$-differentiable functions of a vector field $X$ on a $\cC^{k}$-manifold $\cM$, can be understood as a derivation on the algebra $\cC^{k}(\cM)$ of function of class $\cC^{k}$ on $\cM$:
\[X: \cC^{k}(\cM)\to \cC^{k}(\cM),\]
\[X(f)({\scriptstyle M})=(Xf)({\scriptstyle M}) \text{ denoted } X(f)=Xf. \]

\subsection{Vector field Lie algebra}\index{Lie algebra}
Let us introduce the space $\fT(\cM)$ of all $\cC^{\infty}$ vectors fields on $\cM$, that is the space of all  $ X(f)$ which are differentiable for all
$\cC^{\infty}$ functions $f$ on $\cM$ . 
Under the addition and the multiplication operation:
\[(X+Y)(f)=X(f)+Y(f),\quad X,Y \in \fT(\cM),\, f \in \cC^{\infty}(\cM), \] 
\[gX(f)=g\big(X(f)\big),  X \in \fT(\cM),\, f,g  \in \cC^{\infty}(\cM) \]
this forms a module on the ring \footnote{A ring is a set $\fR$ with two internal  laws, addition and multiplication. For the addition $\fR$ is an abelian group, the multiplication is associative and distributive with respect to addition.
\[ (xy)z=x(yz),\quad x(y+z)=xy+xz,\quad (y+z)x=yx+zx,\quad \forall x,y,z\in\fT.\]

A module $\fT$ over the ring $\fR$ is an abelian group together with an external operation, called scalar multiplication such that
\[ \alpha(x+y)=\alpha x+\alpha y, \quad (\alpha+\beta)x=\alpha x+\beta x,\quad  (\alpha\beta)x= \alpha (\beta x ),  \quad x,y\in\fT,\,\alpha,\beta \in\fR.
\] } $C^{\infty}(\cM)$, but cannot be an algebra for the product of vector fields. 

The product $XY$ of two vector fields defined by $(XY)(f)=X(Y(f))$ does not satisfy the Leibnitz rule. Indeed, 
\[\begin{aligned}
(XY)(fg)&=X(Y(fg))=X\big(fY(g)+gY(f)\big)\\&=X(f)Y(g)+ fX(Y(g))+X(g)Y(f)+gX(Y(f))\\&\not=f(XY(g))+g(XY(f))=g(X(Y(f))+f(X(Y(g)).
\end{aligned}\]
Therefore $XY$ does not form a vector field.

However, notice that the Lie bracket \index{Lie algebra!Lie bracket} $[\cdot,\cdot]$ defined by
   \begin{equation}
[X,Y]f= X(Y(f))-Y(X(f)),\quad X,Y\in \fT,
 \end{equation}
does form a vector field. The multiplication, defined by the lie bracket is :

\begin{itemize}
\item distributive with respect to the addition,
\item anti-commutative,
\item not associative but satisfies the Jacobi identity: \index{Lie algebra!Jacobi's identity}
 \begin{equation}
[[X,Y],Z]+[[Y,Z],X]+[[Z,X],Y]=0,\quad X,Y,Z\in \fT,
 \end{equation}
 \end{itemize}
 
 \begin{lemma}
  The set $\fT(\cM)$ is a Lie algebra~\footnote{ A Lie algebra is a module  with bracket as internal product.} for the Lie bracket.
\end{lemma} 
\begin{proof}
The proof is mostly described above. \end{proof}

 \begin{definition}[Moving Frame]\label{D:Movfram}\index{Bundle!moving frame}
 A set of $n$ linearly independent differentiable vector fields $\{e_{i}\}_{i=1}^{n }$ which form a basis of the module $ \fT(U)$ where $U\subset \cM$ is called a moving frame.
 \end{definition}
 
Notice that a moving frame may not exist {\it globally} on $T(\cM)$.

To end this section, let us mention that vector fields can be used for instance in Index Theory. The Poincaré--Hopf theorem relates zeros of vector fields on $\cM$ to the Euler characteristic of $\cM$. 

\section{Cotangent bundle - Covector field}
The topic considered in this section relates to cotangent bundles. A cotangent bundle is the natural dual to the tangent bundle, encoding the differential structures of a manifold in their most intrinsic form. While the tangent bundle $\cT(\cM)$ describes directions of motion, the cotangent bundle 
$\cT^*(\cM)$  is the realm of differentials. 

For a differentiable manifold $\cM$  the cotangent space at each point consists of linear functionals on the tangent space. That is, if $X_{\scriptstyle M}$ is a tangent vector, an element of $\cT_{\scriptstyle M}^*(\cM)$ assigns to it a real number by evaluating a differential:
\[df_{\scriptstyle M}(X_{\scriptstyle M}).\]
\subsection{Cotangent bundle}

\begin{definition}[Cotangent bundle]\index{Bundle!Cotangent bundle}
The bundle space $(\cT^\star(\cM),\cM,\pi)$ where $\cT^\star(\cM)$ is the space of pairs
$({\scriptstyle M},\omega_{\scriptstyle M})$ for all ${\scriptstyle M}\in \cM$ and all $\omega_{\scriptstyle M}\in \cT^\star_{\scriptstyle M}(\cM)$ is called
cotangent bundle space. 
\end{definition}

In the case of a n-dimensional manifold,
the cotangent-bundle  $(\cT^{\star}(\cM), \cM, \pi,G)$ can be identified to the natural bundle
\begin{itemize}
	\item  fiber at ${\scriptstyle M} \in \cM$, $\cF_{\scriptstyle M}= \cT^{\star}_{\scriptstyle M}(\cM)$ and typical
fiber $\cF =\bbR^n$,  

	\item projection $\pi : ({\scriptstyle M},\omega_{\scriptstyle M}) \mapsto {\scriptstyle M} $, 
		
	\item The structural group $G=GL(n,\bbR)$,   of linear automorphisms on
$\bbR^n$.
\end{itemize}

\subsection{Covector field}
\begin{definition}[Covector field]\index{Covector field}
 A $\cC^k$-differentiable 1-form $\omega$ is a $\cC^k$-cross section of the cotangent bundle. It is often called a covariant vector field.
\end{definition}

A covariant vector field $\omega$ associates to each $ {\scriptstyle M}\in \cM$ a covariant vector $\omega_{\scriptstyle M} \in X\in \cT^{\star}_{\scriptstyle M}(\cM)$ by the mapping \[\omega: {\scriptstyle M}\mapsto ({\scriptstyle M},\omega_{\scriptstyle M}).\]

The covector field $\omega \in  \cT^{\star}(\cM)$ acts on the vector field $X\in  T(\cM)$
by
\begin{equation}
\omega(X)({\scriptstyle M})=\omega_{\scriptstyle M}(X_{\scriptstyle M}).
\end{equation}

If we denote by $ \{dx^{i}\}_{i=1}^{n}$ the dual basis of the  natural basis $\{\partial_{i}=\partial/\partial x^{i}\}_{i=1}^{n}$,
\[dx^{i}\big(\partial_{j}\big) = \partial_{j}\big(dx^{i}\big)=\delta^{i}_{j},\]
and
\[
\omega_{\scriptstyle M}(X_{\scriptstyle M})={\omega_{\scriptstyle M}}_{i} X_{\scriptstyle M}^{i},\quad \omega_{\scriptstyle M}={\omega_{\scriptstyle M} }_{i}\,dx^{i}|_{\scriptstyle M},\quad  X_{\scriptstyle M} =X_{\scriptstyle M}^{i}\, \partial_{i}|_{\scriptstyle M}, \]
where the index ${\scriptstyle M}\in \cM$. 
\, 

From the equation~\eqref{E:difcoord}, we can define the differential 1-form field by:
\begin{equation}
(df(X))({\scriptstyle M})=df|_{\scriptstyle M}(X_{\scriptstyle M}).
\end{equation}
In the natural basis 
\[df=\frac{\partial f}{\partial x^{i}}\,dx^{i},\]
and in an arbitrary local basis  $\{e_{i}\}_{i=1}^{n}$ in $\cT_{\scriptstyle M}(\cM)$, where $\{\varepsilon_{i}\}_{i=1}^{n}$ is the dual basis, it implies that we have: 
\[(df(X))({\scriptstyle M})= X^{i}_{\scriptstyle M} e_{i}|_{\scriptstyle M} (f)= e_{i}|_{\scriptstyle M}(f)\varepsilon^{i}|_{\scriptstyle M}(X_{\scriptstyle M}),\]
and finally
\begin{equation} \label{E:diffieldcoord}
df= e_{i}(f)\varepsilon^{i}.
\end{equation}

\section{Action of a diffeomorphism}\index{Diffeomorphism}

\subsection{Image of a vector field under a diffeomorphism}\index{Diffeomorphism! Image of a vector field}
Assume $\cM$ and $\cM'$ are two manifolds. Let: \[f:\cM\to \cM'\] be a differentiable mapping between  $\cM$ and $\cM'$, such that ${\scriptstyle M}\in \cM$ is mapped to ${\scriptstyle M'}=f({\scriptstyle M})\in \cM'$. 

The map $f$ induces a linear (Jacobian) mapping~\footnote {the Jacobian mapping is sometimes denoted $f'$.} denoted $f_{\star}$, between the tangent space  $T_{\scriptstyle M}(\cM)$ at ${\scriptstyle M}$ and the tangent space $T_{\scriptstyle M'}(\cM')$ at ${\scriptstyle M'}=f({\scriptstyle M})$, 
 \[
  f_{\star} : X_{\scriptstyle M}\mapsto f_{\star}(X_{\scriptstyle M})=X'_{f({\scriptstyle M})},
  \]
   defined in the following way.
   
 Let $g$ a differentiable function in the neighborhood of $f({\scriptstyle M})$. Then, we obtain:  
 
 \begin{equation}\label{E:fstar}
[(f_{\star}(X)(g)](f({\scriptstyle M}))= [X(g\circ f)]({\scriptstyle M}) =X_{\scriptstyle M} (g\circ f)
\end{equation}
\begin{theorem}
Let $f:\cM \to \cM'$ be a $C^\infty$ diffeomorphism between two $n$-dimensional differentiable manifolds. Then $f_{\star}$ is an isomorphism of the Lie algebra: $$ \fT(\cM)\to\fT(\cM),$$ given by 
\begin{equation}
f_{\star}([X,Y])=[f_{\star}(X),f_{\star}(Y)].
\end{equation}
\end{theorem}

\begin{proof}
Exercise!
\end{proof}

\begin{definition}[Invariant vector field]\index{Vector field!Invariant vector field}
A vector field $X$ on $\cM$ is said to be invariant under the diffeomorphism \[f:\cM \to\cM\] if
\begin{equation}\label{E:invecfield}
f_{\star}(X_{\scriptstyle M})=X_{f({\scriptstyle M})},\quad \forall\, {\scriptstyle M}\in \cM.
\end{equation}
\end{definition}

\subsection{Image of a cotangent vector field under a diffeomorphism}\index{Diffeomorphism!Image of a covector field}
Let $\omega_x\in  T^{\star}_x(\cM)$ and $\omega'_{\scriptstyle M'}$ be such that
\[
\omega_{\scriptstyle M} (X_{\scriptstyle M})= \omega'_{\scriptstyle M'}(X'_{\scriptstyle M'}),
\]
The pull-back (reciprocal image) $f^{\star}: \cT^{\star}_{\scriptstyle M'}(\cM') \to \cT^{\star}_{\scriptstyle M} (\cM)$ of a covariant vector $\omega'_{\scriptstyle M'}$ under the differentiable mapping $f$ is defined by the equality:

\begin{equation}
\left(f^{\star}(\omega'_{\scriptstyle M})\right)_{\scriptstyle M}(X_{\scriptstyle M})=\omega'_{\scriptstyle M'}\left(f_{\star}(X)\right |_{\scriptstyle M'}.
\end{equation}

Therefore, the pull-back\index{Diffeomorphism!Pull-back} $f^\star:  T^{\star}(\cM') \to T^{\star}(\cM)$ of the 1-form $\omega'$ under a differentiable mapping $f$ is defined by
\begin{equation}
\left(f^{\star}(\omega')\right)(X)=\omega'\left(f_{\star}(X)\right)\circ f.
\end{equation}

\begin{remark}
The expression of the reciprocal image of a 1-form does not involve $f^{-1}$, whereas it is the case for a vector field.
\end{remark}

\subsection{Tensor field} We generalise the discussion in the previous chapter about tensors. This notion also can evolve using the concept of fiber bundles. 

\, 

In particular, a fiber bundle where the base space is the manifold $\cM$ and the fiber is identified with $\otimes^{p}T^{\star}_{\scriptstyle M}(\cM)\otimes^{q}T_{\scriptstyle M}(\cM)$  at all ${\scriptstyle M}\in \cM$ is called a $(p,q)$-\emph{tensor bundle}.

Notice that a tangent bundle is a $(0,1)$-tensor bundle and that the cotangent bundle is a $(1,0)$-tensor bundle.

\begin{definition}
A $(p,q)$-tensor field on a $C^{k}$-manifold is a $C^{r}$ cross-section, where $r\leq k-1$ of the $(p,q)$-tensor bundle.
\end{definition}

Operations defined in section~\ref{S:tensor} for tensors at a point, are carried over fiber-wise allowing to define similar operations on tensor fields. 

Under these operations, the set of all  $(p,q)$-tensor fields of class  $\cC^{r}$ is a \emph{ module on the ring} $\cC^{r}(\cM)$.

\chapter{Connections, Parallel Transport and Sheafs} 

In differential geometry, connections and parallel transport arise from a fundamental necessity: the ability to compare vectors at different points on a curved space, where a naive translation is no longer well-defined.

\begin{minipage}{0.5\textwidth} 
  \centering
  \includegraphics[]{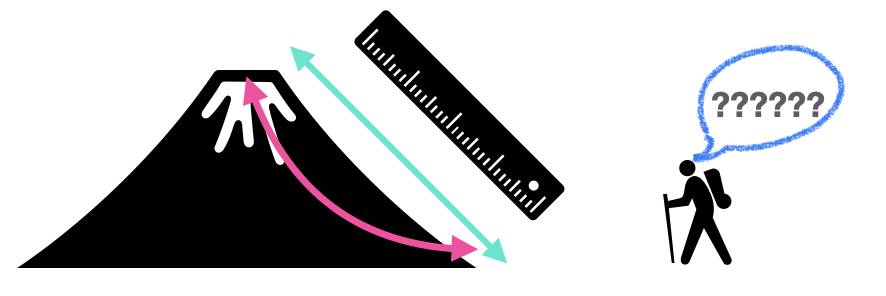}

\end{minipage}

In Euclidean space, one can simply
compare vectors via a parallel transport--shifting,
without change in magnitude or direction.

However, on a curved manifold, the very act of moving a vector must be prescribed by an additional structure, which is provided by a connection.

A connection defines a rule for differentiation that respects the manifold’s geometry, replacing the notion of partial derivative in a given direction with a well-defined covariant derivative. This operator tells us how a vector field changes as we move along the manifold, encoding information about curvature and torsion.

But the true power of connections emerges in the concept of parallel transport. Given a connection, we can transport a vector along a curve, according to the manifold’s structure. This operation is at the heart of many geometric investigations.

In Riemannian geometry, parallel transport reveals the curvature of space, since moving a vector along different paths can lead to different final results. In fact, given a loop $\gamma:[a,b]\to \cM$ where $[a,b]\subset \bbR$ on a smooth manifold $\cM$, where $\gamma(a)=\gamma(b)$ we can transport a vector $\Vec{v}$ along that curve,  starting at the initial point $\gamma(a)$. 

The question that emerges is whether there has been any modification implied on the vector $\Vec{v}$, during this transport, namely does the vector at $\gamma(b)$ coincide with the vector $\Vec{v}$ at $\gamma(a)$ or not. If those vectors coincide then the manifold on is flat on the set where the loop was defined.

In gauge theory, parallel transport describes how fields interact, with connections playing the role of gauge potentials in Yang-Mills theory.

In general relativity, the Levi--Civita connection governs the motion of free-falling observers in curved spacetime.

Thus, a connection is not merely a technical tool but an essential ingredient of geometry, allowing us to define differentiation, curvature, and transport in a way that transcends local coordinates. It is the bridge between the infinitesimal and the global, between algebra and topology, between abstract geometry and the laws of nature.

\section{Linear connection}\index{Connection!Linear}

In Euclidean space,  two vectors of different origin are compared by a parallel translation of the vectors to the one same origin. The derivative of a vector $\Vec{v}$ defined along a curve, is a vector of components $(\partial v^{i}/\partial x^{j}) dx^{j}/dt$ in the Cartesian coordinates. On an arbitrary differentiable manifold, the components  $$\partial v^{i}/\partial x^{j}$$ do not behave as the components of a tensor under a change of coordinates. To solve this difficulty we introduce the notion of  covariant derivative which is of tensorial type.

\begin{definition}[Linear connection]\label{lincon}
A linear connection on a differentiable manifold $\cM$ is a mapping $\nabla : X\mapsto \nabla X$ from the vector fields on $\cM$ into the differentiable tensor field of type $(1,1)$ on $\cM$ such that
\begin{equation}\begin{aligned}
\nabla(X+Y)&=\nabla(X)+\nabla(Y),\\
\nabla(fX)&=df\otimes X+f\nabla(X), 
\end{aligned}\end{equation}
where $f$ is a differentiable function~\footnote{To give a more precise definition we have to speak  of germ of vector (or tensor) field and germ of function on $\cM$. } on $\cM$.

\vspace{3pt}
The tensor $\nabla X $ is called the \emph{covariant derivative}  of $X$.
\end{definition}

For a $n$-dimensional differentiable manifold $\cM$, equipped with moving a frames (Defintion \ref{D:Movfram}) $\{e_{i}\}_{i=1}^{n},\ e_{i}\in T(\cM)$  and 
$\{\varepsilon^{{i}}\}_{i=1}^{n},\ \varepsilon^{i}\in T^{\star}(\cM)$ (its dual), 
we have that $\nabla e_{i}$ is a $(1,1)$-tensor field which can be written as: 

\begin{equation} 
\nabla e_{i}(e_{k},\varepsilon^{j})=\gamma^{j}_{ki}, \end{equation}
and the $(1,1)$-tensor $\nabla e_{i}$ can be expressed as:
\begin{equation}\label{E:concoef0}
\nabla e_{i}=\gamma^{j}_{ki} \varepsilon^{k}\otimes e_{j},
\end{equation}

\begin{definition}[Connection coefficients]\index{Connection!Coefficients}
The coefficients $\gamma^{j}_{ki}=\nabla e_{i}(e_{k},\varepsilon^{j})$ are called the connection coefficients in the basis $\{e_{i}\}_{i=1}^{n}$.
\end{definition}

It follows from the definition~\ref{lincon} and the equation~\eqref{difffunc} that 
\[\begin{aligned}
\nabla X &= \nabla(X^{i}e_{i})=dX^{i}\otimes e_{i}+X^{i}\nabla(e_{i})\\
&=\big(dX^{i} +\gamma^{i}_{kj} X^{j}\varepsilon^{k}\big) \otimes e_{i}\\
&=\big(e_{k}(X^{i}) +\gamma^{i}_{kj}X^{j} \big) \varepsilon^{k}\otimes e_{i}\\
&=\nabla_{k}X^{i}\varepsilon^{k}\otimes e_{i},
\end{aligned}\]
where  $\nabla_{k}X^{i}=\big (\nabla X \big)^{j}_{k}$ denote the component  of the tensor $\nabla X$.
\begin{equation}\label{concoef1}
\nabla X = \nabla_{k}X^{i}\varepsilon^{k}\otimes e_{i},\quad \text{ with } \quad\nabla_{k}X^{i}=e_{k}(X^{i}) +\gamma^{j}_{ki}X^{i}.
\end{equation}

Furthermore, one can rewrite the $(1,1)$-tensor $\nabla_X$ in the following way:
\begin{equation}\label{concoef2}
\nabla X 
=\big(dX^{i} +\omega^{i}_{j} X^{j}\big) \otimes e_{i},\quad  \omega^{i}_{j}=\gamma^{i}_{kj} \varepsilon^{k}
\end{equation}
making the covariant part explicit.
\begin{definition}[Connection forms]\index{Connection!Connection forms}
The quantities defined by 
\begin{equation}\label{E:conform}
\omega^{i}_{j}=\gamma^{i}_{kj} \varepsilon^{k} .
\end{equation}
are called the connection 1-forms.
\end{definition}
We have the relation
\begin{equation}
\omega^{i}_{j}(e_{k})=\gamma^{i}_{kj} =
\nabla e_{i}(e_{k},\varepsilon^{j}).
\end{equation}
As $\nabla X$ is a $(1,1)$-tensor it is invariant under the change of natural coordinate system. In particular:
\[e_{i}= A_{i}^{j}e'_{j},\quad {\varepsilon}^{k}=  A_{\ell}^{k} {\varepsilon'}^{\ell},\quad  A^{\ell}_{i}A_{\ell}^{j}=\delta _{i}^{j},\]
and \[ X^{i}=A^{i}_{j}{X'}^{j}\quad  \text{ implies }\quad   dX^{i}=A^{i}_{j}d{X'}^{j}+{X'}^{j}dA^{i}_{j}.\]
It follows from \eqref{concoef2} that
\[\nabla X=\big(dX^{i} +\gamma^{i}_{kj} X^{j}\epsilon^{k}\big) \otimes e_{i}=\big(d{X'}^{j} +{\gamma'}^{j}_{k\ell} {X'}^{\ell}{\varepsilon'}^{k}\big) \otimes e'_{j}.\]
Moreover,
\[\begin{aligned}
\nabla X&=\big(dX^{i} +\gamma^{i}_{kj} X^{j}\varepsilon^{k}\big) \otimes e_{j}=\big(dX^{i} +\gamma^{i}_{kj} X^{j}\varepsilon^{k}\big) \otimes A^{k}_{i}e'_{k}\\
&=\big(A^{i}_{j}d{X'}^{j}+ {X'}^{j}dA_{j}^{i}+A^{j}_{\ell}{X'}^{\ell}{\gamma}^{i}_{kj}A^{k}_{m} {\varepsilon'}^{m}\big) \otimes A^{k}_{i}e'_{k}\\
&=\big(d{X'}^{j}+ A^{j}_{k}{X'}^{\ell}dA_{\ell}^{k}+A^{m}_{\ell}{X'}^{\ell}{\gamma}^{i}_{km}A^{j}_{i}A^{k}_{n} {\varepsilon'}^{n}\big) \otimes e'_{j}\\
&=\big(d{X'}^{j}+ (A^{j}_{k}dA_{\ell}^{k}+A^{m}_{\ell}{\gamma}^{i}_{nm}
A^{j}_{i}A^{n}_{k} ){X'}^{\ell}{\varepsilon'}^{k}\big) \otimes e'_{j}\\
\end{aligned}\]
and therefore
\[
{\gamma'}^{j}_{h\ell}=A^{j}_{k}dA_{\ell}^{k}+A^{m}_{\ell}{\gamma}^{i}_{nm}
A^{j}_{i}A^{n}_{h}
\]
which, by equation~\eqref{E:difcoord}, gives us
\begin{equation}\label{E: tconform}
{\gamma'}^{j}_{h\ell}=A^{j}_{k}e'_{h}(A_{\ell}^{k})+A^{n}_{h}A^{m}_{\ell}A^{j}_{i}{\gamma}^{i}_{nm}.
\end{equation}
The connection coefficients are not components of a tensor because of the first term on the right hand side of the equation \eqref{E: tconform}.
\vspace{5pt}

In the  local coordinate system $\{x^{1},\dots,x^{n}\}$, we have:
\begin{equation}\label{E:coeffconloccoor}
\nabla X= \left( \frac{\partial X^{i} }{\partial x^{k}}+\Gamma^{i}_{k\ell}X^{\ell} \right) d x^{k} \otimes \frac{\partial}{\partial x^{i}},
\end{equation}
where  $\Gamma^{i}_{k\ell}X^{\ell}$ is the connection coefficients \eqref{E:concoef0} in the local (natural) coordinate system $\{x^{1},\dots,x^{n}\}$. For the vector field $ \partial_i=\frac{\partial}{\partial x^i}$ the components are all $0$ except for $X^i=1$, and we can express: 
\[
\nabla(\partial_i)= \Gamma^{j}_{ki} dx^k\otimes \partial_j
\]

\begin{definition}[Christoffel symbols]\index{Connection!Christoffel symbols}
The $\Gamma^{j}_{ki}$ are called the Christoffel symbols.
\begin{equation}
\Gamma^{i}_{kj} =
\nabla(\partial_{i})(\partial_{k},dx^{j})=\omega^{j}_{i}(\partial_{k}).
\end{equation}
\end{definition}

\,   

Now, the Christoffel symbols transform under a change of local coordinates system: $\{x^{1},\dots,x^{n}\}\to \{{x'}^{1},\dots,{x'}^{n}\}$ as
\begin{equation}
{\Gamma'}^{j}_{h\ell}= \frac{\partial{x'}^{j}}{\partial x^{i}}  \frac{\partial}{\partial x^{k}} \left(  \frac{\partial{x'}^{i}}{\partial x^{\ell}}\right) +     \frac{\partial x^{m}}{\partial {x'}^{\ell}} \frac{\partial x^{k}}{\partial{x'}^{h}} \frac{\partial{x'}^{j}}{\partial x^{i}} \Gamma^{i}_{k m}.
\end{equation}
Hence, the Christoffel symbols does not define a tensor.

\section{Covariant derivative  in the direction $Y$}\index{Covariant derivative}
Let $X$ and $Y$ be two vector fields. By~\eqref{E:coeffconloccoor}, we get $\nabla X(Y)=Y^{k}\nabla_{k}X^{i}e_{i}=Y^{k} \nabla_{k}X$ which can be identified to the  derivative~\eqref{E:dirder} in the $Y$-direction  

\subsection{\bf Covariant derivative of a vector field in the direction of $Y$}
Let $X$ and $Y$ be two vector fields. By~\eqref{E:coeffconloccoor}, they can be expressed as $$\nabla X(Y)=Y^{k}\nabla_{k}X^{i}e_{i}=Y^{k} \nabla_{k}X,$$ which can be identified to the  derivative~\eqref{E:dirder} in the direction of $Y$.  This leads us to the formulation of a definition of the covariant derivative. 

\begin{definition} [Covariant derivative in the direction $Y$] 
The covariant derivative $\nabla_{_{Y}}X$ of $X$ in the direction $Y$ is defined by
\begin{equation}
\nabla_{_{Y}}X=(\nabla(X))(Y)    \text{ or } \nabla_{_{Y}}X(\cdot)=(\nabla(X))(Y,\cdot) 
\end{equation}
\end{definition}

The covariant derivative $\nabla_{_{Y}}X$ is  \emph{linear} in $Y$
\[
\nabla_{{fY_{1}+gY_{2}}}X=(f\nabla_{_{Y_{1}}}+g\nabla_{_{Y_{2}}})X, \] where \[Y_{i},Y_{2}\in T_{_{M}}(\cM),\quad \text{and}\quad f,g : \cM\to \bbR.
\]

In the frame $\{e_{i}\}_{i=1}^{n}$, we obtain: 
\[
\nabla_{_{Y}}X=Y^{k}\left(e_{k}(X^{i}) + \gamma^{i}_{kj}X^{j}\right)e_{i}=\left(Y(X^{i}) + \gamma^{i}_{kj}Y^{k}X^{j}\right)e_{i},
\]
and thus 
\begin{equation}
\nabla_{e_{i}}e_{k}=  \gamma^{j}_{ik}e_{j},
\end{equation}
where $\nabla_{e_{i}}$ denote the covariant derivative  in the direction $e_i$.

\subsection{\bf The Covariant derivative of a tensor}

\, 

The answer to this question is yes. The notion of a covariant derivative in the  direction of $X\in \cT_{_{M}}(\cM)$ can be extended---in a neightborhood of $M\in \cM$--- to an arbitrary type of tensor field, under the following assumptions: 
\begin{subequations}\label{E:cdt}
\begin{gather}
 \nabla_{X}f=X(f),\quad f\in C^{1}( \cM),\label{E:cdt1}\\
 \nabla_{X}(t+ s)= \nabla_{X}t +\nabla_{X}s,\label{E:cdt2}\\
\nabla_{X}(t\otimes s)= (\nabla_{X}t)\otimes s+ t\otimes (\nabla_{X}s),\label{E:cdt3}\\
\nabla_{X} \text{  commutes with the contracted multiplication.}\label{E:cdt4}
\end{gather}
\end{subequations}

Therefore, if $t$ is a $(p,q)$-tensor its covariant derivative forms a $(q+1,p)$-tensor:
\[(\nabla t)(u,v_{1},\dots,v_{p},\omega_{1},\dots,\omega_{p})=(\nabla_{u} t)(v_{1},\dots,v_{p},\omega_{1},\dots,\omega_{p}),\] where \[ u,v_{i}\in T_{_{M}}(\cM),\quad \text{and}\quad \omega_{j}\in T^{\star}_{_{M}}(\cM).
\]

\vspace{3pt}
\noindent{\bf  (a) Covariant derivative of a 1-form}

\vspace{3pt}
We will study the above definition on the example of a 1-form. By the fourth condition above, we have that 
\[
\nabla_{_{Y}}[\mathbf{\alpha}(X)] = (\nabla_{_{Y}}\mathbf{\alpha})(X)+\mathbf{\alpha}(\nabla_{_{Y}}X),\]
which implies that 
\[(\nabla_{_{Y}} \mathbf{\alpha})(X) =\nabla_{_{Y}}[\mathbf{\alpha}(X)]- \mathbf{\alpha}(\nabla_{_{Y}}X).
\]
If we substitute $X$ by $e_{i}$ this gives
\[
(\nabla_{_{Y}} \mathbf{\alpha})_{i}=Y(\alpha^{i}) - \mathbf{\alpha}(\gamma^{j}_{ki}Y^{k}e_{j})=Y(\alpha^{i})-\gamma^{j}_{ki}Y^{k}\alpha_{j},
\]
and thus
\[
\nabla_{_{Y}} \mathbf{\alpha}=Y^{k}[e_{k}(\alpha_{i})-\gamma^{j}_{ki}\alpha_{j}]\varepsilon^{i},
\]
which means in particular that 
\begin{equation}
\nabla_{_{Y}} \varepsilon^{ i} = -Y^{k}\gamma^{j}_{ki}\varepsilon^{j} \text{ or } \nabla_{e_{k}} \varepsilon^{ i} = -\gamma^{j}_{ki}\varepsilon^{j}.
\end{equation}
This allows us to deduce that
\[
\nabla \mathbf{\alpha}= [e_{k}(\alpha_{i}) -\gamma^{j}_{ki}\alpha_{j}]\epsilon^{k}\otimes \varepsilon^{j}=(d\alpha_{i}+ \alpha_{j}\omega^{j}_{i})\otimes \varepsilon^{i}.
\]
\begin{proposition}[Covariant derivative of 1-form]
Let $\mathbf{\alpha}\in T^{\star}_{_{M}}(\cM)$. Then
 \begin{equation}
(\nabla_{_{Y}} \mathbf{\alpha})(X) =\nabla_{_{Y}}(\mathbf{\alpha}X)- \mathbf{\alpha}(\nabla_{_{Y}}X).
\end{equation}
Additionally, if $\{e_{i}\}_{i=1}^{n}$ is a moving frame and $\{\varepsilon^{i}\}_{i=1}^{n}$ then 
\begin{equation}
\nabla \mathbf{\alpha} =(d\alpha_{i}+ \alpha_{j}\omega^{j}_{i})\otimes \varepsilon^{i}
\end{equation}
\end{proposition}

\begin{proof}
    The proof is left as an exercise. 
\end{proof}

\noindent{\bf  (b) Covariant derivative of a tensor}
We now consider a (2,0)-tensor and explore the notion of covariant derivative of a (2,0)-tensor. Let $g=g_{\mu\nu}\varepsilon^\mu\otimes \varepsilon^{\nu}$ be a $(2,0)$-tensor.

\,

The covariant derivative of $g$ is obtained from the properties (3) and (4) outlined in the list of conditions above. 

Therefore, 
\[\begin{aligned}
\nabla_{_{X}}g&= X(g_{\mu\nu})\epsilon^\mu\otimes \epsilon^{\nu}+g_{\mu\nu}\nabla_{_{X}}\varepsilon^\mu\otimes \epsilon^{\nu}+g_{\mu\nu}\varepsilon^\mu\otimes \nabla_{_{X}}\varepsilon^{\nu}\\
&=X^{\alpha}[e_{\alpha}(g_{\mu\nu})-\gamma^{\beta}_{\alpha\mu} g_{\beta\nu}-\gamma^{\beta}_{\alpha\nu} g_{\mu\beta}]\varepsilon^{\mu}\otimes \varepsilon^{\nu},
\end{aligned}\]
or
\begin{equation}\label{covmetric}
\nabla_\alpha g_{\mu\nu}=e_{\alpha}(g_{\mu\nu})-\gamma^{\beta}_{\alpha\mu} g_{\beta\nu}-\gamma^{\beta}_{\alpha\nu} g_{\mu\beta}.
\end{equation}

\begin{ex}
The generalization is rather straight forward and left as an exercise. 
    
\end{ex}

\begin{ex}
As an example, one can work with a $(2,1)$-tensor $t_{kl}^i$ and show that : 
\[
\nabla_j t^{i}_{k\ell}=e_{j}(t^{i}_{k\ell})+\gamma^{i}_{jm}t^{m}_{k\ell} -\gamma^{m}_{jk} t^{i}_{m\ell}-\gamma^{m}_{j\ell} t^{i}_{km}.
\]
    
\end{ex}

We continue with the next remark. 
\begin{remark}
 The covariant derivative in the direction $e_{i}$ of the tensor product of  $s$ and $t$ of two tensors satisfy, by \eqref{E:cdt3}:
 \[\nabla_{e_{i}}(s\otimes t)=(\nabla_{e_{i}}s)\otimes t+s\otimes ( \nabla_{e_{i}}t).\]
However, it is important to keep in mind that the sum of two tensor products is defined only if the corresponding factors have the same rank and therefore:
\[
 \nabla (s\otimes t)\not=(\nabla s)\otimes t+s\otimes ( \nabla t).
\] 
\end{remark}

\section{Parallel transport - Geodesics}

In Euclidean space, vectors based at different points can be compared by translating them parallelly to a common origin. 

Consequently, if a vector 
$\Vec{v}$ is transported parallelly along a curve $\gamma$, the derivative $dv/dt=0$ vanishes.


On a manifold $\cM$, to compare vectors at two different points $M\in \cM$ and $M'\in \cM$, it is necessary to be able to assign a uniquely defined frame at  $M'$ based on a frame at $M$. 

This is precisely where the notion  of connection plays a key role: by enabling the existence of a covariant derivative, a connection provides a convenient notion of parallel transport.

\subsection{Parallel vector along a curve} 
\begin{definition}[Parallel transport]\index{Connection!Parallel transport}
A vector $X$ is said to be parallel along the curve $\gamma : t\to \gamma(t)$ if
 \begin{equation}
 \nabla_{u}X=0,\quad u=\gamma'\,\frac{d}{dt},
 \end{equation}
 \end{definition}
 
 \begin{remark}
 The vector $u$ is defined only at points lying on the curve $\gamma$. Somehow, it can also be extended to a vector field on a neighborhood of the curve $\gamma$. 
 \end{remark}
 
 If $(\varphi,U)$ is a local chart, the component of a point $\gamma(t)$ on the curve $\gamma$  are \[x^{i} (t)=\varphi^{i}\circ\gamma(t)\] and the components of $u$ are given by $u^{i}=dx^{i}/dt$
\[u^{i}= \varphi^{i}\circ u=u( \varphi^{i}).\]

 \subsection{Geodesics}
 We shortly digress on the notion of geodesics, as they intimately related to the topic discussed in this section. 
 
 \begin{definition} [Affine geodesic]\index{Geodesic!Affine}
 An affine geodesic on $\mathcal{M}$ is a curve \[\gamma : t\to \gamma(t)\]  such that
 \begin{equation}
 \nabla_{u}u=\lambda(t)u,\quad \text{where} \quad u=\frac{d\gamma}{dt},
 \end{equation}
 for some function $\lambda$ on $\bbR$.
 
The curve $\gamma$ is called a geodesic i.e.
 \begin{equation} \nabla_{u}u=0.\end{equation}
 \end{definition}

The concept of geodesics  generalizes the notion of straight line in Euclidean space.

There are some easy computations that can be done given $\gamma:I\to M, \quad I=[a,b]\subset \mathbb{R}$, the length of the smooth curve  is given by
\[L(\gamma):=\int_{a}^b|\dot \gamma(t)|dt.\]
In particular, this notions starts to be interesting as soon as we want to understant better the intrinsic distant between two points on a manifold. 
Assume $\cM\subset \bbR^n$ is an $m$-dim smooth manifold. Take a pair of distinct points $p,q\in \cM$. 
Their Euclidean distance is given by $|p-q|$ in the ambient Euclidean space.
However, this type of distance does not tell us much, from the viewpoint of the manifold $\cM$.
We therefore introduce the notion of an intrinsic distance in $\cM$.
\begin{definition}
The intrinsic distance on $\cM$ between $p, q\in \cM$ is a real number $d(p,q)\geq 0$ 
defined by \[d(p,q):=\inf_{\gamma\in \Omega_{p,q}}L(\gamma),\] 
where $\Omega_{p,q}$ is the space of smooth paths of the unit interval joining $p$ to $q$. 
\end{definition} 

\begin{ex}
Prove that $L(\gamma)\geq |p-q|$.
\end{ex}
\begin{center}
    \includegraphics[]{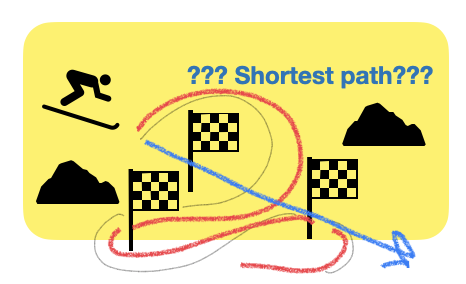}
    \end{center}
 Let us now write the equation of geodesics in local chart $(\varphi,U)$ the component of a point $\gamma(t)$ on the curve $\gamma$  are $x^{i} (t)=\varphi\circ\gamma(t)$ and the component of $u$ are $u^{i}=dx^{i}/dt$

\section{Curvature}
\subsection{Vector bundle - version}
Let $(\mathcal{M}, g)$ be a Riemannian manifold, that is a smooth manifold equipped with a Riemannian metric.
The Riemannian curvature tensor is defined as a map
\[R:\Gamma(\cT\cM) \times \Gamma(\cT\cM)  \times \Gamma(\cT\cM) \to\Gamma(\cT\cM), \] 
characterized by the formula:
\[R(X,Y)Z=\nabla_X\nabla_YZ-\nabla_Y\nabla_YZ-\nabla_{[X,Y]}Z\]
where $[X,Y]$ is a Lie bracket of vector fields and where $\Gamma(\cT\cM)$ are tangent bundle sections.

\begin{example}
    In the flat case, the absence of curvature ensures geodesic parallelism.
\end{example}
\subsection{Sheaf version}
In the scope of using a more modern language, we can consider the notion of Riemannian curvature in a more categorical framework, which relies on the notion of sheaf. To give a rough idea, a sheaf assigns some local data to open sets, and patches these local data together  into a global object.

\begin{definition}
    Let $X$ be a topological space. A presheaf $\cF$ on $X$ is a contravariant functor from the category of open sets of $X$ (denoted $Open(X)$), to a category $\mathscr{C}$, i.e.,
    \[\cF:Open(X)^{op}\to \mathscr{C}.\]
\end{definition}

This definition means that to each open set $U\subset X$, we assign an object $\cF(U$) in the category $\mathscr{C}$. The category $\mathscr{C}$ can be the category of sets, abelian groups, rings etc.

To each inclusion of open sets $V\subset U$, we assign a restriction morphism $\rho_{U,V}:\cF(U)\to \cF(V)$, satisfying: 

\begin{itemize}
    \item {\bf Identity:} $\rho_{U,U}$ is the identity;
    \item {\bf Transitivity:}  $\rho_{V,W}\circ\rho_{U,V}=\rho_{U,W}$ for $W\subset V\subset U$.
\end{itemize}

A sheaf is a presheaf with an extra gluing condition. Namely, for any open covering $U=\bigcup\limits_{\alpha} U_\alpha$ and a family of local sections that agree on overlaps (i.e. {\it compatible local sections}), there exists a unique global section that restricts to them.
In particular, $\cF$ is a sheaf if, for any open cover $U=\bigcup\limits_{\alpha} U_\alpha$, the sequence is exact: 

\[\cF(U)\to \prod_\alpha\cF(U_{\alpha}) \rightrightarrows \prod_{\alpha,\beta}\cF(U_{\alpha}\cap U_{\beta})\]

Let us breakdown this construction: 

\begin{itemize}
    \item The first map sends a global section $s\in \cF(U)$ to its restriction $s|{_{U_\alpha}}$.
    \item The second pair of maps send $(s_{\alpha})$ to their restrictions on overlaps $s_{\alpha}|_{{{U_\alpha}\cap U_{\beta}}}$ and $s_{\beta}|_{{{U_\alpha}\cap U_{\beta}}}$.
\end{itemize}

\subsection{Curvature version 2}
Going back to the definition of curvature, the main differences, between the previous language and the new formalism are depicted below:
\begin{itemize}
    \item the tangent bundle sections are viewed as a \emph{sheaf} $\cE$ of $\cO_\cM$-modules. 

  \item The covariant derivative is interpreted as a \emph{morphism of sheaves} rather than an operation on vector fields.

  \item The curvature tensor is introduced as a functorial transformation of the module structure.
  
  \item The bracket operation is associated with the intrinsic structure of the Lie algebroid given by the sheaf 
$\cE$.

\end{itemize}

Let $(\mathcal{M}, g)$ be a Riemannian manifold, where $\mathcal{M}$ is a smooth manifold equipped with a sheaf $\mathcal{O}_{\mathcal{M}}$ of smooth functions and a sheaf $\mathcal{E} = \mathcal{T}_{\mathcal{M}}$ of sections of the tangent bundle, endowed with a Riemannian metric $g$. The connection $\nabla$ is then a morphism of sheaves of $\mathcal{O}_{\mathcal{M}}$-modules:

\[
\nabla: \mathcal{E} \to \mathcal{E} \otimes_{\mathcal{O}_{\mathcal{M}}} \Omega^1_{\mathcal{M}}
\]

where $\Omega^1_{\mathcal{M}}$ is the sheaf of Kähler differentials on $\mathcal{M}$.

The \emph{Riemannian curvature tensor} is then a natural transformation of the $\mathcal{O}_{\mathcal{M}}$-module functor given by

\[
R: \mathcal{E} \times_{\mathcal{O}_{\mathcal{M}}} \mathcal{E} \times_{\mathcal{O}_{\mathcal{M}}} \mathcal{E} \to \mathcal{E}
\]

defined by the relation

\[
R(X,Y)Z = \nabla_X \nabla_Y Z - \nabla_Y \nabla_X Z - \nabla_{[X,Y]} Z,
\]

where $[X,Y]$ is the Lie bracket of vector fields, arising from the structure of the Lie algebroid associated with $\mathcal{E}$. The operator $[\nabla_X, \nabla_Y]$ is a commutator of differential operators.

Since the right hand side only depends on the values of $X, Y, Z$ at a given point, $R$ is a \emph{tensorial} object, i.e., a section of the sheaf: $ R \in \Gamma(\mathcal{M}, \operatorname{End}(\mathcal{E}) \otimes \Omega^2_{\mathcal{M}})$.

This tensor encodes the \textit{non-commutativity} of the covariant derivative and measures the curvature of the category of $\mathcal{O}_{\mathcal{M}}$-modules endowed with $\nabla$.

\begin{ex}\label{Ex:con1}
Show that we can rewrite it as: $$R(X,Y)=[\nabla_X,\nabla_Y]+\nabla_{[X,Y]}$$
\end{ex}
\begin{ex}\label{Ex:con2}
Symmetries: 
Show that we have the following relations: 
\begin{align}
&R(Y,X)=-R(X,Y) \\
&R(X,Y)Z+R(Y,Z)X+R(Z,X)Y=0\\
&\langle   R(X,Y)Z, W\rangle= \langle   R(Z,W)X, Y\rangle
\end{align}
\end{ex}
\subsection{Sectional curvature}
Let $\cM\subset \bbR^n$ be a smooth $m$-dimensional submanifold. 
Let $p\in \cM$ and let $E\subset \cT\cM$ be a 2-dimensional linear subspace of the tangent space. 
The sectional curvature of $\cM$ at $(p,E)$ is the number 
\[K(p,E)=\frac{\langle R_p(u,v)v,u\rangle }{|u|^2|v|^2-\langle u,v\rangle^2 }\]
where $u,v\in E$ are linearly independent and $R_p$ is the Riemannian curvature tensor. 

\begin{ex}
    Let us consider the manifold of symmetric positive definite matrices i.e. 
    \[\{A\in Mat_{n\times n}(\bbR)\, |\, A^T=A,\,  x^TAx>0 \}.\]
    Compute the sectional curvature for this manifold.
\end{ex}

The space of symmetric positive definite matrices has many interesting applications. 
\begin{itemize}
    \item Optimisation (convex programming).
\item Information geometry and harmonic analysis (via Wishart laws). These are important in random matrix theory and statistics.
\item Medicine (diffusion tensor Imaging) as a model of the anisotropic diffusion of water in the tissues. 
\item Machine learning, one is interested in finding the correlation between given properties (for example one could take weight / volume or weight and height). One models the relationship with a covariance matrix.
\end{itemize}
\begin{definition}
Let $k\in \bbR$ and $m\geq 2$ be an integer. An $m$-manifold has constant sectional curvature 
$k$ if and only if $K(p,E)=k$ for every $p\in \cM$ and every 2-dimensional linear subspace $E\subset \cT_p\cM$.    
\end{definition}

\begin{theorem}
  Let $\cM$ be an $m$-manifold. Fix an element $p\in \cM$ and a real number $k$. Then the following are equivalent:
$K(p,E)=k$ for every 2-dimensional linear subspace $E\subset \cT_p\cM$.
The Riemann curvature tensor of $\cM$ at $p$ is given by:
\[\langle R_p(v_1,v_2)v_3,v_4\rangle=k(\langle v_1,v_4\rangle-\langle v_2,v_3\rangle),\]
for all $v_1,\cdots, v_4 \in \cT_p\cM$.
  
\end{theorem} 

\section{Torsion in Differential Geometry}
\subsection{Torsion tensor}

Let $M$ be a smooth manifold, and let $\nabla$ be an affine connection on $M$. The presence of torsion in $\nabla$ quantifies the extent to which infinitesimal parallel transport fails to be symmetric.

\begin{definition}
The \emph{torsion tensor} $\boldsymbol{T}$ associated with a connection $\nabla$ is the $(1,2)$-tensor defined by
\[
\boldsymbol{T}(X,Y) = \nabla_X (Y) - \nabla_Y (X) - [X,Y],
\]
for any vector fields $X,Y$ on $M$, where $[X,Y]$ denotes the Lie bracket and $\nabla_Y (X)$ is the covariant derivative. The connection $\nabla$ is said to be \emph{torsion-free} if $\boldsymbol{T}=0$.
\end{definition}

\subsection{Local coordinates}
In local coordinates, $\{x^i\}$, the torsion tensor can be expressed in terms of the connection coefficients $\Gamma_{ij}^k$:
\[\boldsymbol{T}^i_{jk}=\Gamma_{jk}^i-\Gamma_{kj}^i,\]

where $\boldsymbol{T}^i_{jk}$ are the components of the torsion tensor, and $\Gamma_{jk}^i$ are the Christoffel symbols of the connection.

\begin{itemize}
    \item The torsion measures the \emph{failure of parallel transport} to preserve the commutativity of vector fields. 

\item Let us mention that the torsion affects the \emph{holonomy} of a connection, which describes how vectors rotate when parallel transported around closed loops.
\end{itemize}

In Riemannian geometry, the Levi-Civita connection is the unique torsion-free connection that is compatible with the metric. This connection is central to the study of curvature and geodesics.

\, 

\subsection{Geodesic Deviations}
In the presence of torsion, the equation of geodesic deviation (which describes how nearby geodesics spread apart or converge) includes additional terms involving the torsion tensor. 

Intuitively, imagine two nearby geodesics on a surface. If the surface is flat (such as a plane), the geodesics are straight lines, and the distance between them remains constant. However, if the surface is curved (for instance, a sphere), the geodesics may converge or diverge over time. Geodesic deviation quantifies this behavior.

\, 

Let $\mathcal{M}$ be a smooth manifold equipped with a sheaf $\mathcal{O}_{\mathcal{M}}$ of smooth functions and a sheaf $\mathcal{E} = \mathcal{T}_{\mathcal{M}}$ of sections of the tangent bundle. A \emph{Riemannian structure} (or a \emph{semi-Riemannian structure} in the case of Lorentzian manifolds) is given by a metric $g: \mathcal{E} \otimes_{\mathcal{O}_{\mathcal{M}}} \mathcal{E} \to \mathcal{O}_{\mathcal{M}}$.

A geodesic on $\mathcal{M}$ is a path satisfying the equation
\[
\nabla_{\dot{\gamma}} \dot{\gamma} = 0,
\]
where $\nabla$ is the Levi-Civita connection associated with $g$.

Now, consider a one-parameter family of geodesics, represented by the morphism
\[
\wp: I \times (-\epsilon, \epsilon) \to \mathcal{M},
\]
where $I$ is an interval in $\mathbb{R}$ parametrizing proper time or arc length, and the second parameter represents a perturbation of the initial geodesic. The \emph{geodesic deviation vector field} is given by
\[
\cJ = \frac{\partial \wp}{\partial s} \Big|_{s=0},
\]
which defines a section of $\mathcal{E}$ along the central geodesic $\gamma(s=0)$. This vector field satisfies the \emph{Jacobi equation}, which in sheaf-theoretic terms can be written as
\[
\nabla_{\dot{\gamma}} \nabla_{\dot{\gamma}} \cJ + R(\dot{\gamma}, \cJ) \dot{\gamma} = 0.
\]
Here, $R: \mathcal{E} \times_{\mathcal{O}_{\mathcal{M}}} \mathcal{E} \to \operatorname{End}(\mathcal{E})$ is the Riemann curvature tensor, viewed as a section of the sheaf $\operatorname{End}(\mathcal{E}) \otimes_{\mathcal{O}_{\mathcal{M}}} \Omega^2_{\mathcal{M}}$.

\subsubsection{Flat vs. Curved Geometry:}
In the case where $(\mathcal{M}, g)$ is flat (i.e., $\mathcal{M}$ is locally isometric to an open subset of $\mathbb{R}^n$ with the Euclidean metric), we have $R = 0$, and hence $\cJ$ satisfies
\[
\nabla_{\dot{\gamma}} \nabla_{\dot{\gamma}} \cJ = 0.
\]
This implies that geodesics remain equidistant, meaning that nearby geodesics neither converge nor diverge.

If $(\mathcal{M}, g)$ is curved, then the curvature tensor $R$ introduces a term that governs the deviation of geodesics. The sign and structure of $R(\dot{\gamma}, \cJ) \dot{\gamma}$ determine whether nearby geodesics converge or diverge, capturing the intrinsic curvature of $\mathcal{M}$ in a functorial manner.

Thus, geodesic deviation measures the sheaf-theoretic failure of parallel transport to be trivial in the presence of curvature.

\subsection{Skew-Symmetry and Vanishing Torsion}

The torsion tensor is skew-symmetric in its lower indices:
\[\boldsymbol{T}(X,Y)=-\boldsymbol{T}(Y,X).\]

\subsubsection{\bf Vanishing Torsion}

If $\boldsymbol{T}=0$, the connection is said to be torsion-free. In this case, the connection coefficients satisfy 
\[\Gamma_{jk}^i=\Gamma_{kj}^i,\] and the connection is symmetric. 
\subsection{Relations to Curvature}
The torsion tensor and the curvature tensor $R$ of a connection are related via the Bianchi identity: 

The \textbf{First Bianchi Identity} for a connection with torsion can be written as:

\begin{equation}
    R^\nabla(X, Y)Z + R^\nabla(Y, Z)X + R^\nabla(Z, X)Y = (\nabla_X \boldsymbol{T})(Y, Z) + (\nabla_Y \boldsymbol{T})(Z, X) + (\nabla_Z \boldsymbol{T})(X, Y).
\end{equation}

In components, this can be rewritten as: 

\begin{align*}
 &R^\lambda_{\;\;\mu\nu\rho} + R^\lambda_{\;\;\nu\rho\mu} + R^\lambda_{\;\;\rho\mu\nu} = \\
 & \nabla_\mu \boldsymbol{T}^\lambda_{\;\;\nu\rho} + \nabla_\nu \boldsymbol{T}^\lambda_{\;\;\rho\mu} + \nabla_\rho \boldsymbol{T}^\lambda_{\;\;\mu\nu} + \boldsymbol{T}^\sigma_{\;\;\mu\nu} \boldsymbol{T}^\lambda_{\;\;\rho\sigma} + \boldsymbol{T}^\sigma_{\;\;\nu\rho} \boldsymbol{T}^\lambda_{\;\;\mu\sigma} + \boldsymbol{T}^\sigma_{\;\;\rho\mu} \boldsymbol{T}^\lambda_{\;\;\nu\sigma}.
\end{align*}

\, 

For a more categorical approach, let $\cE$ be a sheaf of sections of the tangent bundle (or a vector bundle over a space), and let 
$\nabla$ be a connection on $\cE$. The curvature $R^{\nabla}$ and torsion $\boldsymbol{T}$
appear as components of a functorial construction on the Atiyah algebroid 
(the sheaf of first-order differential operators preserving a given structure).

\subsection{Spencer differential}
In particular, using the Lie-algebroid formalism, the first Bianchi identity corresponds to the failure of the \textit{Spencer differential} $d_\nabla$ acting on torsion to be zero, that is:

\begin{equation}
    d_{\nabla} \boldsymbol{T} + R^\nabla \wedge_{\nabla} \operatorname{id} = 0,
\end{equation}

where \begin{itemize}
    \item $d_\nabla$ is the Spencer differential associated with the connection.
    \item $\wedge_{\nabla}$ denotes the wedge operation compatible with the connection.
\end{itemize}

Specifically, if you parallel transport a vector $Y$ along a curve tangent to $X$, 
\begin{example}
The Torsion can be interpreted as an obstruction to the integrability of certain geometric structures,
such as foliations or distributions on a manifold.
\end{example}
\begin{example}
The torsion plays a central role in non-Riemannian geometries. This includes Finsler geometry for instance and 
teleparallel gravity. In those frameworks, the torsion tensor is used to describe the gravitational field, 
replacing the curvature tensor of general relativity.
\end{example}


\part{Probabilities, Statistics, and Related Topics }\label{Ch:pms}
\chapter{Probability theory - Measure theory- Statistics}

In this chapter we give a short overview on probabilistic and statistic notions.
We will see that everything that has seen introduced so far can be used in the context of probability and statistics,
namely information theory. Such a mixture gave birth to the theory of information geometry.  

We will also recall some notions of distance or divergence between probabilities, which is important for the ``learning process''.

\section{Key Information for Statistics}

When studying a mathematical problem in statistics, several key pieces of information come into play:
\begin{enumerate}
    \item  {\bf Qualitative Description of the Observed Phenomenon:} 
    
    \item[] This refers to a detailed understanding of the nature and characteristics of the event or situation being observed, which can be represented through various inputs $\omega_i$ collected from the sample space $\Omega$.
\item  {\bf Information About the Unknown Distribution of Outcomes $\omega$}: 
\item[]
This involves understanding how different outcomes are distributed, which is fundamental to statistical analysis.
\item  {\bf A Priori Probabilities}: 
\item[]
This entails the probabilities associated with observing various phenomena, represented as $(\Omega, \mathbf{S}, P_{\theta})$. Here, $\mathbf{S}$ denotes a $\sigma$-algebra that organizes measurable sets for the outcomes.
\item  {\bf A Distribution on the Parameter Space $Q\{d\theta\}$}:
\item[] This refers to a probability distribution defined on the range of possible parameter values that could explain the outcomes, denoted by the parameter space $\Theta$.
\end{enumerate}

In this context, the first item is associated with a collection of inputs gathered from the sample space $\Omega$. The second item involves the probability distribution $P[\cdot]$, which comes from a family of distributions ${P_{\theta}}$, where the parameter $\theta$ belongs to the parameter space $\Theta$. 
Overall, this transforms the issue of mathematical statistics into a decision-making problem, as one must determine which pieces of information are most beneficial. Formally, every statistical problem relates to a measurable space $(\Omega, \mathbf{S})$, which captures the outcomes of the observed phenomenon, a family of probability distributions ${P_{\theta}}$, and potentially additional a priori knowledge about the unknown phenomenon linked to another measurable space $(\Omega', \mathbf{S'})$, which may concern hypotheses or actions taken.

\section{Probability }\label{sec:proba}
\subsection{Probability space }
The basic notion in probability theory is the one of random experiment.
Such an experiment is
mathematically described by a probability space characterized by the triplet
$(\Omega,\bS, P)$ where:

\begin{enumerate}
\item $\Omega$ is a sample space i.e. the space of all possible outcomes
(samples) $\omega$ of the
random experiment.

\item ${\bS}$ is the family of events. An event  $A\in
{\bS}$, is identified with a subset $ A$ of $\Omega$.

 \item The probability associates to each event a real number between 0 and 1. A probability is
defined by a map from the space of events ${\bS}$ to the set $[0,1]$.
\end{enumerate}

\begin{definition}[$\sigma$-algebra]
A family ${\bS}$ of subsets of a sample space $\Omega$ is a $\sigma$-algebra defined by 
the following algebraic structures:

\begin{enumerate}
\item $\Omega \in {\bS}$, 
\item $ A \in {\bS}\Longrightarrow A^c \in {\bS},$
\item $A_1,A_2,\dots  \in {\bS}\Longrightarrow
\bigcup\limits_{i=1}^\infty A_i
\in {\bS}.$
\end{enumerate}
the pair $(\Omega,{\bS})$ is called measurable space.
\end{definition}

\begin{examples}
\ 

\begin{itemize}
\item The \emph{smallest} $\sigma$-algebra containing one $A\subset \Omega$ is $\bS=\{\emptyset, A,A^c,\Omega\}$.

\item  The \emph{largest}  $\sigma$-algebra is consisted of all subset  of $\Omega$. This  algebra served us in discret spaces, but is too large to be useful in general

\item If $\Omega$ is a topological space the $\sigma$-algebra generated by all the open subset of $\Omega$ is called the \emph{Borel's algebra}.
\end{itemize}
\end{examples}

\begin{ex}\label{Ex:6.1}
    Show that the $\sigma$-algebra generated by $n$ disjoint subsets of a set $\Omega$ has cardinality $2^n$
\end{ex}

 \begin{definition}
 A probability $P$ is the following map  $$P: {\bS}\to [0,1]$$ which
satisfies the following properties: 

\begin{enumerate}
\item $P[\Omega] = 1,$
\item $ P[\emptyset] = 0,$
\item $ P[\mathop{\cup}\limits_{i=1}^\infty A_i] =
\sum\limits_{i=1}^\infty P[A_i]\enspace {\mathrm if}\enspace  A_i\cap A_j =
\emptyset,\enspace
 \forall  i \not= j.$
\end{enumerate}
\end{definition}

\begin{itemize}
\item The triplet
$(\Omega,{\bS},P)$ is called a probability space.

\item An event $A \in {\bS}$ such that $P[A] = 0$ is called a $P$-null set.

\item An event $A \in {\bS}$ such that $P[A] = 1$  is said to be realized
almost surely (usually denote: ``a.s.'').
\end{itemize}

From this definition, we can obtain have  useful relations:
\begin{itemize}
 \item $P[A^c] = 1-P[A], \quad A\in {\bS}$
\item  $P[A\setminus B] = P[A] -P[B], \quad B\subseteq A \in {\bS}$, in
particular $P[B] \leq P[A]$
\item  $P[A\cup B] =P[A] +P[B]-P[A\cap B],\quad A\, B \in {\bS}$.
\end{itemize}

\section{Examples}
To  illustrate the  construction of probability spaces, let us discuss the following examples.

\subsection{Flipping coin examples: How biased is it?}
\begin{example}
    A fixed coin is flipped independently  $n$ times. The sequence of outcomes  
  $(\omega_1, \dots, \omega_n)$, consisting of heads  $H$ and tails $T$, might look like \[\underbrace{HTHTHHHHT\ldots THTTT)}_{n}.\] To determine whether a coin is biased, we need to analyze the probability of obtaining heads $H$ and tails $T$ over multiple independent flips.

  \, 
  
  The question is whether we can assume that the coin in question is not biased.   
\, 

    Equivalently, we can ask: does the probability of obtaining heads  $H$ deviate from $\frac{1}{2}$ by less than $10^{-3}$?
\end{example}

Now, consider the following situation:

\begin{example}
    The coin is flipped independently, but the probability distribution governing the outcomes is unknown. The only available information is a constant parameter $\theta $.  

    We do know, however, that the probability of observing any specific sequence  $\omega$  with $k(\omega$)  heads  $H$  follows the Bernoulli probability law (this is the Bayesian approach):
    \[
    P(\omega) = \theta^{k(\omega)} (1 - \theta)^{N-k(\omega)}.
    \]
    
    Consequently, the probability distribution $P[\cdot]$  is given by
    \[
    p_j(\theta) = \theta^{k({\omega}_j)} (1 - \theta)^{N- k ({\omega}_j)}, 
    \]
    where
    \[
    j=1, \dots, 2^{N},
    \]
    and
    \[
    \theta \in [0,1].
    \]
    
    This defines a parametric curve in the $(n - 1)$-dimensional simplex of all probability distributions on the sample space  $\Omega$, where $ n=2^N $.
    
    Now, suppose we impose a lexicographical order on all possible sequences of $ \omega$'s. Then, the number $ k({\omega}_j) $ corresponds to the number of zeros in the binary representation of the integer $ j - 1 $.
\end{example}

\,

To end our discussion on the example of the flipping coin, we observe the following. 

\begin{example}~
    \begin{itemize}
        \item For $N=1$, we obtain $n=2^1=2$.
         The unknown probability distribution  $P[\cdot]$ is represented by a  ``curve'' with parametric equation
    \[
    p_j(\theta) = \theta ^{k(\omega_j)}
    (1- \theta)^{1- k({\omega}_j)}, 
    \]
for all probability distributions on $\Omega$    
\[
    j=1,2;\quad \theta \in [0,1]
    \]
     in the $1$-dimensional simplex, that is on the segment.

      \item  For $N=2$, we obtain $n=4$.
         The probability distribution $P[\cdot]$ is represented by a curve with parametric equation:   
\[
p_j(\theta) = \theta^{k({\omega}_j)} (1- \theta)^{2-k({\omega}_j)},
\]
where
\[
 j= 1,...,4
\]
and
\[
 \theta \in [0,1].
\]
In this case, the parametric curve is a curve in the $3$-dimensional probabilistic simplex.

\, 

We explain how this curve behaves in relation to the representation of words in letters $H$ and $T$. 

We begin with $\theta=0$, which corresponds to the vertex $P[TT]=1$. Then, the curve goes through the center of our simplex, since we have $P[HH]=P[HT]=P[TH]=P[TT]=\frac{1}{4}$ and where $\theta =\frac{1}{2}$. The endpoint of the curve is given by the vertex $P[HH]$ which corresponds to $(\theta=1)$.
    \end{itemize}   
\end{example}

More generally, the flipping coin example can be re-contextualized as follows. 
\begin{example}
For finite probabilities, we may use the \emph{multinomial distribution},  with sample space $\Omega=
\{\omega_1,\dots,\omega_m\}$ and $\sigma$-algebra generated by ${\bS}=
{\mathcal P}(\Omega)$, which is the set of subsets of $\Omega$ generated by the
elementary one element sets $\{\omega_i\}$. The probability is defined by 
\begin{equation}
 p_k=p(\omega_k)=P[\{\omega_k\}],\quad \sum\limits_{k=1}^m p_k=1,
\end{equation}
with
\[ p=\sum\limits_{k=1}^m
p_k\,\delta_{\omega_k}\]
and this induces \begin{equation}
p=\sum\limits_{k=1}^{m-1}p_k\,\delta_{\omega_k}
 +(1-\sum\limits_{k=1}^{m-1} p_k)\,\delta_{\omega_m}.
\end{equation}

\end{example}

\subsection{Other examples}
\begin{example}
    In the case of countably infinite probabilities, let 
$\Omega=\{\omega_1,\omega_2,\dots\}$ and let the $\sigma$-algebra be generated
by the elementary sets $\{\omega_i\}$. The probability is defined by
its value on these sets such that:
\begin{equation}
 p_k=p(\omega_k)=P[\{\omega_k\}],\quad \sum\limits_{k=1}^\infty p_k=1.
\end{equation}

A typical example of such probability distribution is the Poisson
distribution,
\begin{equation}
p_k= e^{-\lambda}{\lambda^k\over k!}.
\end{equation}

 In the case of continuous probability distributions with respect to the Lebesgue measure, we consider the case where the sample space is identified to 
$\mathbb{R}$ with the Borel
$\sigma$-algebra ${\bS}$ being the subset of ${\mathcal P}(\Omega)$
generated by the intervals $(-\infty,a],$ where $ a\in {\mathbb R}$. The
probability is then given by 
\begin{equation}
P\big[ (-\infty,
a]\big] = \int\limits_{-\infty}^a
\rho(\omega)\, d\omega 
\end{equation}
with
\[\rho(\omega)\geq 0,\quad \int\limits_{-\infty}^\infty
\rho(\omega)\, d\omega =1.\]
A classical example is the Gaussian distribution: 

\begin{equation}
P\big[ (-\infty, a]\big]= \int_{-\infty}^a
\rho(x)\, dx\,,\quad
\rho(x) = \frac{1}{ \sigma\sqrt{2\pi}}
e^{-\displaystyle{\frac{(x-\mu)^2}{ 2\sigma^2}}}.
\end{equation}
\end{example}

\begin{ex}\label{Ex:BorelThm} Prove the following theorem on $\sigma$-Algebras and subcovers.
Let $(X, \mathcal{A}, \mu)$ be a measure space, and let $\mathcal{B} \subseteq \mathcal{A}$ be a sub-$\sigma$-algebra of $\mathcal{A}$. If $\{ A_n \}_{n \in \mathbb{N}}$ is a countable cover of $X$ by sets in $\mathcal{A}$, then there exists a subcover in $\mathcal{B}$ if and only if the measure $\mu$ on $\mathcal{A}$ is uniquely determined by its restriction to $\mathcal{B}$.
\end{ex}

\section{Measure}\label{sec:meas}
In many contexts, it is beneficial to consider a notion of measure that extends beyond the confines of probability theory. Unlike a probability measure--which by definition assigns a total measure of 1 to the entire space--a general measure does not require such normalization. By relaxing the normalization condition, we can define and work with measures that better reflect the intrinsic properties of the space under consideration. This generalization opens up new possibilities for both theoretical developments and practical applications, ranging from quantum field theory to ergodic theory and beyond.

\begin{definition}
 A measure $\lambda$ on the measurable space $(\Omega,{\bS})$ is a map, which maps events to real positive numbers such that:
\begin{enumerate}
\item $\lambda: {\bS}\to \bbR^+$,
 \item $\lambda[\emptyset] = 0,$
\item $\lambda[\mathop{\cup}\limits_{i=1}^\infty A_i] =
\sum\limits_{i=1}^\infty P[A_i]\enspace$ if $ A_i,\cap A_j = \emptyset, \,
 \forall\enspace i \not= j.$
\end{enumerate}
\end{definition}

\vspace{5pt}
\begin{itemize}
\item A measure $\lambda$ is said to be emph{bounded (or finite)} if  $\lambda[\Omega]<\infty $.
Hence a probability is a finite measure (bounded by 1).\\

\item A measure $\lambda$ is \emph{$\sigma$-finite} if 
\[\Omega =
\bigcup\limits_{k=1}^\infty
A_k, \, \text{ with  }\,  \lambda[A_k]< \infty.\] 
\end{itemize}

\begin{remark}
We draw the readers attention upon the fact that a $\sigma$-finite measure is in general not finite. An example is
the Lebesgue measure $\lambda$ on the real line, which is finite on all bounded
interval $\lambda[ a,b] = \mid b-a\mid$, but not finite.
\end{remark}

 \begin{ex}\label{Ex:Borel}
   If $X$ is a metric  space and $A$ and $B$ are two disjoint closed subsets of $X$, then there exists a continuous function $f(x)$ on $X$ with properties: 

   \,

   1) $0\leq f(x)\leq 1$

\, 

   2) \[\begin{cases}
       f(x)=0 & \text{for} \quad  x\in A \\
       f(x)=1 & \text{for} \quad x\in B
   \end{cases}\]
    
   \end{ex} 

\section{Measurable function - Random variables}\label{sec:rand}

\subsection{Definition}
\ 

Let us first consider  a random experiment described by the probability
space $(\Omega,\bS,P)$.  A real random variable is an application from the sample space
$\Omega$ to a subspace $E\in \bbR$  which preserves the
algebra of events. 
 
 \begin{definition}[Mesurable function]
  Let $(E, \cB)$ a mesurable space, a \emph{mesurable function} $f$ from
$(\Omega,\bS)$ to  $(E, \cB)$ is  such the inverse image 
of a mesurable set in $\cB$ is a mesurable set in  $\bS$
\[ 
f^{-1}(B)=\{\omega \in \Omega,\, X(\omega) \in B\} \in \bS, \quad \forall B\in
\cB.
\]
\end{definition}

\begin{definition}[Random variable]
  Let $(\Omega, \bS,P)$ a probability space, a $E$ valued \emph{random variable} $X$ is the class of $P$-ae ($P$-almost everywhere, i.e. up to set of $P$-measure $0$)  measurable function  from $(\Omega,\bS,P)$ to  $(E, \cB)$   such the inverse image of a measurable set in $\cB$ is a measurable set in  $\bS$

\end{definition}

\begin{itemize}
\item A real random variable on the measurable space $(\Omega,\bS)$ is a measurable
function $X$ with value in $E \subseteq\bbR$ (or $\overline\bbR=\bbR\cup\{-\infty,+\infty\}$) and $\cB$ is the $\sigma$-algebra generated by the topology on $E$. \\

\item If $E$ is finitely or infinitely many
countable we speak about finite or discrete random variables.\\

\item If $E$ is a vector space we speak about random vectors.
\end{itemize}

Let Let $(\Omega, \bS,P)$ a probability space, the integral over $\Omega$, if there exist, of a random variable $X$  is called the {\emph{expectation (or mean)} of $X$, and denoted 
\[
\bbE_P[X]=\int_\Omega X(\omega) P[d\omega].
\]

That is $\bbE_P[X]$ exist if $X\in L(\Omega,\bS,P)$.

\section{Dominating Measure}

\begin{definition}
If $ \mu $ and $ \nu$ are two measures on the same measurable space  $(\Omega,\bS)$,
$\mu $ is said to be \emph{absolutely continuous} with respect to $ \nu $ if $ \mu ( A)=0$ for every set 
$A$ for which $\nu (A)=0$. This is written as "$ \mu \ll \nu $". That is:
\[\
mu\ll \nu \text{ if and only if  \  for all } A \in \bS, \ (\nu (A)=0 \Leftarrow  \mu (\cA)=0).
\]

When $ \mu \ll \nu $,then $\nu $ is said to be \emph{dominating} $ \mu $.
\end{definition}

Absolute continuity of measures is:
\begin{itemize}
\item reflexive, 
\item transitive, 
\item not antisymmetric
\end{itemize}

\vspace{3pt}
If $ \mu \ll \nu $ and $ \nu \ll \mu$, the measures $ \mu $ and $\nu $ are said to be \emph{equivalent}. 
Thus absolute continuity induces a partial ordering of such equivalence classes.

\begin{ex}\label{Ex:domMeas}
   As a concrete example of a dominating measure,  let $\{P_k\}$ be a finite or countable family of probabilities on the same measurable space $(\Omega,\bS)$ and consider
   \begin{itemize}
       \item the arithmetic mean of a finite collection of probability measures:
       \[
       P_0[\cdot] = \left(\frac{1}{n} \sum_{k=1}^n P_k\right)[\cdot],
       \]
       \item or, in the case of a countable family, the weighted series:
       \[
       P_0[\cdot] = \left(\sum_{k=1}^{\infty} \frac{1}{2^k} P_k\right)[\cdot].
       \]
   \end{itemize}
   Prove that $P_k \ll P_0$
\end{ex}

\vspace{3pt}

If $ \mu $ is a signed or complex measure, it is said that $ \mu$  is absolutely continuous with respect to 
$\nu $if its \emph{variation } $ |\mu |$ satisfies$ |\mu |\ll \nu$ ; equivalently, if every set 
$\cA$ for which $ \nu (A)=0$ is $ \mu $-null.

\begin{theorem}[The Radon--Nikodym theorem]
If $\mu$ is absolutely continuous with respect to $\nu$, and both measures are $\sigma$-finite, then $\mu$ has a density, or "Radon--Nikodym derivative", with respect to 
$\nu$, which means that there exists a $\nu$-measurable function 
$\rho$ taking values in $ [0,+\infty )$, denoted by $ \rho=d\mu /d\nu$ , such that for any 
$\nu$-measurable set $\cA$ we have:
\[
 \mu (A)=\int _{A}\rho\,d\nu .
\]
\end{theorem}

If $\mu=P\ll \nu$ is a probability and $\nu$ is a positive $\sigma$-finite measure, the Radon--Nikodym derivative $\rho=\frac{dP}{d\nu}$ is a random variable, called density  (or distribution) of $P$ with respect to the measure $\nu$. We have

\[
\rho= \frac{dP}{d\lambda},\quad \quad P(A)=\int_{A}\rho d\nu,\  \forall \, A\subset \bS, \quad \int_{\Omega}\rho d\nu=1 . 
\]

In specific cases, a probability measure  $\mu$ can behave as a dominating measure for another measure $\nu$
if $\nu \ll\mu.$ 

\, 

\section{About Information Geometry}

The family of probability densities of (parametric) probability measures is denoted: 
\[\sfS= \{\rho_\theta \in L^1(\Omega, \bS,\mu) \quad \rho>0, \quad \mu - a.e \quad \int_{\Omega}\rho d\mu=1\}.\]
We consider the case when $\sfS$ is a smooth topological manifold.

Probability distributions in a statistical manifold of exponential type are such that there exist parameterizations $\theta$ satisfying:
\[
\rho_\theta = \exp \left( \sum_{i=1}^n \theta_i X^i- \psi (\theta)\right),
\]

where parameters $\theta$ and random variables $X=(X^i)_{i=1}^n$ are chosen adequately,
$\psi(\theta)$ is a \emph{potential function}.

In information geometry, an important question is whether a given family of probability distributions can be recognized as an exponential family--or, more generally, if it exhibits an exponential structure. Recall that a probability distribution belongs to an exponential family if its density function can be written in the form:
\[\rho_\theta(\omega)=h(\omega)\exp\left( \eta(\theta)\cdot \sfT(\omega)-\sfA(\theta)\right),\]
where: 

\begin{itemize}
    \item $h(\omega)$is a base measure (Radom--Nikodym derivative w.r.t the dominated measure $\mu$),
    \item $\eta(\theta)$ represents natural parameters
    \item $\sfT(\omega)$ is a sufficient statistic,
    \item $\sfA(\theta)$ is the log-partition function ensuring normalization.
\end{itemize}

\vspace{5pt}
This representation is not merely a convenient rewriting; it endows the parameter space with a rich geometric structure. 
In particular, when a family of distributions is exponential, the Fisher information metric
 $g_{ij}=\bbE[ \partial_i \ln \rho \; \partial_j \ln \rho]$ and the dual affine connections
(which emerge naturally in this setting) yield a dually flat space. This flatness simplifies the study of geodesics, divergence functions, and many aspects of statistical inference.

A classical example is the normal (Gaussian) family. 
While the normal distribution indeed constitutes an exponential family, its precise exponential form depends critically on the choice of parametrization. 

\begin{ex}\label{Ex:Gauss}
The Normal/Gaussian distributions on ($\bbR, \bB,\lambda)$, where $\bB$ is the Borel algebra on $\bbR$ and $\lambda[(a,b)]=\vert b-a\vert$ the Lebesgue measure is defined by  the probability distribution with respect to the Lebesgue measure on $\bbR$ by: 
\[\rho_{\mu,\sigma}(x)=\frac{1}{\sqrt{2\pi}\sigma}\exp{-\frac{(x-\mu)^2}{2\sigma^2}}, \quad \mu\in \bbR,\ \sigma >0.\]
Is the normal family exponential? If yes, give a proof. 
\end{ex}

The specific choice of parametrization affects not only the expression of the density but also the form of the Fisher information metric and the associated affine connections. 
Thus, determining whether a family of probability distributions is of exponential type--and selecting an appropriate parametrization in cases like the normal family--is fundamental in information geometry. It allows one to exploit powerful geometric techniques to derive insights into the behavior of statistical models, facilitating tasks such as parameter estimation, hypothesis testing, and the study of statistical divergences.

\subsection{Tangent space to a manifold of probability distributions}
Let $u\in L^{2}(\Omega,\bS, P_{\theta})$ be a tangent vector to the manifold  $\sfS$ (of dimension $n$) of
probability distributions (we assume of exponential type) at the point $P_{\theta}$.
\[
\cT_{\theta}=\left\{ u\in L^{2}(\Omega, \bS, P_{\theta}); \quad \bbE_{P_{\theta}}[u]=0,\quad u=\sum_{i=1}^{d}u^i\partial_{i}\ell_{\theta}         \right\}
\]
where $\bbE_{P_{\theta}}[u]$  is the expectation value, with respect to the probability ${P_{\theta}}$.

The tangent space to the considered manifold $\sfS$ is (locally) isomorphic to the $n$-dimensional linear space generated by the family of (centered) random variables (called \emph{score vector}),
$( \partial_{i}\ell_\theta), \ \{i=1,\dots,n\}$ where $\ell_{\theta} = \ln \rho_{\theta}.$

\begin{remark}
Densities are positive random variables almost everywhere; the tangent vectors are signed measures (it is a real-valued measure) with vanishing mean value.  
\end{remark} 

\begin{ex}\label{Ex:esp0}
  Show that $\bbE_{P_{\theta}}[u]=0$.
\end{ex}

\begin{ex}\label{Gauss2}
Show that the tangent space $\cT_{\mu,\sigma} $ to the manifold of Gaussian distribution \eqref{Ex:Gauss} is spanned, in the $(\mu, \sigma)$-parametrization, by the two random variables

\begin{itemize}
    \item ${\partial_\mu \ln \rho (x) = \frac{\partial \ln \rho}{\partial \mu}(x)  =\frac{x-\mu }{
\sigma^2}}$
\item ${\partial_\sigma \ln \rho (x) = \frac{\partial \ln \rho }{ \partial \sigma}(x)  =\frac{(x-\mu)^2 }{\sigma^3}-{1\over \sigma}}.$
\item Deduce that the plane $\cT_{\mu,\sigma}$ consists of all the quadratic polynomials in $x$ whose
expectation vanishes: 
\[
T_{\mu,\sigma}= \{ Ax^2+Bx+C, \quad C = A(\sigma^2 +\mu^2) -B\mu\}.
\]
\end{itemize}
\end{ex}

\subsection{The Riemannian metric}
  In the basis $\partial_i\ell_{\theta}$ the (Fisher--Rao) metric is the covariance matrix of the score vector:
\[   {g_{i,j}(\theta)=\mathbb{E}_{P_{\theta}}[\partial_i\ell_{\theta}\partial_{j}\ell_{\theta}]   } \]
 with   
\[{g^{i,j}(\theta)=\mathbb{E}_{P_{\theta}}[a^{i}_{\theta}a^{j}_{\theta}]},\] 

where  $\{a^{i}\}$ form a dual basis to $\{\partial_j\ell_{\theta}\}$:

\[a^{i}_{\theta}(\partial_j\ell_{\theta})=\mathbb{E}_{P_{\theta}}[a^{i}_{\theta}\partial_j\ell_{\theta}]=\delta^i_{j}\]

and
\[
\mathbb{E}_{P_{\theta}}[a^{i}_{\theta}]=0.
\]

\subsection{Applying the notions of parallel transport}
Suppose that we perturb infinitesimally $\theta $ such that $ \theta'=\theta+d\theta$.
Consider the linear map

\[
m:\cT_{\theta+d\theta}\to \cT_{\theta}
\] (this map depends on $d\theta$ and reduces to the identity map as $d\theta$ tends to 0). 

The given vector $$u^i\partial'_i \in \cT_{\theta+d\theta}$$ is mapped to $(u^k+d\theta^i\Gamma_{ij}^ku^j)\partial_k\in\cT_{\theta}$. 

This construction allows to establish a correspondence between the vectors in $T_{\theta}$ and $T_{\theta'}$.

The function $\Gamma_{ij}^k(\theta)$ are the coefficients of the affine connection.

\subsection{Covariant derivative}

Let $\pi:E\to \sfS$ be a vector bundle. A covariant derivative (also knows as a connection) is an $\bbR$-bilinear map 

\[
\nabla: \Gamma(\cT\sfS)\times \Gamma(E)\to \Gamma(E) \quad (X,s)\mapsto \nabla_{(X,s)}
\]
Consider the intrinsic change in the $j$-th basis vector $\partial_j (\theta)$ as $\theta$ deforms into $\theta'$, 
in the direction of $\partial_i$. We obtain the following vector field:
\[
\nabla_{\partial i}\partial_j=\Gamma_{ij}^k(\theta)\partial_k(\theta).
\]
This is a covariant derivative of the vector field $\partial_j$ along $\partial_i$. It is determined from the coefficients of the affine connection (namely $\Gamma_{ij}^k(\theta)$).

\subsection{Amari--Chentsov tensor}

There exists a skewness tensor. It is a fully symmetric covariant tensor of rank 3: 

\[\cT\cM \times \cT\cM \times \cT\cM\to \mathbb{R},\] 

given by \[ T|_{P_{\theta}}(u,v,w)= \mathbb{E}_{P_{\theta}}[u_{\theta}v_{\theta}w_{\theta}] \] 

so that we have:  \[T_{ijk}(\theta)=\mathbb{E}_{P_{\theta}}[\partial_i\ell_{\theta}\partial_j\ell_{\theta}\partial_k\ell_{\theta}].\]

The skewness tensor was introduced to formalize the notion of statistical curvature, via the affine connections $\{\nabla^{\alpha}\}_{\alpha\in \mathbb{R}}$. 
We have the following:
\[\nabla^{\alpha}_{X}Y= \nabla^{0}_{X}Y- \frac{\alpha}{2}(\underbrace{T.g^{-1}}_{\overline{T}})(X,Y),\]
 for any couple of vector fields $X, Y$ over $\sfS$; 
$\nabla^{0}$ is the Levi--Civita connection and $(\cdot)$ is the “contraction” (of two tensors).   

\subsection{Statistical inference}
The coefficients of the affine connection are:

\[\Gamma_{ijk}^\alpha=\mathbb{E}[\{\partial_i\partial_j\ell(x,\theta)+\frac{1-\alpha}{2}\partial_i\ell(x,\theta)\partial_j\ell(x,\theta)\}\partial_k\ell(x,\theta)].\] 
This called the $\alpha$-connection. 

Recall that the third order Amari--Chentsov tensor is defined to be: 
\[T_{ijk}(\theta)=\mathbb{E}[\partial_i\ell_\theta\partial_j\ell_\theta\partial_k\ell_\theta].\]

We use this tensor to simplify calculations of the $\alpha$-connection i.e.: 
$$\Gamma_{ijk}^\alpha=\Gamma_{ijk}^1+\frac{1-\alpha}{2}T_{ijk}.$$

 The $\alpha$-connection has a meaning of its own, depending on $\alpha$.  It plays, for example, an important role in statistical inference.

\section{Hints and solutions to Exercises}

Exercise \ref{Ex:6.1}. It is easy to show it using basic set theory.

\, 
 
Exercise \ref{Ex:BorelThm}.
~

1. \textbf{Second-Countability of $ X $:} \\
   By definition, $ X $ is second-countable, meaning it has a countable basis $ \mathcal{B} = \{B_n\}_{n \in \mathbb{N}} $ for its topology. That is, every open set in $ X $ can be written as a union of elements of $ \mathcal{B} $.

2. \textbf{Open Cover $ \mathcal{U} $:} \\
   Let $ \mathcal{U} = \{U_\alpha\}_{\alpha \in I} $ be an open cover of $ X $. This means:
   \[
   X = \bigcup_{\alpha \in I} U_\alpha.
   \]

3. \textbf{Constructing a Countable Subcover:} \\
   For each $ x \in X $, there exists some $ U_\alpha \in \mathcal{U} $ such that $ x \in U_\alpha $. \\
   Since $ \mathcal{B} $ is a basis, there exists some $ B_{n_x} \in \mathcal{B} $ such that:
   \[
   x \in B_{n_x} \subseteq U_\alpha.
   \]
   Let $ \mathcal{B}' = \{B_{n_x} \mid x \in X\} $. This is a subset of $ \mathcal{B} $, and since $ \mathcal{B} $ is countable, $ \mathcal{B}' $ is also countable.

4. \textbf{Extracting a Countable Subcover:} \\
   For each $ B_{n_x} \in \mathcal{B}' $, choose one $ U_\alpha $ such that $ B_{n_x} \subseteq U_\alpha $. Let $ \mathcal{U}' $ be the collection of all such $ U_\alpha $. \\
   Since $ \mathcal{B}' $ is countable, $ \mathcal{U}' $ is also countable.

5. \textbf{Verification that $ \mathcal{U}' $ is a Cover:} \\
   For any $ x \in X $, there exists $ B_{n_x} \in \mathcal{B}' $ such that $ x \in B_{n_x} \subseteq U_\alpha $ for some $ U_\alpha \in \mathcal{U}' $. \\
   Thus, $ x \in U_\alpha $, and since $ x $ was arbitrary, $ \mathcal{U}' $ covers $ X $.

6. \textbf{Conclusion:} \\
   We have constructed a countable subcover $ \mathcal{U}' $ of $ \mathcal{U} $, proving the theorem.

\,

  Exercise \ref{Ex:Borel}. 
   \, 
   
   If $\inf_{x\in A,y\in B} d(x,y)=\delta>0$ then the function $f$ can be chosen to be uniformly continuous. 

   Let $f(x)=d(x,A)/[d(x,A)+d(x,B)]$. It is easy to verify that $f$ satisfies both conditions 1 and 2. The other part of the theorem follows from the fact that $d(a,X)+d(x,B)\geq \delta$ and from the following fact that the function $d(x,A)$ satisfies the inequality \[|d(x,A)-d(y,A)|\leq d(x,y).\]
   In particularly, $d(x,A)$ is uniformally continuous. 

\, 

Exercise. \ref{Ex:domMeas}. The proof is omitted since it is considered in Ex. 7.2.
\, 

Exercise. \ref{Ex:esp0}. The proof is omitted since it is a straightforward calculation.

\, 

Exercise \ref{Ex:Gauss} and Ex. \ref{Gauss2}.
By an appropriate change of variables, for the Gaussian, this density can be rewritten in the exponential family form, 
revealing its natural parameters and sufficient statistics.

\chapter{Categorical Structures in Probability and Measure Theory}
In this chapter, we embark on a journey into probability and statistics in a manner that may initially surprise the reader: through the lens of philosophy. Rather than taking the conventional route of axiomatic probability theory, we begin with a deeper reflection on the nature of mathematical structures and their underlying symmetries. Our point of departure is a generalization of Felix Klein’s revolutionary insight--his Erlangen Program--which reinterpreted geometry as the study of figures and their congruences under transformation groups.

\, 

We invite the reader to consider an analogous perspective in the realm of probability distributions. Just as Klein’s figures reside in a structured mathematical world shaped by transformations, we view probability distributions as abstract entities within a broader conceptual space. This space, much like Plato’s world of ideas, consists of idealized probabilistic objects related by structural equivalences. These equivalences, or congruences, form the foundation upon which we build a categorical framework for probability theory.

\, 

This shift in perspective naturally leads us to the introduction of categories of probability distributions, where morphisms arise from Markov kernels--stochastic mappings that embody the very notion of congruence in this setting. In this way, Markov kernels serve as the probabilistic analogues of Klein’s geometric transformations, revealing deeper structures in statistical theory that might otherwise remain hidden.

\, 

Through this approach, we illuminate the rich interplay between probability, statistics, and abstract mathematical structures, setting the stage for a unifying perspective that extends beyond classical interpretations.

\section{Philosophical Approach}
\subsection{}Let us begin by considering a broad and fundamental question: 
\begin{center} {\it what is, in general, the aim of geometry?}  
\end{center}

One can state succinctly that geometry is primarily concerned with:  
\begin{itemize}
    \item the study of spaces.
    \item The transformations of these spaces, namely, their motions or the actions of groups.  
    \item Entire classes of geometric objects.  
\end{itemize}

\, 

\subsubsection{}
The foundations of geometry, as formulated by Klein, are based on the study of geometric figures and the transformations that act upon them. This perspective, known as \textit{Erlangen’s program}, emphasizes the structural relationships between figures rather than their specific realizations. In its classical setting, this formalism provides a natural framework for understanding symmetries, invariants, and classification problems within geometry.

\,

\subsection{}
For us, however, we will think of those (Klein) figures as figures {\it in the sense of Plato}. Plato's world of geometric figures refers to his philosophical concept of the {\it World of Forms}, where perfect and unchanging geometric shapes exist, independently of the physical world.

\, 

In Plato's philosophy, as described in dialogues such as the \emph{Republic} and the \emph{Phaedo}, mathematical objects: circles, triangles, and other ideal forms do not exist merely as physical approximations but as pure, abstract entities in a higher, non-material realm.
\footnote{In {\it The Republic,} Plato’s ``Allegory of the Cave'' illustrates how ordinary humans perceive mere shadows (imperfect physical objects), while true philosophers seek to understand the Forms, including perfect geometric structures. A drawn triangle or sphere in the physical world is an imperfect shadow of the true ideal triangle or sphere that exists in the intelligible realm.}

\,

It is precisely {\it this concept} of Plato's world of figures/forms which is used in Klein's perception and which we will use also to describe geometric structures in probability and measure theory.  
 
\subsection{}
In the present work, we extend this viewpoint beyond its original domain by considering a formalism in which Klein’s approach is adapted to the realm of probability theory and statistics. The motivation for such a generalization arises naturally: if one seeks to construct a statistical geometry, it is necessary to first identify the appropriate notion of \textit{geometric figures} and the corresponding transformations that preserve the relevant structures. 

\,
\subsubsection{}
Our approach follows a twofold strategy. First, we identify probabilistic objects that play the role of geometric figures in this generalized setting. These may include, for instance, families of probability distributions, families of distributions, or spaces of statistical models. Second, we introduce transformations between these objects that preserve key statistical or probabilistic properties. 
\, 

The fundamental problem then becomes: 

\,

\emph{how does one define and classify these transformations, and what structures do they induce?}

\,
\subsubsection{}
From this perspective, Markov categories and other categories called $CAP$, $CAPH$, 
$FAMH$ emerge as natural generalizations of geometric figures in Klein’s sense, and their transforms. These transformations--given by statistical morphisms, stochastic maps, or Markov kernels--serve as the analogues of geometric transformations in classical geometry. This interpretation provides a unifying language in which statistical inference, decision theory, and information geometry can be understood in terms of an overarching geometric framework.

\subsubsection{}

The goal of this work is to explore this formalism systematically and to establish the necessary mathematical structures that allow for a coherent treatment of statistical geometry. In particular, we shall investigate the role of categorical constructions, invariants under transformations, and the emergence of new geometric notions adapted to probabilistic contexts. Ultimately, our approach seeks to trivialize Klein’s original vision in such a way that its extension to statistics appears as a natural and inevitable development.

\,
The Klein figure formalism admits a natural extension to probabilistic contexts, particularly in the study of parametrized families of probability distributions. This adaptation establishes a coherent theoretical framework that facilitates systematic analysis and operational flexibility within such probabilistic structures.

\subsection{Towards a generalization of Klein's Plato-Figures}
The most elementary class of geometric objects consists of subsets of a given space, referred to by Felix Klein as \textit{figures}.

\,

\begin{itemize}
    \item Two figures $A$ and $B$ are said to be \textit{congruent} if there exists a motion mapping $A$ onto $B$, along with an inverse motion mapping $B$ onto $A$. The geometer’s principal concern is to investigate the properties of figures that remain invariant under such motions--those properties that are identical for all congruent figures. 
    \item[]
    \item The central problem thus becomes the determination of \textit{invariants} of figures, namely, real-valued functions that take equal values on congruent figures. A complete system of such invariants allows for a full {\bf classification of geometric configurations. }
\end{itemize}

\, 

Since any figure can be regarded as a set of points, one is naturally led to a {\bf generalization of this concept } of geometry--extending beyond the study of figures in the classical sense. Alongside figures in Klein’s sense, which are simply sets, one may consider \textit{parametrized sets}, thereby introducing a new geometric framework to examine their intrinsic properties.

\subsection{}

Let $\Theta$ be an index set parametrizing such a collection, and let us construct an epimorphism from this parameter set onto the space under consideration. That is, for each parameter value $\theta \in \Theta$, a corresponding point of the considered figure is assigned. In general, a single point of the figure may be associated with multiple parameter values, reflecting a richer underlying structure.

\,

Following Klein's approach to geometry, we may consider parametrized sets as Klein figures, where the parametrization is defined in a specific way. Such examples include:
    \begin{itemize}
        \item finite ordered sequences of points,  
        \item parametrized smooth curves or surfaces,  
        \item smooth manifolds parametrized by locally smooth coordinates--these can be treated both as points or as smooth surfaces.
        \end{itemize}

Furthermore, in the scope of generalizing objects and structures, we can recall what a category is, as we will need it later. 

\,
\subsection{Categories as a classifying tool}
A category is a mathematical structure that captures relationships between objects in an abstract way. At its core, a category consists of objects and morphisms (arrows) between them, which satisfy rules of composition and identity. One should think of a category not just as a formalism but a lens through which mathematics reveals its hidden symmetries and harmonies.

To be more precise, A category 
$\mathscr{C}$ consists of the following data:
\begin{itemize}
    \item A collection of objects, denoted by ${\rm Ob}(\mathscr{C})$
       \item  For each pair of objects $X, Y \in {\rm Ob}(\mathscr{C})$, a set of morphisms from $X$ to $Y$.
\end{itemize}
These morphisms must satisfy the following axioms:

\,

{\bf (1) Composition} 

For any three objects $X,Y,Z$ there exists a composition map $\circ$ which assigns to each pair of morphisms $f:X\to Y$ and $g:Y\to Z$ a composite morphism 
\[g\circ f:X\to Z.\]

\,

{\bf (2) Identity Morphisms}

For each object $X$, there exists an identity morphism $id$ such that for all 
$f:X\to Y$ and $g:Y\to X$ we have 
\[id_Y\circ f=f\quad \text{and} \quad g\circ id_Y=g\]

\, 

{\bf (3) Associativity}

For any morphisms $f:X\to Y$, $g:Y\to Z$, and $h:Z\to W$
one requires:
\[h\circ(g\circ f)=(h\circ g)\circ f.\]

\, 

\section{Generalization of Klein's (Plato)-figures/Congruences}

\subsection{Generalized Klein Figures}
One extends Klein’s notion of figures, beyond these classical cases, namely in the framework of probabilities. 
The following examples provide examples of  generalized Klein figures.

\, 
    \begin{itemize}
        \item Sets of probability distributions of a given type;  
        \item Dominated families of probability distributions, namely, all finite families of probability distributions on finite $\sigma$-algebras.  
    \end{itemize}
(Indeed, any finite or countable family $\{P_k\}$ of probability measures on a finite $\sigma$-algebra is necessarily dominated, as it admits a dominating measure.)

\subsection{Markov Transformations}

Using the Klein figure formalism, an essential ingredient in constructing a geometric framework for probabilistic or statistical quantities is the study of transformations acting on generalized Klein figures. We now introduce such transformations.

\begin{definition}
    Let $(\Omega, \bS)$ and $(\Omega', \mathbf{\Sigma})$ be two measurable spaces. A \textit{Markov transition probability}, or equivalently, a \textit{Markov morphism}, is a real-valued function \[\Pi(\omega, \mathcal{A})\] defined for $\omega \in \Omega$ and $\mathcal{A} \in \mathbf{\Sigma}$, satisfying the condition that for each fixed $\mathcal{A}$, the function \[\omega \mapsto \Pi(\omega, \mathcal{A})\] is $\bS$-measurable.  
\end{definition}

This transition probability distribution describes a Markovian random transition from the measurable space $(\Omega, \bS)$ to $(\Omega', \mathbf{\Sigma})$. The following theorem establishes an extension property for such transformations.

\,

    Let us consider now, classical example of a statistical problems.

We illustrate the congruences, mentioned in the first section of this chapter on a concrete example, which is presented below. 

\begin{example}
 We are interested in estimating the parameters of a normal distribution based on a set of $N$ independent observations $\{x^{(1)},...,x^{(N)}\}$ of  the parameters of a normal distribution with a density given by:
\[
p(x;\mu,\sigma)= \frac{1}{\sqrt{2 \pi} \sigma} \exp{ - \frac{(x-\mu)^2}{2\sigma^2}}.
\]  
Here, 
$\mu$ represents the mean, and 
$\sigma^2$ denotes the variance of the distribution.

From a geometric perspective, the set of all possible normal distributions (with densities of the form given above) forms an infinite-dimensional manifold, as it includes all probability distributions on the real line that are absolutely continuous with respect to the Lebesgue measure. However, when we restrict our attention to the family of normal distributions parameterized by $(\mu,\sigma)$, we obtain a two-dimensional surface within this infinite-dimensional space. This surface consists of all normal distributions defined on the Borel 
$\sigma-$algebra of the real line.
\end{example}

\subsection{Klein's Congruences and the Universal property}
The following theorem provides a condition under which two families of probability distributions are considered congruent, meaning that they share an underlying structural equivalence in terms of how their probability measures are defined. This  gives rise to a certain notion of {\bf universal property}, since it is reminiscent of certain  universal properties in algebraic geometry. 

\, 
\begin{theorem}[Universal property]\label{T:Univ}
Let us consider two families of probability distributions, denoted as: $ w_1$ and $w_2$ which are defined as

\[w_i=\{P_{\theta}^{{i}}(d\omega^{(i)}), \quad \theta\in \Theta\},\] 

where these families are probability distributions indexed by a parameter $\theta$, meaning that for all $\theta\in \Theta$ we have a corresponding probability measure on the spaces $(\Omega^{1},{\bf S}^{(1)})$ and $(\Omega^{2},{\bf S}^{(2)})$, respectively.

\, 

The two families $w_1$ and $w_2$ are said to be {\bf congruent} if there exists:

\begin{enumerate}
    \item {\bf A measurable mapping:}
$\epsilon=\varphi_{i}(\omega)$ of the measure spaces $(\Omega^{i},{\bf S}^{(i)})$, on which they are defined, onto a {\bf finite} space \[(\mathscr{E},{\bf S}(\mathscr{E})).\]

\item {\bf A common family of probability distributions:}
    \[R_{\theta}(d\epsilon)\] on the space $\mathscr{E}$, which serves as an intermediate representation of the distributions in both families.
    \item {\bf A family of transition distributions:}
     $$Q^{(i)}(\epsilon;d\omega^{(i)})$$
that describes how the probability measure on $\mathscr{E}$ is "transferred" back to the original space $\Omega^{(i)}$.
This relation is formalized as: 
\[P^{(i)}_{\theta}\{\cdot\}=\int_{\mathscr{E}}R_{\theta}\{d\epsilon\}Q^{(i)}(\epsilon,\cdot),\] where the key consistency condition is such that \[ Q(\epsilon,\varphi^{-1}(\epsilon))=1.\]
\end{enumerate}
This ensures that each outcome 
$\omega^{(i)}_{j}$ belongs to the preimage of its corresponding 
$\epsilon$, meaning that 
correctly classifies elements into equivalence classes in 
$\mathscr{E}$.

\[
\begin{tikzcd}
{} & (\mathscr{E},{\bf S}(\mathscr{E})) \arrow{dr}{\varphi^{-1}_2} \\
(\Omega^{1},{\bf S}^{(1)}) \arrow{ur}{\varphi_1} \arrow{rr}{\varphi^{-1}_2\circ \varphi_1} && (\Omega^{2},{\bf S}^{(2)})
\end{tikzcd}
\]
\end{theorem}

\begin{example}
Let us apply this theorem on congruences between families of distributions to a congruence between certain simplices, that contain them. The original families of probability distributions can be described within each simplex using the same barycentric coordinates. This allows us to focus our study on the structure of equivalent families, given by simplices in a special position. Such families will be called maximal.

More precisely, consider two simplices:

\begin{itemize}
    \item The first simplex, denoted as $Cap(\Omega',\mathbf{S'})$,  has vertices ${\bf e}_1,...,{\bf e}_n$ corresponding to the  outcomes   $(\omega_1,...,\omega_n)$.
    \item These outcomes are partitioned into 
$m$ subsets, denoted as $E_1,\cdots, E_m$.
\end{itemize}

For each face of the original simplex 
$Cap$, determined by a subset 
$E_i$ of its vertices, we select a point with barycentric coordinates $P_{E_i}^{(1)}[\omega]$, for each  $\omega \in E_i$ and for each $i\in \{1,\cdots,m\}$. We assume that all but one of these points are chosen in the interior of their respective faces.

\, 

The collection of these points 
$P_{E_i} [.]$ are the vertices of $m$-dimensional forms an 
$m$-dimensional subsimplex. This subsimplex is equivalent to a similarly constructed 
$m$-dimensional subsimplex in the second simplex. Through this construction, we establish a structural equivalence between the two families in terms of their representation within simplicial geometry.

\end{example}

\begin{theorem}
    Let $\Pi(\omega, \mathcal{A})$ be a transition probability distribution from the measure space $(\Omega, \mathbf{S})$ to $(\Omega', \mathbf{\Sigma})$. Then, by extending the probability distribution \[\Pi(\omega, \cdot)\quad \text{on}\quad \mathbf{\Sigma} \quad \text{for each }\quad \omega \in \Omega\]  to a probability distribution \[\Pi^*(\omega, \cdot) \quad \text{on}\quad \mathbf{\Sigma^*},\] we obtain a transition probability distribution from $(\Omega, \mathbf{S^*})$ to $(\Omega', \mathbf{\Sigma^*})$. 
    
    Consequently, for any probability distribution $P$, we have:
    \[
    (P\Pi)^*[\cdot] = (P^*\Pi^*)[\cdot].
    \]
\end{theorem}
\begin{remark}
We make a link between the statement outlined above and the universal property of Theorem \ref{T:Univ}.  Theorem \ref{T:Univ} says that there a {\bf bijective correspondence} is established by the operators $\Pi_{ij}$, which are given by 
\[\Pi_{ij}=\Pi^{(i)}\circ Q^{(j)},\]
where $\Pi^{(i)}$ is given by the mapping $\varphi_i(\omega)$.
\end{remark}

\begin{ex}\label{Ex:ProofThm7.2}
The proof is left as an exercise. However, since it is a very ambitious exercise, we suggest that the reader looks up the definition of an inner and outer measure. 
\end{ex}

\section{Markov Transformations and Associated Subcategories}

The following theorem shows the existence of a natural subcategory, which we denote by $CAP^*$, within the category $CAP$ of measurable spaces and Markov morphisms.

\begin{theorem}[Subcategory of Markov Morphisms]
Let $(\Omega, \mathbf{S})$ be a measurable space whose $\sigma$-algebra $\mathbf{S}$ is closed, and let $CAP$ denote the category of measurable spaces with Markov morphisms. Then the subclass of such spaces, together with all their Markov morphisms, forms a subcategory $CAP^*$ of $CAP$. Moreover, the natural extension of each probability measure from $\mathbf{S}$ to its closure $\mathbf{S}^*$ defines a functor mapping $CAP$ into $CAP^*$.
\end{theorem}

\begin{proof}
By definition, any collection of objects along with all the corresponding morphisms (when these morphisms are suitably restricted) forms a subcategory. The key point is that the operations of extending a probability measure and of applying a Markov morphism commute. Indeed, suppose that for a Markov morphism $\Pi_{01}$ we have a corresponding probability measure $P_\omega\{.\} = \Pi_{01}(\omega; \cdot)$, and let $\Pi_{12}$ be another Markov morphism. Then, for every $\omega \in \Omega$, one may show that
\[
\bigl( \Pi_{01} \circ \Pi_{12} \bigr)^* = \Pi_{01}^* \circ \Pi_{12}^*,
\]
in the sense that the extension of the composite is equal to the composite of the extensions. In particular, 
\[
(P_\omega \Pi_{12})^*\{.\} = (P_\omega^* \Pi_{12}^*)\{.\}.
\]
This commutativity property ensures that the extension procedure defines a functor from $CAP$ to $CAP^*$.
\end{proof}

\begin{remark}
The passage to the closure of the algebra in the category of measurable spaces is functorial.
\end{remark}

We now show that Markov morphisms form a category in a natural way.

\begin{lemma}
Let $\Pi(\omega';\omega'')$ be a Markov morphism from $(\Omega', \mathbf{S}')$ to $(\Omega'', \mathbf{S}'')$, and let $f(\omega'')$ be a measurable real function on $(\Omega'', \mathbf{S}'')$ with 
\[
f(\omega'') \in [0,1] \quad \text{for all } \omega'' \in \Omega''.
\]
Then, the function 
\[
g(\omega') = (\Pi^T f)(\omega') = \int_{\Omega''} \Pi(\omega'; d\omega'') \, f(\omega'')
\]
is a measurable function on $(\Omega', \mathbf{S}')$ and satisfies 
\[
g(\omega') \in [0,1] \quad \text{for all } \omega' \in \Omega'.
\]
\end{lemma}

\begin{proof}
The proof follows from the monotonicity and linearity of the integral, combined with the fact that $\Pi(\omega'; \cdot)$ is a probability measure for each fixed $\omega'$.
\end{proof}

\section{Dominated Families in the Category of Markov Morphisms}

We now state a fundamental theorem concerning dominated families (denoted \textbf{DomFam}) in the context of Markov morphisms.

\begin{theorem}
Relative to the category of Markov morphisms, the property of being dominated is absolutely invariant; that is, if a measure $\mu$ on $(\Omega, \mathbf{S})$ dominates another measure $\nu$ (denoted $\mu \gg \nu$), then for any Markov morphism $\Pi$ from $(\Omega,\mathbf{S})$ to $(\Omega',\mathbf{\Sigma})$ the induced measure satisfies 
\[
\mu\Pi \gg \nu\Pi.
\]
\end{theorem}

\begin{proof}
Let $\mu$ be a nonnegative measure on $(\Omega, \mathbf{S})$ such that $\mu \gg \nu$, meaning that for any set $A \in \mathbf{S}$, if $\mu(A)=0$, then $\nu(A)=0$. Let $\Pi(\omega; d\omega')$ be a transition distribution determining a Markov morphism from $(\Omega,\mathbf{S})$ to $(\Omega',\mathbf{\Sigma})$. 

To show that $\mu\Pi \gg \nu\Pi$, it suffices to prove that for any $\mathcal{A} \in \mathbf{\Sigma}$ with $(\mu \Pi)(\mathcal{A}) = 0$, we have also $(\nu \Pi)(\mathcal{A}) = 0$. For a fixed $\mathcal{A}$, define the set
\[
B = \{ \omega \in \Omega \mid \Pi(\omega; \mathcal{A}) > 0 \}.
\]
It is known that if $f: \Omega \to [0,1]$ is a measurable function with $\int_\Omega f(\omega)\, \mu(d\omega)=0$, then $\mu\{ \omega \mid f(\omega) > 0 \} = 0$. Hence, taking $f(\omega)=\Pi(\omega; \mathcal{A})$, we deduce that
\[
\mu(B)=0.
\]
Since $\mu \gg \nu$, it follows that $\nu(B)=0$. Furthermore, for $\omega \in \Omega \setminus B$, we have $\Pi(\omega; \mathcal{A}) = 0$. Thus,
\[
(\nu \Pi)(\mathcal{A}) = \int_{\Omega} \nu(d\omega) \, \Pi(\omega; \mathcal{A})
= \int_B \Pi(\omega; \mathcal{A}) \, \nu(d\omega) + \int_{\Omega \setminus B} \Pi(\omega; \mathcal{A}) \, \nu(d\omega) = 0.
\]
This completes the proof of the invariance of the domination relation.
\end{proof}

\begin{theorem}
Let $CAP$ denote the category of measurable spaces with Markov morphisms. Then the subclass of dominated measurable spaces $(\Omega, \mathbf{S})$ whose $\sigma$-algebras are closed, together with all their Markov morphisms, forms a complete subcategory $\mathsf{FAMD}$ of $CAP$. Moreover, the collection of all families of mutually absolutely continuous probability distributions forms a complete subcategory $\mathsf{FAMH}$ of $\mathsf{FAMD}$.
\end{theorem}
\begin{proof}
We organize the proof in two parts.

\, 

\textbf{(i) Dominated Families Form a Subcategory.}  
Suppose $Q$ is a probability measure on $(\Omega,\mathbf{S})$ that dominates a family $\{P_\theta\}_{\theta \in \Theta}$, i.e., for each $\theta$, we have $Q \gg P_\theta$. By the previous theorem, for any Markov morphism $\Pi$ from $(\Omega,\mathbf{S})$ to $(\Omega',\mathbf{\Sigma})$, the extended measure $Q\Pi$ dominates the family $\{P_\theta\Pi\}_{\theta \in \Theta}$. Thus, the property of domination is preserved under Markov morphisms, and the dominated families with all their Markov morphisms form a subcategory of $FAM$, which we denote by $\mathsf{FAMD}$.

\, 

\textbf{(ii) Mutual Absolute Continuity.}  
In a family of mutually absolutely continuous probability measures, each measure dominates every other. Since this domination relation is invariant under any Markov morphism, the subcategory $\mathsf{FAMH}$ of such families is complete.
Hence, we conclude that $\mathsf{FAMD}$ is a complete subcategory of $FAM$, and within it, $\mathsf{FAMH}$ forms a complete subcategory.
\end{proof}

\begin{remark}
The process of passing to the closure of the $\sigma$-algebra in the category of measurable spaces is functorial.
\end{remark}

We end this subsection with the following lemma. 
\begin{lemma}\label{L:7.2}
If $P(\cdot)\in CAP(\Omega,{\bf S})$ is a constructive probability distribution (i.e. such that the distribution can be constructed from the family $Q_\theta$
 then all distributions that it dominates are also constructible.  
\end{lemma}

\subsubsection{Conclusion}

In summary, our construction shows that the property of being dominated is invariant under Markov morphisms, and that dominated families and the families of mutually absolutely continuous distributions naturally form complete subcategories of the category of measurable spaces with Markov morphisms.

\subsubsection{Dominating Measures and Constructive Probabilities}

\begin{lemma}
    The measure 
    \[
    P_0\{.\} = \left(\frac{1}{n} \sum_{k=1}^n P_k\right) \{.\}
    \]
    is a dominating law.
\end{lemma}
\begin{ex}\label{Ex:LemmaDom}
Show that for each $P_k$ we have the bound: 
    \[
    P_k\{.\} \leq 2^k P_0\{.\}.
    \]

\end{ex}

\section{Homogeneous Klein Geometries and Geodesic Structures}

\begin{definition}
    A Klein geometry generated by an elementary category of topological spaces is said to be \textit{almost homogeneous} or \textit{quasi-homogeneous} if any two points $a \in A$ and $b \in B$ of any two objects $A$ and $B$ are totally arbitrarily approximable.
\end{definition}

\begin{lemma}
    Consider an almost-homogeneous Klein geometry. Every invariant function in the category of topological spaces over a scalar field is identically constant.
\end{lemma}

\begin{proof}
    The proof is evident.
\end{proof}

\subsection{Geodesic Structures}

\begin{definition}
    Let $\gamma(t)$ be a curve. We say that $\gamma(t)$ is \textit{geodesic} if the family of tangent vectors 
    \[
    X(\tau) = \frac{d}{dt} \Big|_{\tau}
    \]
    is parallel along the curve, that is, for all $t$:
    \[
    \nabla_{\frac{d}{dt}} \left(\frac{d}{dt}\right)_{p(t)} = 0.
    \]
\end{definition}

\begin{remark}
    Geodesics are the analogs of straight lines in Euclidean spaces.
\end{remark}

In local coordinates, the geodesic equation takes the form:
\[
\frac{d^2 x^{(k)}}{dt^2} + \sum_{i,j} \Gamma_{ij}^k \frac{dx^{(i)}}{dt} \frac{dx^{(j)}}{dt} =0.
\]

Since parametrization is essential, any other canonical parametrization $s = s(t)$ of $\gamma(t)$ must be an affine function of $t$, that is, $s(t) = at + b$.

\subsection{Torsion, Curvature, and Flat Connections}

The linear connection operator $\nabla$ is associated with two multi-linear operators $T(X,Y)$ and $R(X,Y)$:
\begin{itemize}
    \item The torsion operator:
    \[
    T(X,Y) = \nabla_{X}(Y) - \nabla_{Y}(X) - [X,Y].
    \]
    \item The curvature operator:
    \[
    R(X,Y)Z = \nabla_{X}(\nabla_{Y}Z) - \nabla_Y(\nabla_{X}Z) - \nabla_{[X,Y]}Z.
    \]
\end{itemize}
where $[X,Y] = XY - YX$ is the commutator.

\begin{remark}
    If both the torsion and curvature vanish identically, then:
    \begin{itemize}
        \item The connection is flat.
        \item A manifold with such a flat connection is locally affine.
        \item Any parallel translation on such a manifold depends only on its initial and final points.
        \item Parallel translation depends only on the homotopy class of paths.
    \end{itemize}
\end{remark}

\subsection{Totally Geodesic Submanifolds}

\begin{definition}
    A submanifold $N$ of a manifold $M$ equipped with a linear connection $\nabla$ is called \textit{totally geodesic} if, for any two points of $M$, it contains the whole geodesic passing through these points.
\end{definition}

\begin{lemma}
    A submanifold $N \subset M$ is totally geodesic if, for any tangent vector fields $X, Y$ on $N$, the vector field $Z = \nabla_X Y$ is also tangent to $N$.
\end{lemma}

\begin{proof}
    The proof is evident.
\end{proof}

\begin{remark}
    A submanifold $N \subset M$ is totally geodesic if every locally shortest curve in $N$ is also locally shortest in $M$.
\end{remark}

\section{Equivariant Connections in the Category of Topological Spaces}

Let us return to the framework of Klein geometry and consider a category of topological spaces, where morphisms reflect the fundamental structures of geometric transformations.

\begin{definition}
    Let $\mathcal{C}$ be a category of topological spaces, and let $\nabla$ be a linear connection on each object of $\mathcal{C}$. The connection $\nabla$ is said to be \textit{absolutely equivariant} if, for any morphism $\varphi: X \to Y$ in $\mathcal{C}$, the image of every geodesic in $X$ under $\varphi$ is a geodesic in $Y$.
\end{definition}

This definition encapsulates the idea that a natural connection should be preserved under morphisms of the category, ensuring that geodesic structures remain invariant under transformations.

\subsection{Geodesic Midpoints and Absolute Covariance}

To better understand the implications of an absolutely equivariant connection, let us recall the fundamental properties of geodesics in canonical coordinates.

\begin{itemize}
    \item Given two points on a geodesic, one can uniquely define the \textit{midpoint} of the geodesic segment connecting them.
    \item This midpoint is determined by a symmetric function of the two endpoints.
    \item If the connection $\nabla$ is absolutely equivariant in the category $\mathcal{C}$, then the midpoint function is necessarily \textit{absolutely covariant} with respect to pairs of points.
\end{itemize}

Thus, the notion of geodesic symmetry emerges naturally in the categorical framework: an equivariant connection ensures that midpoint operations behave consistently across all objects and morphisms of the category.

\chapter{Categorical and Geometric structures in Statistical Manifold Theory}

In the previous chapter, we introduced key ideas in probability and statistics and briefly alluded to a more structural perspective using category theory and differential geometry. However, these tools have not yet been systematically employed. In this chapter, we take a decisive step in that direction by developing a structured approach relating to category theory and differential geometry in the study of statistical manifolds.

\, 

The categorical viewpoint allows us to formalize probabilistic transformations, treating probability distributions as objects and Markov kernels as morphisms. At the same time, differential geometry provides powerful tools for analyzing the geometric structure of statistical manifolds, particularly through notions such as connections, curvature, and geodesics. By combining these perspectives, we gain deeper insight into the intrinsic structure of statistical models and their transformations.

\, 

This chapter is devoted to making these ideas explicit. We will establish the categorical framework for statistical manifolds and then show how geometric methods naturally arise within this setting. The interplay between these two formalisms will not only clarify fundamental structures but also pave the way for further developments in information geometry, learning theory, and beyond.

\, 

In order to rigorously define the notion of statistical manifolds, we introduce five fundamental collections of probability distributions, which serve as building blocks for our categorical formulation:

\begin{itemize}
    \item The collection $\mathsf{Cap}$,
    \item The collection $\mathsf{Capd}$,
    \item The collection $\mathsf{\mathsf{Caph}}$,
    \item The collection $\mathsf{\mathsf{Conh}}$,
    \item The collection $\mathsf{Var}$.
\end{itemize}

These collections are formally defined as follows:

\begin{definition}
    Let $(\Omega, \mathbf{S})$ be a measurable space. We define:
    \begin{itemize}
        \item $\mathsf{Cap}(\Omega, \mathbf{S})$ as the collection of all probability distributions on $(\Omega, \mathbf{S})$.
        
        \item $\mathsf{Capd}(\Omega, \mathbf{S}, \mathbf{Z})$ as the collection of all probability distributions on $(\Omega, \mathbf{S})$ that vanish on the ideal $\mathbf{Z}$.
    
        \item $\mathsf{\mathsf{Caph}}(\Omega, \mathbf{S}, \mathbf{Z})$ as the collection of all probability distributions on $(\Omega, \mathbf{S})$ that vanish on $\mathbf{Z}$ and only on $\mathbf{Z}$, when considered over an open simplex.
        
        \item $\mathsf{\mathsf{Conh}}(\Omega, \mathbf{S}, \mathbf{Z})$ as the collection of all nonnegative, mutually absolutely continuous probability measures on $(\Omega, \mathbf{S})$ that vanish on $\mathbf{Z}$ and only on $\mathbf{Z}$.
      
        \item $\mathsf{Var}(\Omega, \mathbf{S})$ as the collection of all probability distributions on $(\Omega, \mathbf{S})$.
    \end{itemize}
\end{definition}

The following result formalizes the correspondence between these spaces:

\begin{proposition}
    The collections $\mathsf{Var}(\Omega, \mathbf{S})$ and $\mathsf{Cap}(\Omega, \mathbf{S})$ are in one-to-one correspondence.
\end{proposition}

\begin{proof}
    The claim follows from the structural equivalence of probability distributions in both settings.
\end{proof}

\begin{proposition}
    The collections $\mathsf{Cap}, \mathsf{Capd}, \mathsf{\mathsf{Caph}}, \mathsf{\mathsf{Conh}}, \mathsf{Var}$ each form a manifold equipped with a corresponding atlas.
\end{proposition}

\begin{proof}
    See Chentsov, pp. 73-74.
\end{proof}

\subsection{The Category of Statistical Decision Rules}

A statistical problem is naturally associated with the measurable space $(\Omega, \mathbf{S})$ of sample outcomes, together with:
\begin{itemize}
    \item A family $\{P_{\theta}\}$ of admissible probability distributions on $(\Omega, \mathbf{S})$,
    \item An additional structure encoding prior information about the phenomenon,
    \item A measurable space $(\Omega', \mathbf{S'})$ of inferences (actions),
    \item A statistical test quantifying the quality of a given decision rule.
\end{itemize}

\begin{definition}
    A \emph{statistical decision rule} is a transition probability distribution $\Pi(\omega, d\epsilon)$ describing a Markov random transition from the sample space $(\Omega, \mathbf{S})$ to the inference space $(\Omega', \mathbf{S'})$.
\end{definition}

\begin{remark}
    Any transition probability distribution $\Pi(\omega, d\epsilon)$ may be interpreted as a decision rule within any statistical model whose sample space is $(\Omega, \mathbf{S})$ and whose inference space is $(\Omega', \mathbf{S'})$.
\end{remark}

Furthermore, given a Markov random transition from $\Omega'$ to $\Omega''$ described by a transition probability $\Pi(\omega', d\omega'')$, we obtain a \emph{Markov morphism} between the measurable spaces $(\Omega', \mathbf{S'})$ and $(\Omega'', \mathbf{S''})$.

\begin{theorem}
    The class of objects:
    \begin{itemize}
        \item The collection $\mathsf{Cap}(\Omega, \mathbf{S})$,
        \item The collection $\mathsf{Capd}(\Omega, \mathbf{S}, \mathbf{Z})$,
        \item The collection $\mathsf{\mathsf{Caph}}(\Omega, \mathbf{S}, \mathbf{Z})$,
        \item The collection $\mathsf{\mathsf{Conh}}(\Omega, \mathbf{S}, \mathbf{Z})$,
    \end{itemize}
    equipped with the system of Markov homeomorphisms defined by
    \[
    (P\Pi) \{.\} = \int \Pi(\omega;.) P\{d\omega\},
    \]
    forms the categories of statistical decisions $\mathsf{CAP}, \mathsf{CAPD}, \mathsf{\mathsf{Caph}}, \mathsf{\mathsf{Conh}}$, respectively, which are isomorphic to categories of statistical decision rules.
\end{theorem}

\begin{proof}
    The proof follows from the categorical structure imposed by the Markov homeomorphisms and the natural compatibility of the decision rule formulation.
\end{proof}

\subsection{Types of Manifolds in Theory of Statistical Manifolds}

In the theory of statistical manifolds, four types of manifolds are of particular interest:
\begin{enumerate}
    \item manifold $\mathsf{\mathsf{Conh}}$,
    \item manifold $\mathsf{\mathsf{Caph}}$,
    \item manifold $\mathsf{Capd}$
    \item manifold $\mathsf{Cap}$.
\end{enumerate}

     \subsubsection{Characteristics of the Manifold $\mathsf{\mathsf{Conh}}$} 
      
Consider the category $\mathsf{Caph}F$ of collections $\mathsf{\mathsf{Caph}}(\Omega,\mathbf{S},\mathbf{Z})$ with finite quotient algebras $\mathbf{S}/ \mathbf{Z}$. These collections are finite-dimensional manifolds. 
Any $\mathbf{Z}$-dominated measure on $(\Omega, \mathbf{S})$, where $\mathbf{S}/ \mathbf{Z}$ is finite, is completely determined by the vector
\[\boldsymbol{\mu}=(\mu\{A_1\},\mu\{A_2\},\cdots,\mu\{A_m\})=(\mu_1,\mu_2,\cdots,\mu_m).\]

\begin{definition}
Let $\mathsf{\mathsf{Conh}}(\Omega,\mathbf{S},\mathbf{Z})$ be the collection of all nonnegative, mutually absolutely continuous measures on $(\Omega,\mathbf{S})$ that vanish on $\mathbf{Z}$-sets and only there. Then if $S_m \simeq \mathbf{S}/\mathbf{Z}$, the correspondence
\[
\mu \{.\}  \longleftrightarrow \boldsymbol{\mu} = (\mu\{A_1\},\mu\{A_2\},\cdots,\mu\{A_m\}),
\]
and
\[
\mu_j > 0,\quad j\in\{1,\cdots,m\}.
\]
defines a chart of the entire cone $\mathsf{\mathsf{Conh}}$. This chart is called a \emph{natural chart}.
The atlas of charts ${\mathsf{Conh}}$ also includes other charts obtained from the natural chart through analytic or infinitely differential coordinate transformations. However, the natural chart holds a privileged position.
\end{definition}
\begin{proposition}
   The collection $\mathsf{Caph}(\Omega,\mathbf{S},\mathbf{Z})$ is selected from $\mathsf{Conh}(\Omega,\mathbf{S},\mathbf{Z})$
by the condition
\[
<\mu,I>=1,
\]
making it a hypersurface. Specifically, $\mathsf{Caph}(\Omega, \mathbf{S}, \mathbf{Z})$ is the intersection of the cone  $\mathsf{Conh}(\Omega,\mathbf{S},\mathbf{Z})$ where
$\mu_j > 0$ for $j=1,\cdots,m$, where $\mu_j= \mu\{A_j\}$, with the hyperplane defined by $<\mu,I>=1$.

The vectors $\mathbf{p}=(P\{A_1\},\cdots,P\{A_m\})$ corresponding to ${\bf P\{\cdot\}}$, exhaust the interior of the unit simplex 
exhaust the interior of the unit simplex
\[
\sum_{j=1}^m p_j\mathbf{e}_j, 
\]
where $p_j\geq 0$, $\sum_{j=1}^m p_j=1$, and  $\mathbf{e}_1=(1,0,\cdots,0),\mathbf{e}_2=(0,1,0,\cdots,0),\cdots, \mathbf{e}_m = (0,0,0,\cdots,0,1)$ are the vertices of the simplex.
\end{proposition}
\begin{definition}
The probabilities $P\{\mathbf{A}_j\}=p_j\{P\}$ simultaneously serve as:

\begin{itemize}
    \item {\bf Natural coordinates}: The bijection

\[
\mathbf{P} \longleftrightarrow \mathbf{p}=(p_1,\cdots,p_m),
\] identifies points in the ambient space.
     \item {\bf Barycentric coordinates}: Affine coordinates relative to the simplex vertices ${\bf e}_1,\cdots {\bf e}_m$, where 
     $\mathbf{p}=\sum_{j=1}^mp_j{\bf e}_j$.
\end{itemize}
\end{definition}
\subsection{Non-Independence of Natural Coordinates}
\begin{remark}
Let  $\mathsf{Caph}$  be a manifold and consider a surface within it. The natural coordinates $\mathbf{p}=(p_1,\cdots,p_m)$ are not local coordinates of manifold $\mathsf{Caph}$.
\end{remark}
\begin{proof}
 The coordinates  $p_1,\cdots,p_m$ are linearly dependent on the surface since $\sum_{j=1}^mp_j$. This reduces the intrinsic dimension to 
$m-1$, rendering one coordinate redundant.
 \end{proof}
 \subsection{Constructing Local Coordinates}
To define a local coordinate system for $\mathsf{Caph}$, we discard one component of {\bf p}.

For example, omitting $p_m$, we obtain the chart :
\[
P\{A_j\} > 0, j=1,\cdots,m, \quad  \sum_{j=1}^m P\{A_j\} =1,\quad \Longrightarrow\, P\{A_m\}=1-\sum_{j=1}^{m-1} P\{A_j\}.
\]
However, such charts lack invariance under permutations of the atoms $A_1,\cdots, A_m$. To resolve this:

\subsubsection*{Homogeneous Coordinates}
Introduce equivalence classes under scaling:
\[
\mathbf{P} \longleftrightarrow c \cdot \mathbf{p},\quad c>0,
\]
where {\bf p} is defined up to a positive multiplicative constant. This ensures permutation invariance.

\subsubsection*{Rectilinear Coordinates}

Another type of coordinates, which can be used to solve the problem of chart invariance, is to introduce the rectilinear coordinates.
Let  $f_j:\Omega\to \bbR$, $(j=0,\cdots,m-1)$ be ${\bf S}$-measurable functions with $f_0(\omega)=I(\omega)$. Assume the matrix $\bigg(f_j(\omega_i)\bigg)$ 
has full rank $m$. Define coordinates:

\[
t_j(P)= \mathbf{M}_P f_j{\omega}= \mathbb{E}_P[f_j] = \langle P,f_j \rangle, \quad (j= 1,\cdots,m-1),
\]

where $\mathbf{M}_Pf(\omega)$ is an expectation of a random variable $f(\omega)$, defined as a following integral:
\[
\mathbf{M}_P f(\omega)= \int_{\Omega} f(\omega)P\{d\omega\}= \langle P, f \rangle,
\]
where the function $f$ must be integrable with respect to the measure $P$ (or quasi-integrable).

\subsubsection*{Natural Coordinate System}
\begin{definition}
Let $\mathsf{Caph}$ be an $(m-1)$-dimensional manifold.
A local coordinate system $\mathbf{t}= (t_1,\cdots,t_{m-1})$ for $\mathsf{Caph}$ is called \emph{natural} if it is induced by the expectations:
    \[
    \mathbf{t} = \int_{\Omega} f(\omega)P\{d\omega\},
    \]
    and such that 
    \[
    t_j(P) =  \int_{\Omega} f_j(\omega)P\{d\omega\}, \quad j=1,\cdots, m-1,
    \]
  where $f_j(\omega)$ are linearly independent functionals and each $f_j$ is $P$-quasi-integrable.
\end{definition}

\subsection{Canonical Affine Coordinates in the Theory of Statistical Manifolds}

The introduction of canonical affine coordinates in statistical manifold theory arises naturally from fundamental geometric and probabilistic considerations. Let us first define the indicator function 
\[
\epsilon_j(\omega) = 
\begin{cases} 
1, & \text{if } \omega \in A_j, \\
0, & \text{otherwise}.
\end{cases}
\]
These functions are not necessarily linearly independent; however, any $m-1$ of them, together with the identity function $I(\omega)$, form a linearly independent set.

\begin{definition}
    A coordinate system $\mathbf{s}=(s^1,\dots,s^{m-1})$ on the open simplex $\mathsf{Caph}(\Omega, \mathbf{S}, \mathbf{Z})$ is called \emph{canonical affine} if
    \[
    P_{\mathbf{s}} \{ .\} = T(\mathbf{s}) P_0\{ .\},
    \]
    where $P_0$ is the origin and $\mathbf{s}$ is a covariant linear coordinate system of the group of translations.
\end{definition}

More explicitly, if $P_0\{.\}$ represents the probability distribution at the origin and the functions $g_0(\omega) = 1$, $g_1(\omega),\dots,g_{m-1}(\omega)$ define the coordinate axes, then:
\[
\frac{dP_{\mathbf{s}}}{d\mu}(\omega) 
= p(\omega;\mathbf{s}) = p(\omega, 0)\exp \left[\sum_{\alpha=1}^{m-1} s^{\alpha} g_{\alpha} (\omega) - \Psi(\mathbf{s})  \right].
\]
Here, $\mu\{.\}$ is an arbitrary $\mathbf{Z}$-positive measure vanishing on $\mathbf{Z}$-sets only, and the normalization factor $\exp[\Psi(\mathbf{s})]$ is given by:
\[
\exp[ \Psi(\mathbf{s})] = \int_{\Omega}\exp \left[ \sum_{\alpha = 1}^{m-1}  s^{\alpha} g_{\alpha} (\omega) \right] P_0\{ d \omega\}.
\]

\subsubsection{Tangent Vectors and Tangent Spaces}
To analyze the structure of statistical manifolds, we introduce tangent vectors and vector fields. Let $\{P_t\}$ be a family of probability distributions parameterized by a smooth curve $\mathbf{x}(t)$. The functional $(Y)_{P}$ of smooth functions $f(\mathbf{x})$ is defined as:
\[
\frac{d}{dt}f(\mathbf{x}(t)) \bigg|_{t =\theta},
\]
which represents a tangent vector at the point $P=P_{\theta}$.

\begin{lemma}
    \begin{itemize}
        \item The curve $P_t$ admits the following decomposition:
        \[
        P_t{.} = P_{\theta}{.} + \tau P'_{\theta}{.}+ o(\tau),
        \]
        where $\tau = t - \theta$, $P'_{\theta}$ is a charge of bounded variation with total measure zero, i.e., $P'_{\theta} {\Omega} =0$, and the norm of the remainder term vanishes faster than $\tau$.
        \item There exists a one-to-one, linear, continuous, and analytic correspondence between tangent vectors at a given point and charges with zero total measure.
    \end{itemize}
\end{lemma}

\begin{proof}
    Consider the representation $P_t\{.\} \leftrightarrow \mathbf{p}(t) = (p_1(t),\dots,p_{m-1}(t),p_m(t))$, where $p_m(t)= 1 - p_1(t) - \dots - p_{m-1}(t)$. Near $t = \theta$, we expand:
    \[
    p_j(t) = p_j(\theta) + \tau \dot{p}_j(\theta) + o(\tau).
    \]
    Adding these expansions and subtracting from unity, we obtain the corresponding expression for $p_m$. Defining the auxiliary charges:
    \[
    P'_{\theta} \{H\} = \sum_{A_j \preceq H} \dot{p}_j (\theta),
    \]
    we obtain:
    \[
    P_t\{H\}= P_{\theta}\{H\}+ \tau P'_{\theta}\{H\} + R_{t,\theta}\{H\}.
    \]
    Using the chain rule,
    \[
    \frac{d}{dt}f(\mathbf{p}(t))\bigg|_{t= \theta} = \sum_{j=1}^{m-1} \frac{\partial f}{\partial p_j} \bigg|_{\mathbf{p}(\theta)} \dot{p}_j(\theta).
    \]
    The result follows from continuity and the existence of fiberings of the simplex into smooth trajectories.
\end{proof}

\subsection{Vector Fields and Ceva Lines}

In classical differential geometry, vector fields are typically expanded in a coordinate basis $X_i = \frac{\partial}{\partial x^i}$. However, in statistical manifold theory, this approach is often inconvenient due to the transformation properties of probability distributions. Instead, we construct an alternative basis.

We define $m$ vector fields $Y_j$ such that:
\[
(Y_j)_{\mathbf{p}} \leftrightarrow  \mathbf{e}_j  - \mathbf{p}.
\]
Additionally, we introduce $m$ vector fields $X_j$ satisfying:
\[
X_j= p_j Y_j, \quad (X_j)_{\mathbf{p}} \leftrightarrow  p_j \mathbf{e}_j - p_j \mathbf{p}.
\]

To understand the structure of these vector fields, consider the simplex and an arbitrary segment within it. Among all possible segments, we identify those connecting a vertex $\mathbf{e}_i$ to a point on the opposite face. These segments, known as \emph{Ceva lines}, define a special class of trajectories in the simplex.

\subsection{Ceva Lines and Tangent Vector Fields}

Let $\mathbf{p}(t; q_2,\dots,q_m)$ be a Ceva line. More precisely, we define it as
\[
\mathbf{p}(t;q_2,\dots,q_m)= p_1(t)\mathbf{e}_1 + [1 - p_1(t)] \sum_{i=2}^m q_i \mathbf{e}_i.
\]

\begin{lemma}
The vector fields $X_i = \frac{\partial}{\partial x^i}$, where $(x^i,\dots,x^n)$ is the field of differentiation along the Ceva line with respect to the parameter $t$, satisfy the differential equation
\[
\dot{p}_i(t) = p_i(t) - [p_i(t)]^2 = p_i(t) [1 - p_i(t)].
\]
In particular, for $i=1$, we obtain
\[
p_1(t)= \frac{\exp\{t\}}{\exp\{t\}+1} = \frac{1}{1+\exp\{-t\}}.
\] 
\end{lemma}

\begin{remark}
    The vector field $Y_1$ corresponds to the choice of parameter $p_1(t) =1 - \exp\{ -t \}$. The remaining fields $X_j$
    and $Y_j$ are obtained analogously along the Ceva lines:
    \[
    \mathbf{p}^{(j)}(t) = p_j(t) \mathbf{e}_j + [ 1 - p_j(t)] \sum_{i \neq j} q_i \mathbf{e}_i,
    \]
    where $p_j(t)$ for $X_j$ satisfies the differential equation
    \[
    \dot{p}_j (t) = p_j (t) - [ p_j(t)]^2.
    \]
    Similarly, for the field $Y_j$, we have $p_j(t)= 1 - \exp\{-t\}p$.
\end{remark}

\begin{proof}
We compute $\dot{p_i}(t)$, which defines the tangent vector of differentiation with respect to $t$. For $i=1$, we have
\[
\dot{p}_1 (t)= \frac{\exp\{ - t\}}{(1+ \exp \{-t\})^2}= \frac{1}{1 + \exp \{ -t\}}- \frac{1}{(1 + \exp \{ -t\})^2} = p_1 (t) - [p_1(t)]^2.
\]
Additionally, we obtain
\[
\dot{p}_1(t)= - \dot{p}_1(t)q_i = -p_1(t) [1 - p_1(t)]q_i = - p_1 (t)  p_i (t).
\]
Since differentiation with respect to $t$ satisfies $\frac{d}{dt} \longleftrightarrow p_1(t) \mathbf{e}_1 - p_1(t)\mathbf{p}(t)$, we conclude the proof of the lemma. 
\end{proof}

\begin{lemma}
    The system of tangent vectors \[(X_j)_{\mathbf{p}},\quad j= 1,\dots,m\] is complete at each point of the manifold $\mathsf{Caph}$. That is, any tangent vector $(Z)_{\mathbf{p}}$ may be expanded in terms of the $(X_j)_{\mathbf{p}}$, and any $m-1$ vectors $(X_j)_{\mathbf{p}}$ form a basis. 
    
    \, 
    
    Every vector field $Z$ admits a unique smooth expansion:
    \[
    Z= \sum_{j=1}^m \zeta^j(\mathbf{p})X_j,
    \]
    subject to the additional condition
    \[
    \sum_{i=1}^m \zeta^i p_i = 0.
    \]
    All other expansions of the form $Z= \sum_{j=1}^m \eta^j X_j$ satisfy
    \[
    \eta^i(\mathbf{p})= \zeta^i(\mathbf{p})+ \phi(\mathbf{p}),
    \]
    where $\phi(\mathbf{p})$ is a scalar field.
\end{lemma}

\begin{remark}
    If we replace the condition $\sum_{i=1}^m \zeta^i p_i = 0$ by 
    \[
    \sum_{j=1}^m \zeta^j = 0,
    \]
    then the above lemma remains valid.
\end{remark}

\begin{proof}
Let $(Z)_{\mathbf{p}} \leftrightarrow \mu\{\cdot\}$, where $\mu\{\Omega\}=0$. Then
\[
\mathbf{\mu} = \sum_{j=1}^m \mu_j \mathbf{e}_j = \sum_{j=1}^m \mu_j \mathbf{e}_j - \sum_{j=1}^m \mu_j \mathbf{p} 
= \sum_{j=1}^m \mu_j (\mathbf{e}_j - \mathbf{p}) 
= \sum_{j=1}^m \frac{\mu_j}{p_j} p_j(\mathbf{e_j} - \mathbf{p})
\longleftrightarrow  \sum_{j=1}^m \frac{\mu_j}{p_j}(X_j)_{\mathbf{p}}.
\]

Since $(X_j)_{\mathbf{p}} \longleftrightarrow  p_j \mathbf{e_j} - p_j \mathbf{p}$, we obtain:
\[
\sum _{j=1}^m (X_j)_{\mathbf{p}} \longleftrightarrow  \sum_{j=1}^m  (p_j\mathbf{e_j} - p_j \mathbf{p}).
\]
Thus,
\[
\sum_{j=1}^m (X_j)_{\mathbf{p}}  \longleftrightarrow \sum_{j=1}^m p_j {\mathbf{e}_j} - \mathbf{p}\sum_j^m p_j 
  = \mathbf{p} -\mathbf{p} = \mathbf{0}.
\]
Since $\sum_{j=1}^m X_j = {\mathbf{0}}$, it follows that $X_k = - \sum\limits_{j \neq k} X_j$. Inserting this expression into $(Z)_{\mathbf{p}}$, we obtain an expansion in a basis of $m-1$ vectors.

Let $\sum\limits_j \zeta^j X_j = Z = \sum\limits_j \eta^j X_j$ be two expansions of $Z$. Then
\[
\sum_j[\zeta^j - \eta^j]X_j =\mathbf{0},
\]
or equivalently,
\[
(\zeta^1 - \eta^1)X_1 = - \sum_{i=2}^m (\zeta ^i - \eta^i)X_i.
\]
Since $X_2,\dots,X_m$ form a basis, we deduce:
\[
\zeta^1 - \eta^1 = \zeta^2 - \eta^2 =\dots= \zeta^m - \eta ^m = \phi(\mathbf{p}).
\]
From
\[
\sum_{j=1}^m \zeta^j (\mathbf{p})p_j = \sum_{j=1}^m \zeta^j(\mathbf{p})p_j - \phi(\mathbf{p}),
\]
we see that $\sum_{j=1}^m \zeta^j(\mathbf{p}) p_j=0$ if and only if $ \sum_{j=1}^m\zeta^j (\mathbf{p}) p_j = \phi(\mathbf{p})$.
Thus, the lemma is proved.
\end{proof}

\subsection{Affine Connections and Ceva Lines on Statistical Manifolds}

Let us now examine the interplay between the vector fields $X_j$ and canonical affine coordinates.

\begin{lemma}
In a canonical affine coordinate system aligned with the directions $\epsilon_j(\omega)$ for $j\neq k$, the coordinate curves parametrized by 
\[x^i= t,\quad  x^j = \text{const}, (j \neq i, k),\]

correspond to the Ceva lines of the $i$-th family. The tangent vector fields associated with these curves are given by \[X_i = \frac{\partial}{\partial x^i}.\]
\end{lemma}

\begin{proof}
    Consider the probability measure $P_0$ at the center of the simplex:
    \[
    P_0 (\cdot) \leftrightarrow  \left(\frac{1}{m}, \dots, \frac{1}{m}\right).
    \]
    Let $k=m$ and integrate with respect to $\mu$ over the atom $A_i$ in the expression
    \[
    \frac{d P_s}{d\mu} (\omega) = p(\omega;s) = p(\omega; 0)
    \exp\left\{\sum_{\alpha=1}^{m-1}s^{\alpha}g_{\alpha}(\omega)- \Psi(s)\right\},
    \]
    where $\mu$ is an arbitrary positive measure vanishing only on $\mathbb{Z}$-sets. Noting that $\epsilon_j(\omega) = 0$ on $A_i$ for $j\neq i$, we obtain 
    \[
    p_i (\mathbf{x}) = P_{\mathbf{x}}\{ A_i\} =P_0 \{ A_i\} \exp\left\{x^i - \Psi(\mathbf{x})\right\}
    \]
    for $i < m$,
    and
    \[
    p_m(\mathbf{x}) = P_0 \{ A_m\} \exp\{- \Psi(\mathbf{x})\},
    \]
    where
    \[
    \exp{\Psi(\mathbf{s})} = \int_{\Omega} \exp\left[\sum_{\alpha=1}^{m-1} s^{\alpha}g_{\alpha}(\omega)\right] P_0\{ d \omega\}.
    \]
    Thus,
    \[
    \exp[ \Psi(x) ] = 
    \sum_{j=1}^{m-1} \frac{1}{m} \exp[x^j].
    \]

    Consider the coordinates $(x_0, x^1, x^2,\dots, x^{m-1})$ and let only $x^1$ vary, that is, $x^1 = x_0^1 +t$, with $P_{x(t)}= R_t$ and $b_i(t) = B_t\{ A_i\}$, while $x^2, \dots, x^{m-1}$ remain fixed.

    Then we have
    \[
    b_1(t) = b_1 (0) \exp [ t - \Phi(t) ], \quad b_i(t) = b_i (0) \exp [ - \Phi (t)],
    \]
    where
    \[
    \Phi (t) = \Psi(x_0^1 + t, x_0 ^2, \dots, x_0^{m-1}) - \Psi (x_0^1, x_0 ^2,\dots,x_0^{m-1}).
    \]
    Summing over $i=2,\dots,m$, we obtain
    \[
    p(t; q_2,\dots,q_m) = p_1(t)\mathbf{e}_1 +[1 - p_1(t)]\sum_{i=2}^m q_i \mathbf{e}_i.
    \]
    This confirms that the coordinate line is a Ceva line.

It remains to calculate $b_1(t)$. Substituting our expression for $\exp{- \Phi(t)}$, we obtain
    \[
    \frac{b_1(t)}{b_1(0)} = \frac{1 - b_1(t)}{1-b_1(0)} \exp{t},
    \]
    and differentiating with respect to $t$,
    \[
    \dot{b}_1(t) = b_1(t) - [b_1(t)]^2.
    \]
    This establishes the required result. The remaining cases follow by permutation of atoms and coordinates.
\end{proof}

\section{Translation-Invariant Tensor Fields}

\begin{theorem}
A tensor field defined on the simplex $\mathsf{Caph}$ is invariant under the group of translations if and only if its components are constant with respect to the the canonical coordinate system.
\end{theorem}

\begin{proof}
Consider the canonical coordinate system $(z_1, z_2,\dots,z_{m-1})$. In these coordinates, the simplex becomes an $(m-1)$-dimensional affine space, where translations correspond to parallel shifts.

Since the group of parallel translations is both transitive and simply transitive, we conclude that carrying the components of a tensor unchanged across all points yields an invariant field.
Let $f(\mathbf{p})=c$ be an invariant scalar field. For any point $\mathbf{q}$, there exists a unique translation mapping $\mathbf{p}$ to $\mathbf{q}$, denoted by $T\mathbf{p}=\mathbf{q}$. Since $f^T = f$, we obtain
    \[
    f(\mathbf{q})= f^T(\mathbf{q})= f(T^{-1}(\mathbf{q}))= f(\mathbf{p})=c.
    \]
    Hence, $f(\mathbf{q})$ is necessarily constant.

    Now, let
    \[
    (Y)_{\mathbf{p}}= \sum_{j=1}^{m-1} a^j (\mathbf{p})\frac{\partial}{\partial z^j}\Big|_{\mathbf{p}}.
    \]
    Under translation, canonical coordinates shift as
    \[
    T: z^j \longrightarrow z^j + s^j,
    \]
    which implies
    \[
   T:  \left(\frac{\partial}{\partial z^j}\right)_{\mathbf{p}} \longrightarrow \left(\frac{\partial}{\partial z^j}\right)_{T\mathbf{p}}.
    \]
    If $Y$ is invariant, then
    \[
    \sum a^j (\mathbf{p})\frac{\partial}{\partial z^j} = \sum a^j (T\mathbf{p})\frac{\partial}{\partial z^j}.
    \]
    Since the $\frac{\partial}{\partial z^j}$ form a basis at each point, it follows that
    \[
    a^j (\mathbf{p})= a^j (\mathbf{q}).
    \]
    Hence, the theorem is proved.
\end{proof}

\begin{corollary}
    On the manifold $\mathsf{Caph}$, there exist translation-invariant Riemannian metrics that convert $\mathsf{Caph}$ into a Euclidean space.
\end{corollary}

\section{Flat Connections}

\begin{definition}
    A connection is said to be \textit{flat} if
    \[
    \nabla_{X_j}X_k = 0, \quad j,k=1,2,\dots,m.
    \]
\end{definition}

\begin{remark}
    The flat connection defined above is compatible with any translation-invariant Riemannian metric.
\end{remark}

\section{Exercises}
Some easy exercises. 
\begin{itemize}
\item Exercise 1. Determine the affine transformation which maps the points $(0,0)$, $(1,0)$ and $(0,1)$ to the points
\begin{itemize}
    \item $(0,-1)$, $(1,1)$, and $(-1,1)$, respectively,
    \item $(-4,-5)$, $(1,7)$, and $(2,-9)$, respectively.
\end{itemize}
\item Exercise 2.The triangle $\bigtriangleup ABC$  has vertices $A(2,0)$, $B(-3,0)$ and $C(3,-3)$, and the points $P(-1,-1)$, $Q(1,3)$ and $R(-\frac{1}{4},0)$ lie on $BC$, $CA$, and $AB$ respectively.
\begin{itemize}
    \item Determine the ratios in which $P$, $Q$, and $R$ divide the sides of the triangle,
     \item Determine whether or not the points $P$, $Q$,and $R$ are collinear.
\end{itemize}
\item Exercise 3. Consider a statistical manifolds of exponential type. Using the charts description proof that the structure of statistical manifold is affine and flat.
\end{itemize}

\,
\subsection{Some corrections}
Exercise \ref{Ex:ProofThm7.2}, correction.
\begin{proof}
    The proof is divided into three parts.  
 \subsection*{First part of the Proof}   
    Firstly, since any measure on $\mathbf{\Sigma}$ can be extended to a measure on $\mathbf{\Sigma^*}$, where it coincides with the induced inner and outer measures, it remains to show that for any fixed $\mathcal{A} \in \mathbf{\Sigma^*}$, the transition probability function $\omega \mapsto \Pi^*(\omega, \mathcal{A})$ is $\mathbf{S^*}$-measurable.
    
\subsection*{Second part of the Proof}
Having established the first part of the proof, we now proceed to the second step. Our aim is to show that the real-valued function  
\[
\Pi_{\mathcal{A}}(\omega) = \Pi(\omega; \mathcal{A})
\]
is $\mathbf{S^*}$-measurable for any fixed $\mathcal{A} \in \mathbf{\Sigma^*}$. To do so, it suffices to verify that for any probability measure $P$ on $\mathbf{S}$ and any real number $z \in [0,1]$, the outer and inner $P$-measures of the preimage 
\[
U = \{ \omega \mid \Pi_{\mathcal{A}}(\omega) < z\}
\]
of the half-open interval $[0, z)$ coincide, i.e., 
\[
\overline{P} [U ] = \underline{P} [U ],
\]
which ensures that $U$ is indeed $\mathbf{S^*}$-measurable.

Consider probability measures $P$ and $Q$. For any $\mathcal{A} \in \mathbf{\Sigma^*}$, there exist two sets $\mathcal{G}, \mathcal{F} \in \mathbf{\Sigma}$ satisfying:
\[
\mathcal{F} \subset \mathcal{A} \subset \mathcal{G},
\]
such that for every $\omega$, the following equalities hold:
\[
Q[\mathcal{G}] = \overline{Q} [\mathcal{A}] = \underline{Q}[\mathcal{A}] = Q[\mathcal{F}].
\]

Since probability measures are monotone, we obtain the inequalities:
\[
\Pi(\omega; \mathcal{G}) \geq \Pi^*(\omega; \mathcal{A}) \geq \Pi(\omega; \mathcal{F}).
\]
Therefore, the functions $\Pi(\omega; \mathcal{G})$ and $\Pi(\omega; \mathcal{F})$ are $\mathbf{S}$-measurable and equal almost everywhere,
since
\[
\int \Pi(\omega; \mathcal{G}) P\{d\omega\} = \int \Pi(\omega; \mathcal{F}) P\{d\omega\},
\]
except possibly on a $\mathbf{S}$-measurable $P$-null set $N$ where $P\{N\} = 0$.

From this, it follows that:
\[
\{ \omega \mid \Pi(\omega; \mathcal{G}) < z \} \subseteq \{ \omega \mid \Pi^* (\omega; \mathcal{A}) < z \}.
\]
Moreover, we also have:
\[
\{ \omega \mid \Pi^*(\omega; \mathcal{A}) < z \} \subseteq \{ \omega \mid \Pi(\omega; \mathcal{G}) < z \} \cup N.
\]
Thus, we conclude:
\[
\overline{P} [ U ] = \underline{P} \{ \omega \mid \Pi(\omega; \mathcal{G}) < z \},
\]
which proves the second part.

\subsection*{Third part of the Proof}

Finally, we proceed to the third and last step: proving that the measures $(P\Pi)^*$ and $P^*\Pi^*$, both defined on $(\Omega',\mathbf{\Sigma^*})$, coincide. That is, we claim that for any $\mathcal{A} \in \mathbf{\Sigma^*}$, we have:
\[
(P \Pi)^* \{\mathcal{A}\} = (P^* \Pi^*)\{\mathcal{A}\}.
\]

Let us choose sets $\mathcal{F}$ and $\mathcal{G}$ satisfying for a given  $\mathcal{A}$:
\[
\mathcal{F}  \subseteq \mathcal{A} \subseteq \mathcal{G}.
\]
Then, as before, we have:
\[
Q\{\mathcal{G} \} = Q^*\{ \mathcal{A}\} = Q\{ \mathcal{F} \},
\]
where $Q = P \Pi$. Given that the function $\Pi_{\mathcal{A}}^*(\omega) = \Pi^*(\omega; \mathcal{A})$ is $\mathbf{\Sigma^*}$-measurable, we obtain:
\[
(P\Pi)\{ \mathcal{G} \} = \int \Pi(\omega; \mathcal{G}) P[d\omega] =
\int \Pi^*(\omega; \mathcal{G}) P^*[d\omega].
\]
By monotonicity, we conclude:
\[
\int \Pi^*(\omega; \mathcal{G}) P^*[d\omega] \geq \int \Pi^*(\omega; \mathcal{A})P^*[d\omega].
\]
Thus, we deduce:
\[
(P^* \Pi^*)\{ \mathcal{A} \} = (P\Pi)\{\mathcal{G}\}.
\]
Since this holds for all $\mathcal{A} \in \mathbf{\Sigma^*}$, we obtain the desired equality:
\[
(P^* \Pi^*)\{\mathcal{A} \} = (P \Pi)^* \{ \mathcal{A}\}.
\]
This completes the proof.

\end{proof}
\, 

Exercise \ref{Ex:LemmaDom}.

    The measure $P_0$ is dominating since for each $P_k$ we have the bound 
    \[
    P_k\{.\} \leq 2^k P_0\{.\}
    \]
    and the inclusion of null sets:
    \[
    \mathbf{Z}_{P_0} \subset \mathbf{Z}_{P_k}.
    \]
    If we can show that $P_0$ is a constructive probability measure, the desired conclusion follows from Lemma \ref{L:7.2}. 

    Consider a measurable mapping $\omega = f_k(x)$, where $f_k : \mathbf{E} \to \Omega$ determines the constructive distribution 
    \[
    P_k\{.\} = \lambda \frac{1}{f_k(.)}.
    \]
    Define $f_0$ piecewise as follows: for $2^{-k} < x < 2^{-k+1}$, let
    \[
    f_0(x) = g_k(x) = f_k (2^k x -1),
    \]
    and set $f_0(0) = f_1(0)$. 

    This function is measurable, as it is composed of countably many measurable functions, each defined on a distinct measurable subset. Consequently, the $f_0$-preimage of any $\mathbf{S}$-set is a countable union of disjoint $g_k$-preimages. 

    Furthermore, we have the measure transformation:
    \[
    \lambda g_k^{-1}\{.\} = 2^{-1} \lambda f_k^{-1}\{.\},
    \]
    since the $g_k$-preimage is the $f_k$-preimage shifted by $2^{-k}$ and contracted by a factor of $2^k$. This completes the proof.

\part{Frobenius Manifolds and Information Geometry}

\chapter{Frobenius Manifolds}
\begin{center}
\includegraphics[scale=0.4]{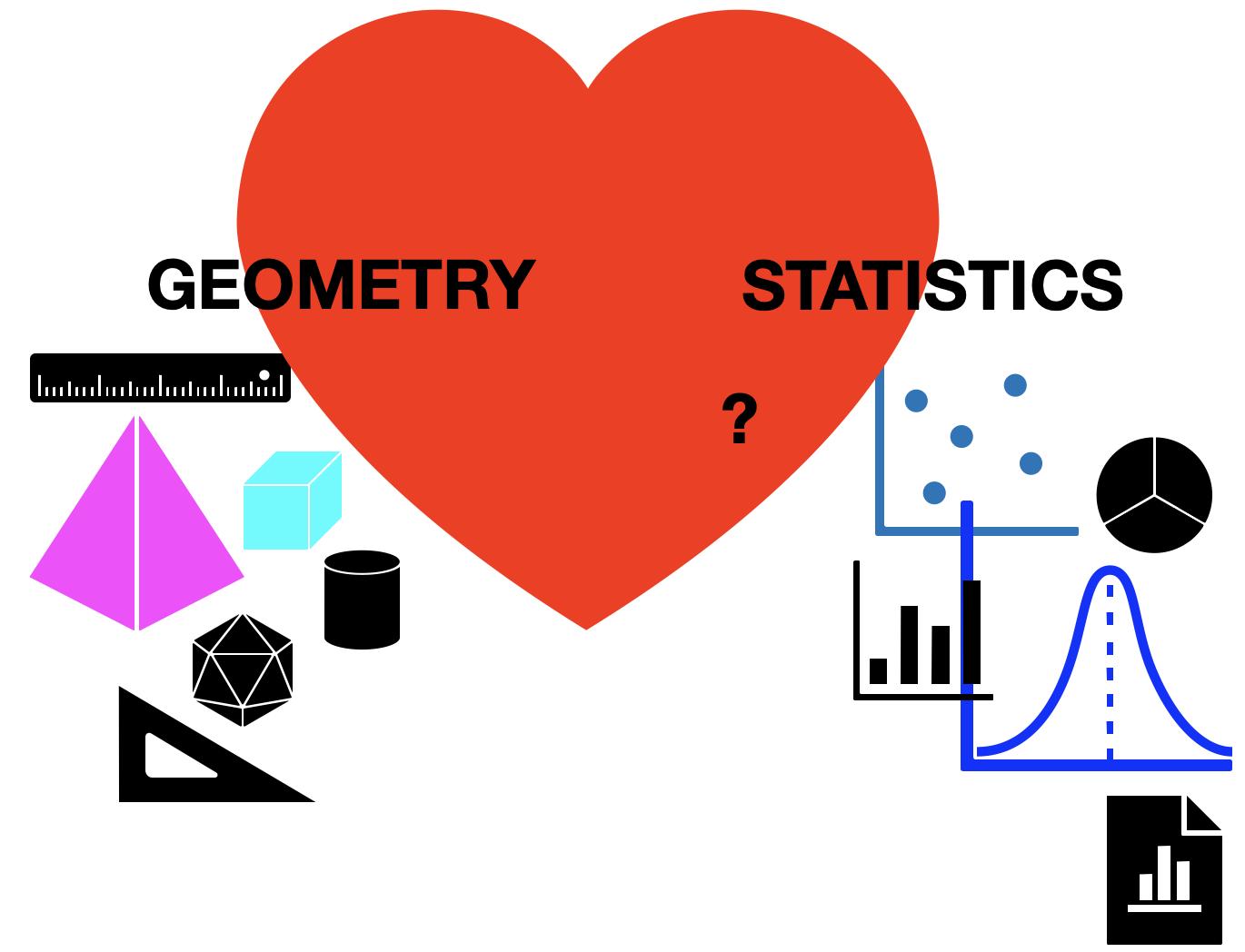}
  \end{center}

\section{A short story of Frobenius manifolds}

A Frobenius manifold $\cM$ is a geometric realization of a certain class of partial nonlinear differential equations of degree $3$. These equations, commonly referred to as the WDVV equations—named after Witten, Dijkgraaf, Verlinde, and Verlinde—are also known in some mathematical contexts as the Associativity Equations.

\,

There is a subtle distinction between Frobenius manifolds and the WDVV equations. While the notion of Frobenius manifolds represents a geometric object, the WDVV equations belong to the realm of analysis. The terminology “Frobenius manifolds” and “Associativity equations” appears predominantly in mathematical literature, whereas “WDVV equations” is more frequently encountered in discussions closer to physics.

\,

These perspectives, however, are deeply interwoven. As shown by Yu. I. Manin \cite{Man99} and B. Dubrovin \cite{Du}, they ultimately coincide.

\,

In the context of physics, solutions of the WDVV equations encode the moduli spaces of topological conformal field theories. These solutions are pivotal in the formulation of mirror symmetry for Calabi--Yau $3$-folds. Notably, specific solutions to the WDVV equations with particular properties serve as generating functions for the Gromov--Witten invariants of Kähler or symplectic manifolds.

\,
\subsection{Homological Mirror Symmetry (and more!)}
In the 1990s, the concept of Frobenius manifolds gained traction among both algebraic and analytic geometers due to the discovery of three significant classes, all stemming from the profound interplay between mathematics and physics. This development led to the creation of advanced mathematical tools and generated an abundance of challenging and intriguing problems. Some of these arose naturally from attempts to provide a rigorous mathematical interpretation of mirror symmetry (e.g., Homological Mirror Symmetry). Others emerged from the ambition to axiomatize and deepen the mathematical understanding of topological quantum field theory.

\,

The first class of Frobenius manifolds originated in the study of singularities. Specifically, Saito spaces—predating the explicit definition of Frobenius manifolds—were associated with the unfolding spaces of singularities. A particularly straightforward example is the space of complex polynomials of a fixed degree $d$ with distinct roots, which can also be identified with the configuration space of $d$ marked points on the complex plane.

\,

Subsequent classes discovered in the 1990s have a formal nature. Examples include the (formal) moduli spaces of solutions to Maurer–Cartan equations modulo gauge equivalence and the formal completions of cohomology spaces of smooth projective (or compact symplectic) manifolds. The latter is commonly referred to as quantum cohomology.

\,

More recently, a new class of Frobenius manifolds has emerged, demonstrating the versatility of this structure. It has been shown that under certain conditions, the manifold of probability distributions can generate a Frobenius manifold (this is the Combe–Manin construction). This discovery inspires many open questions, particularly regarding the complete classification of Frobenius manifolds, the identification of novel classes, and the elucidation of the intricate relationships between the existing ones.

\,

In the first section of this chapter, we provide a concise survey of Frobenius manifolds. We introduce the definitions of Frobenius, pre-Frobenius, and potential pre-Frobenius manifolds and briefly outline the WDVV equations. We also touch upon the hidden algebraic structures intrinsic to Frobenius manifolds. This section concludes with a brief summary.

\,

In the second section, we delve more deeply into the theoretical framework of Frobenius manifolds.

\,

In the third section, we present a collection of exercises and open problems for the reader to explore.

\section{Preliminary Notions: Affine Structures}

An essential ingredient for the definition of Frobenius manifolds is the concept of affine flat structures. For simplicity, we shall often refer to these as {\it affine structures}.

\,

\subsection{}
Let $\cM$ be a smooth manifold of dimension $\cM$. Affine flat structures can be defined in multiple equivalent ways.

\,
\subsubsection{}
An affine structure on an $n$-dimensional manifold $\cM$ is defined via a collection of coordinate charts $\{(U_\alpha, \phi_\alpha)\}$, where $\{U_\alpha\}$ forms an open cover of $\cM$, and $\phi_\alpha: U_\alpha \to \mathbb{R}^n$ is a local coordinate system such that the transition functions $\phi_\beta \circ \phi_\alpha^{-1}$ are affine transformations on $\phi_\alpha(U_\alpha \cap U_\beta)$, mapping it to $\phi_\beta(U_\alpha \cap U_\beta)$.

\,

Recall that the group of affine transformations is given by:
\[
\left\{
\begin{pmatrix}
A & b \\ 
0 & 1
\end{pmatrix}
\; \bigg\vert \; A \in GL(n,\mathbb{R}), \; b \in \mathbb{R}^n
\right\}.
\]

\,

\begin{definition}
An affine manifold is a smooth manifold equipped with an affine structure.
\end{definition}

\,
\subsubsection{}
The existence of an affine flat structure on $\cM$ is equivalent to the presence of a specific class of connections on the tangent bundle of $\cM$. Namely, there is a bijective correspondence between affine flat structures and flat, torsion-free affine connections $\nabla$ on $\cM$.

\,

In the context of differential geometry, a manifold over $\mathbb{K}$ with an affine structure is characterized by a tangent bundle whose underlying $G$-structure corresponds to the group of affine transformations $Aff(n) = GL(n, \mathbb{K}) \rtimes \mathbb{K}^n$, where $GL(n, \mathbb{K})$ is the general linear group over the field $\mathbb{K}$.

\,
\subsubsection{}
Alternatively, an affine flat structure can be described by a subsheaf $\cT_\cM^f \subset \cT_\cM$ of linear spaces of pairwise commuting vector fields. Locally, one has the tensor product over the ground field:
\[
\cT_\cM = \mathcal{O}_M \otimes \cT_\cM^f.
\]
Sections of $\cT_\cM^f$ correspond to flat vector fields. Furthermore, the metric $g$ is compatible with the structure $\cT_\cM^f$ if $g(X, Y)$ is constant for all flat vector fields $X$ and $Y$ (Exercises $12.1$ and $12.2$).

\,
\subsection{Crystallographic groups and fundamental groups}
A group $\Lambda$ is called an $n$-crystallographic group if it contains a normal, torsion-free, maximal abelian subgroup of rank $n$ and finite index. Crystallographic groups fit into the short exact sequence:
\[
0 \to V \to \Lambda \to P \to 1,
\]
where $V$ is a complex vector space, and $P \leq GL(n, \mathbb{Z}) \cong \mathrm{Aut}(V)$ is a finite group acting faithfully on $V$.

\,

A complex crystallographic group is a discrete subgroup $\Lambda \subset \mathrm{Iso}(\mathbb{C}^n)$ such that $\mathbb{C}^n / \Lambda$ is compact, where $\mathrm{Iso}(\mathbb{C}^n)$ denotes the group of biholomorphisms preserving the standard Hermitian metric.

\,

We note the following property:
\begin{lemma}
The fundamental group of a compact, complete flat affine manifold is an affine crystallographic group.
\end{lemma}

\,

Among the crystallographic groups, one encounters the Bieberbach groups, which are torsion-free crystallographic groups.
\

\subsection{Pre-Lie structures}
This affine structure leads to interesting algebraic properties on the tangent bundle/tangent sheaf. Let $\Gamma(TM)$ denote the space of vector fields on a given manifold $\cM$.

\,

According to the previous chapters, a flat and torsionless connection satisfied the following. 
An affine connection $\nabla$ is {\it torsion-free} (or torsionless) if
\begin{equation}\label{E:1}
\nabla_X (Y) - \nabla_Y (X) - [X, Y] = 0.
\end{equation}
The connection is called {\it flat} if
\begin{equation}\label{E:2}
\nabla_X (Y) - \nabla_Y (X) - \nabla_{[X, Y]} = 0.
\end{equation}

Such a connection defines a covariant differentiation for vector fields $X, Y \in \Gamma(TM)$:
\[
\nabla_X: \Gamma(TM) \longrightarrow \Gamma(TM), \quad 
\nabla_X(Y) \longmapsto \nabla_X(Y).
\]

\,

\begin{ex}\label{Ex:preLie}
A pre-Lie algebra is an algebra satisfying the relation 
\[
a(bc) - b(ac) = (ab)c - (ba)c.
\]    
Show that one can recover the structure of a pre-Lie algebra on the space of vector fields on $\cM$, by putting  $X \circ Y := \nabla_X(Y)$ for an affine flat, torsionless connection $\nabla$.
\end{ex}

\,

\begin{example}
Let $V$ be a finite-dimensional real Euclidean space endowed with a real inner product, and let $C$ be a closed convex cone $C \subset V $ with its vertex at the origin. The polar of the cone is defined by
\[
C^* = \{ y \in V \; | \; \langle x, y \rangle > 0, \; \forall x \in C \}.
\]
A symmetric cone satisfies $C^* = C $.

Such symmetric cones carry an affine flat structure. The tangent sheaf of this cone is equipped with the structure of a pre-Lie algebra, defined by the operation $X \circ Y := \nabla_X(Y) $, where $X, Y $ are vector fields on $C$.
\end{example}

\subsection{2-Dimensional Affine Structures}

\,

Let us consider two well-known structures in topology: the torus and the Klein bottle. These topological objects are the only compact $2$-dimensional manifolds that admit Euclidean structures.

\,

Assume $\cM$ is a closed manifold different from the torus or Klein bottle. Then there exists {\it no} affine structure on it. This is a direct consequence of the following result:

\begin{theorem}[Benzecri, 1955]
A closed surface admits affine structures if and only if its Euler characteristic vanishes.
\end{theorem}

\subsection{$n$-dimensional affine structures, $n \geq 3$}

When dealing with manifolds of dimension greater than two, there is no definitive criterion for determining the existence of an affine structure.

\,

In particular, according to Smillie's theorem, a closed manifold does \textit{not admit} an affine structure if its fundamental group is built up out of finite groups by taking free products, direct products, and finite extensions. Specifically, a connected sum of closed manifolds with finite fundamental groups admits \textit{no affine structure}. This result provides a profound insight into the interplay between algebraic and geometric properties in the study of such manifolds.

\,

Certain Seifert fiber spaces also do not possess affine structures. This is clarified by the following statement:

\begin{proposition}[Y. Carri\`ere, F. Dal’bo, G. Meigniez]
Let $\cM$ be a Seifert fiber space with vanishing first Betti number. Then, $\cM$ does not admit any affine structure.
\end{proposition}

\subsection{Complex cases}
We discuss the complex case. Due to works of Kobayashi \cite{Ko1}, a classification has been established. 
\begin{theorem}
In the complex case, one has the following collection of compact complex manifolds that admit affine structures.  
\begin{enumerate}
    \item the complex tori,
    \item the hyperelliptic surfaces,
    \item the minimal elliptic surfaces with odd $b_1,p_g>0$ and $c_1^2=0$,
    \item the minimal surfaces with $b_1=1$, $b_2=0$ and $p_g=0$ with the exception of the Hopf surfaces which are covered by primary Hopf surfaces satisfying certain specific properties. 
\end{enumerate}
\end{theorem}

However, it is interesting to note that there is {\it no other} compact complex surface admitting even holomorphic affine connections.

\section{Pre-Frobenius manifolds}
The pre-Frobenius manifolds exhibit interesting relations to the Monge--Ampère domains. Before we outline such relations, we recall the definition of such manifolds, using the tools that have been previously introduced.  

\subsection{Affine structure and metric}
Let us consider an affine flat structure on a manifold $\cM$. To define such a structure, several ingredients are required:
\begin{itemize}
    \item An atlas with transition functions that are affine and linear (in the affine case).
    \item A metric $g$ that is compatible with the affine flat structure.
    \item A symmetric tensor of rank $3$, denoted as:
    \[
    A: S^3(\cT_\cM) \longrightarrow \mathbb{R}.
    \]

\end{itemize}

\subsection{Multiplication Operation } 
We define a multiplication operation $\circ$ on the tangent sheaf $\cT_\cM$.

 Define a bilinear symmetric multiplication $\circ = \circ_{A, g}$ on the tangent sheaf $\cT_\cM$ as follows:
\[
\cT_\cM \otimes \cT_\cM \longrightarrow S^2(\cT_\cM) \xrightarrow{\ \mathcal{A}' \ } {\cT}^* \xrightarrow{\ g' \ } \cT_\cM,
\]
such that:
\[
X \otimes Y \longrightarrow X \circ Y,
\]
where the prime denotes partial dualization.

\subsection{Compatibility Relations}
A compatibility relation between the rank-$3$ tensor $A$, the rank-$2$ tensor $g$, and the multiplication operation $\circ$ is given by:
    \[
    A(X, Y, Z) = g(X \circ Y, Z) = g(X, Y \circ Z).
    \]

This invariance of the metric with respect to multiplication ensures that the structure is well-defined.

\, 

\begin{definition}
    A pre-Frobenius manifold is a manifold $\cM$ equipped with the above  properties.
\end{definition}

Certain additional requirements on the algebraic structure of the tangent sheaf $(\cT_\cM, \circ)$ lead to having a \emph{Frobenius manifold.} 
\, 

There are however two important axioms to keep under consideration and that we shall express below. 

\subsection{Potential pre-Frobenius manifolds}

An important axiom to have is the one of potentiality. Namely, this axiom requires the existence of a family of local potentials $\Phi$, such that:
\[
g(X \circ Y, Z) = g(X, Y \circ Z) = (XYZ) \Phi.
\]
This axiom is particularly important regarding the relations to the Monge--Ampère equations.

If $ \mathscr{D}$ is a strictly convex bounded subset of $\mathbb{R}^n$ then for any nonnegative function $f$ on $\mathscr{D}$ and continuous $\tilde{g}:\partial  \mathscr{D} \to \mathbb{R}^n$ there is a unique convex smooth function $\Phi\in C^{\infty}( \mathscr{D})$ such that 
\begin{equation}\label{E:EMA}
\det \mathrm{Hess}(\Phi)= f, 
\end{equation} in $D$ and $\Phi=\tilde{g}$ on $\partial \mathscr{D}$.

\, 

An elliptic Monge--Ampère equation domain refers to the geometric data generated by $(\mathscr{D}, \Phi)$, where \begin{itemize}
    \item $\mathscr{D}$ is a strictly convex domain 
    \item $\Phi$ a real convex smooth function (with arbitrary and smooth boundary values of $\Phi$)  
    \end{itemize}
    such that Eq.~\eqref{E:EMA} is satisfied.

We state the following result: 
\begin{theorem}
  A potential pre-Frobenius manifold satisfies everywhere locally the Monge--Amp\`ere equation. In other words, a potential pre-Frobenius manifold can be identified with an (elliptic) Monge--Amp\`ere domain.
\end{theorem}

\begin{ex}\label{Ex:Ma}
    Make a proof of the statement above.
\end{ex}

\subsection{Associative pre-Frobenius manifolds}
An associative pre-Frobenius manifold is a pre-Frobenius manifold such that there exists an associativity property \[(X \circ Y)\circ Z=X\circ(Y\circ Z),\]
where $X,Y,Z$ are vector fields. We will discuss this axiom fully in the context of Frobenius manifolds, below.  A Frobenius manifold is a pre-Frobenius manifold where the axioms of potentiality and associativity both hold. 
\section{Frobenius Manifolds}
\subsection{Witten-Dijkgraaf-Verlinde-Verlinde equation}

To derive the Witten-Dijkgraaf-Verlinde-Verlinde (WDVV) equation, let us rewrite the associativity of the multiplication $\circ$:
\[
(\partial_a \circ \partial_b) \circ \partial_c = \partial_a \circ (\partial_b \circ \partial_c).
\]
This yields a non-linear system of associativity equations, which are partial differential equations for the potential $\Phi$. For all $a, b, c, d$, these equations are written as:
\[
\sum_{ef} \Phi_{abe} g^{ef} \Phi_{fcd} = \sum_{ef} \Phi_{bce} g^{ef} \Phi_{fad}.
\]

These equations are highly non-linear and of third order.

\subsection{Geometrization}

Let $\cM$ be a manifold. A Frobenius algebra $(\mathcal{A}, \circ)$ over a field $\mathbb{K}$ is a commutative, associative, and unital algebra with a multiplication operation $\circ$ equipped with a symmetric bilinear form $\langle -, - \rangle$ satisfying:
\[
\langle x \circ y, z \rangle = \langle x, y \circ z \rangle,
\]
for all $x, y, z \in \mathcal{A}$.

\begin{definition}
A Frobenius manifold is an associative potential pre-Frobenius manifold.
\end{definition}

\, 

A manifold $\cM$ admits the structure of a Frobenius manifold if:
\begin{itemize}
    \item At any point of $\cM$, the tangent space has the structure of a Frobenius algebra $\mathcal{A}$.
    \item The invariant inner product $\langle -, - \rangle$ defines a flat metric on $\cM$.
    \item The unity vector field satisfies $\nabla e = 0$.
    \item The tensor of rank 4, $(\nabla_W A)(X, Y, Z)$, is fully symmetric, where $X,Y,Z,W$ are vector fields.
    \item A vector field $E$, called the Euler field, exists on $\cM$ such that $\nabla(\nabla E) = 0$.
\end{itemize}

The Euler field $E$ belongs to the class of affine vector fields. Its existence is inherently tied to the affine structure on $\cM$. This can be observed through the equivalence of the following statements:
\begin{itemize}
    \item[(1)] $E$ is an affine vector field.
    \item[(2)] $\nabla(\nabla E) = 0$.
    \item[(3)] For all vector fields $Y, Z$ on $\cM$:
    \[
    \nabla_Y(\nabla_Z E) = \nabla_{\nabla_Y Z} E.
    \]
    \item[(4)] The coefficients of $E$ are affine functions. Writing $E = \sum_m E^m \partial_m$, we have:
    \[
    E^m = a^m_j x^j + b^m,
    \]
    where $a^m_j$ and $b^m$ are constants in $\mathbb{R}$.
\end{itemize}

Thus, we propose a more concise definition of Frobenius manifolds based on Frobenius bundles. This new approach offers a practical and geometrical perspective.


\begin{remark}
If we choose local flat coordinates $(x^a)$ and the corresponding local basis of tangent fields $\partial_a$, then:
\[
(\partial_a \circ \partial_b \circ \partial_c) \Phi = \partial_a \partial_b \partial_c \Phi,
\]
and the compatibility of $\Phi$ and $g$ implies:
\[
\partial_a \circ \partial_b = \sum \Phi_{ab}^c \partial_c,
\]
where
\[
\Phi_{ab}^c := \sum (\partial_a \partial_b \partial_c \Phi) g^{ec}.
\]
Here,
\[
g_{ab} := (g_{ab})^{-1},
\]
with $g_{ab}$ interpreted as the inverse metric tensor.
\end{remark}

\subsection{Structure Connections of Pre-Frobenius manifolds}
Consider a pre-Frobenius manifold given by a triple $(\cM, g, A)$. Define the following geometrical objects:
\begin{itemize}
    \item A connection:
    \[
    \nabla_0: \cT_\cM \longrightarrow \Omega_M^1 \otimes \cT_\cM,
    \]
    where $\nabla_0$ is determined by the condition that flat fields are $\nabla_0$-horizontal.
    \item A pencil of connections depending on a parameter $\lambda$:
    \[
    \nabla_{\lambda, X}(Y) := \nabla_{0, X}(Y) + \lambda (X \circ Y).
    \]
    This is called the \textit{structure connection} of $(\cM, g, A)$.
\end{itemize}

\begin{theorem}[Manin]\label{T:Manin}
Let $(M, g, A)$ be a pre-Frobenius manifold. 
Let $\nabla_{\lambda}$ be the structure connection of a pre-Frobenius manifold $(\cM, g, A)$. $(\cM, g, A)$ is a Frobenius manifold if and only if the pencil $\{\nabla_\lambda\}$ is flat.
\end{theorem}

\begin{ex}\label{Ex:ManinProof}
Prove the Theorem \ref{T:Manin}.
\end{ex}

\section{Emergence of Hidden Structures for Statistical Manifolds}

We give a glimpse of how in statistical manifolds, it is possible to unravel the structures of a Frobenius algebra on the tangent space. This is only a short explanation that will be further developed in the next chapters, and serves only as a guide giving and intuition behind the construction.

\, 

Let $\bar{T} = T \cdot g^{-1}$ denote the mixed $(1, 2)$ tensor of third rank (that is the 1 contravariant and 2 covariant tensor). In components, this is defined by:
\[
\bar{T}_{ij}^k = \sum_m g^{km} T_{ijm}.
\]
where $g$ is the metric tensor, compatible with the affine connection on the manifold under consideration. 

We illustrate the construction of the operation $\circ$, defined on $\cT_\sfS$, where $\sfS$ is a statistical manifold of exponential type and of finite dimension. 

Generally:
\[
\bar{T}_{ij}^k = \bar{T} \big|_{P_{\theta}} (\partial_i \ell_{\theta}, \partial_j \ell_{\theta}, a^k) = \mathbb{E}_{P_{\theta}} [\partial_i \ell_{\theta} \partial_j \ell_{\theta} \partial_k \ell_{\theta} a_{\theta}^k].
\]
where $\{a^i\}$ form a dual basis to $\{\partial_j\ell_{\theta}\}$.
The hidden multiplication structure is explained by the following theorem:

\begin{theorem}
The tensor $\bar{T}$ defines a multiplication $\circ$ on $\cT_{P_{\theta}} \sfS$, as follows:
\[
\bar{T}: \cT_{P_{\theta}}\sfS \times \cT_{P_{\theta}}\sfS \longrightarrow \cT_{P_{\theta}}\sfS,
\]
and for $u, v \in \cT_{P_{\theta}}\sfS$:
\[
u \circ v = \bar{T}(u, v).
\]
\end{theorem}

The following lemma aids in understanding this hidden multiplication structure:

\begin{lemma}
For $u, v, w \in \cT_{P_{\theta}}\sfS$:
\[
g(u \circ v, w) = g(u, v \circ w).
\]
\end{lemma}

\subsection{Summary of Key structures} 

We sumarize the key elements existing for pre-Frobenius manifolds, in the example of statistical manifolds. This is as follows:
\begin{itemize}
    \item An affine flat structure. This structure exists for statistical manifolds of exponential type (see Exercise $12.3$).
    \item A triple $(\sfS, g, T)$, where $\sfS$ is a finite-dimensional statistical manifold:
    \begin{itemize}
        \item $g$ is a Riemannian (Fisher) metric,
        \item $T$ is the rank-3 symmetric Amari–Chentsov tensor,
        \item A multiplication defined by:
        \[
        u \circ \nu = \bar{T}(u, \nu),
        \]
        where locally, 
        \[
        \bar{T}_{ij}^k = \sum g^{km} T_{ijm},
        \]
            \end{itemize}
        \item Metric invariance:
        \[
        g(u, \nu \circ w) = g(u \circ \nu, w), \text{for flat vector fields } u,v,w.
        \]
        \item The flatness condition, which requires that the affine space of connections $\nabla_{\lambda}$ is flat.

\end{itemize}

\section{Semisimple Frobenius Manifolds}
We take as our point of departure the classical notion of a semisimple Frobenius manifold, formulated within the established framework. Semisimple Frobenius manifolds are interesting in relation to configuration spaces and Saito spaces. 

\,

This preliminary exposition, serves as a scaffolding upon which a more intrinsic and geometrically transparent reformulation will be erected. The latter will emerge naturally from a careful reconsideration of the underlying structures, unveiling with clarity the interplay between the metric, the associative multiplication, and the coherence conditions that bind them into a unified whole.

\subsection{}
 
Let $(\cM, g, A)$ be a triple, where $\cM$ is an associative pre-Frobenius manifold of dimension $n$. We introduce the following definition, which will play a fundamental role in the structure theory of such manifolds:

\begin{definition}
The manifold $\cM$ is said to be \emph{semisimple} (or \emph{split semisimple}) if there exists an isomorphism of sheaves of $\mathcal{O}_M$-algebras
\[
(\cT_\cM, \circ) \xrightarrow{\sim} (\mathcal{O}_M^n, \cdot).
\]
Here, $\circ$ denotes the multiplication on $\cT_\cM$, while $\cdot$ represents componentwise multiplication in $\mathcal{O}_M^n$. The isomorphism is required to exist everywhere locally (or globally).
\end{definition}

Let $(e_1, e_2, \dots, e_n)$ be a local basis of $\cT_\cM$. In this local basis, the multiplication takes the form:
\[
\left( \sum f_i e_i \right) \circ \left( \sum g_j e_j \right) = \sum f_i g_i e_i.
\]
In the simplest case, this reduces to
\[
e_i \circ e_j = \delta_{ij} e_i.
\]
Thus, the basis $(e_i)$ provides a well-defined (up to renumbering) family of idempotents. If $\cM$ is semisimple, there exists an unramified covering of $\cM$ (of degree at most $n!$) on which the induced pre-Frobenius structure becomes a splitting structure.

\, 

\subsection{}
Let $(e_i)$ be a local coordinate basis and let $(\epsilon^i)$ be its dual basis i.e. $1$-forms. 
The structure tensor $A$, encoding the pre-Frobenius structure, is given by the condition:
\[
A(X, Y, Z) = g(X \circ Y, Z) = g(X, Y \circ Z),
\]
where $g$ is the flat metric and 
$\circ$ denotes the associative product on $\cT_\cM$.
Since the basis vectors satisfy $e_i \circ e_j = \delta_{ij} e_i$, it follows, by a direct application of the above identity, that
\[
g(e_i, e_j) = g(e_i \circ e_i, e_j) = g(e_i, e_i \circ e_j) = \delta_{ij} g(e_i, e_i).
\]
We introduced the notation $\eta_i$, for the diagonal components of the metric,
so that $\eta_i=g(e_i, e_i)$. 

With this notation, the three-tensor $A$ takes the form:
\[
A = \sum_{i=1}^n \eta_i (\epsilon^i)^3,
\]
exhibiting its  diagonalizability in a chosen coordinate system.

Finally, considering the identity element $e = \sum_{i=1}^n e_i$ in $(\cT_\cM, \circ)$, we obtain the corresponding co-identity:
\[
\mathcal{E} = \sum_{i=1}^n \eta_i \epsilon^i.
\]

As a consequence of the above discussion, we are able to reformulate our initial definition differently. 
\subsection{Definition: semisimple Frobenius structure}
\begin{definition}
A \emph{semisimple Frobenius structure} on a smooth manifold  $\cM$ consists of the following data:
\begin{itemize}
    \item[1.] A reduction of the structure group of the tangent sheaf $\cT_\cM$;
    \item[2.] A flat metric $g$, diagonal in a distinguished  basis $(e_i)$ and its dual $(\epsilon^i)$;
    \item[3.] A diagonal cubic tensor $A$, sharing the same coefficients as $g$.
\end{itemize}
\end{definition}
\begin{remark} It is interesting to note that while the conditions of potentiality and flatness imposed on 
$g$ are of a non-trivial nature,  the associativity of the product structure on $\cT_\cM$ follows automatically under these hypotheses.
\end{remark}

\subsection{}
We are now in a position to articulate a fundamental characterization of Frobenius structures within this formalism. This characterization, which encapsulates the essential interplay between the multiplication, the metric, and their compatibility conditions, will serve as a guiding principle in the subsequent development of the theory.

\begin{theorem}
\label{thm:3.3}
The semisimple pre-Frobenius structure on $\cM$ defines a Frobenius structure if and only if:
\begin{itemize}
    \item The vector fields $e_i = \frac{\partial}{\partial u^i}$, in a system of canonical coordinates $\{u^i\}$, satisfy $[e_i, e_j] = 0$. We have $\epsilon^i = du^i$ in this system of canonical coordinates $\{u^i\}$;
    \item The functions $\eta_i$ satisfy the relation
$\eta_i = e_i \eta$, for some local function $\eta$ uniquely determined up to the addition of a constant. Equivalently, the form $\mathcal{E}$ is closed.
\end{itemize}
\end{theorem}

The function $\eta$ is referred to as the \textit{potential metric} of the structure. This metric corresponds to a Hessian metric of the form:
\[
g_{ij} = \frac{\partial^2 \Phi}{\partial u^i \partial u^j},
\]
where $\Phi$ is the potential. Note that the canonical coordinates $\{u^i\}$ are defined up to renumbering and constant shifts.

\begin{proof}

Consider the structure connection of a pre-Frobenius manifold, $\nabla_\lambda$. According to the previous theorem (cf. Theorem\ref{T:Manin}), the manifold $\cM$ is Frobenius if and only if the curvature $\nabla_\lambda^2$ vanishes. This is equivalent to the satisfying the following expression:
\begin{equation}\label{E:*}
    [\nabla_{\lambda, e_i}, \nabla_{\lambda, e_j}](e_k) = \nabla_{\lambda, [e_i, e_j]}(e_k).
\end{equation}

Since $\cM$ is assumed to be associative and $g$ is flat, we only need to consider the $\lambda$-linear terms in the above equation \ref{E:*}. Let $\{e_i\}$ be a basis, and let $\Gamma_{ik}^j$ denote the coefficients of the Riemannian connection:
\begin{equation}\label{E:coeff}
    \nabla_{0, e_i}(e_k) = \sum_j \Gamma_{ik}^j e_j.
\end{equation}

Since the structure connections are given by $\nabla_{\lambda, X}(Y) = \nabla_{0, X}(Y) + \lambda X \circ Y$, the left-hand side of $\ref{E:*}$ produces the $\lambda$-term:
\begin{equation}\label{E:**}
    \big(\nabla_{0, e_i} + \lambda e_i \circ\big)\big(\nabla_{0, e_j} + \lambda e_j \circ\big)(e_k) - \{i \leftrightarrow j\}.
\end{equation}
Inputting Eq. \ref{E:coeff} in Formula \ref{E:**}, we get that the $\lambda$-term reduces to:
\begin{equation}\label{E:3.8}
    \lambda \sum_q \big(\delta_{iq} \Gamma_{jk}^q + \delta_{jk} \Gamma_{ik}^q - \delta_{jk} \Gamma_{ik}^q - \delta_{ik} \Gamma_{jk}^q\big)e_q + \dots
\end{equation}

Now consider the Lie bracket $[e_i, e_j] = \sum\limits_q f_{ij}^q e_q$. The $\lambda$-term in the right-hand side of \ref{E:*} becomes:

\[
\nabla_{\lambda, [e_i, e_j]}(e_k) = \lambda \sum_q f_{ij}^q (e_q \circ e_k) + \dots.
\]

The coefficients of $e_k$ vanish in Eq.\ref{E:3.8}. If $\cM$ carries a Frobenius structure then the equality in \ref{E:*} holds so that $f_{ij}^k=0$.

Thus, the basis elements $e_i$ pairwise commute, and local canonical coordinates $\{u^i\}$ exist.

For the Levi-Civita connection of the metric $g = \sum g_{ij} \, du^i \, du^j$, the connection coefficients are given by:
\[
\Gamma_{ij}^k = \sum_l \Gamma_{ijl} g^{lk},
\]
where:
\[
\Gamma_{ijk} = \frac{1}{2}\big(e_i g_{jk} - e_k g_{ij} + e_j g_{ki}\big).
\]

For the metric $g = \sum \eta_i (du^i)^2$, the non-vanishing coefficients are:
\[
\Gamma_{ii}^i = \frac{1}{2} \eta_i^{-1} e_i \eta_i, \quad \Gamma_{ij}^i = \Gamma_{ji}^i = \frac{1}{2} \eta_i^{-1} e_j \eta_i \quad (i \neq j).
\]

Therefore,
\[
\nabla_i(e_i)= \frac{1}{2}\eta_i^{-1}e_i \eta_i  \cdot  e_i - \Sigma_{i \neq j}\frac{1}{2} \eta_j^{-1} e_j \eta_i \cdot e_j,
\]
and 
\begin{equation}\label{E:****}
    \nabla_i(e_j)= \frac{1}{2}\eta_i^{-1}e_j \eta_i  \cdot  e_i + \frac{1}{2} \eta_j^{-1} e_i \eta_j \cdot e_j,
\end{equation}

Finally, the vanishing of the $\lambda$-terms implies the following fundamental identity, valid for all indices  $i, j, k$:
\begin{equation}\label{E:*6}
e_i \circ \nabla_j(e_k) + \nabla_i(e_j \circ e_k) = (i \leftrightarrow j).
\end{equation}

Applying directly Eq.\ref{E:****} one observes that that Equation~\ref{E:*6} is identically satisfied for $i=j$, as well as in the case $i \neq j \neq k \neq i$. However, when considering the particular case  $i \neq j = k$, one obtains the relation 

\[
e_i \eta_j = e_j \eta_k.
\]
By symmetry, the same condition must hold for $k = i \neq j$, leading to the conclusion that \[\eta_i = e_i \eta\] for some $\eta$, defined at least locally.
This verifies the required condition in all cases and thus establishes the result.
\end{proof}
\subsection{Structure Connection and Curvature} We now turn to an analysis of the geometric properties of the pencil of connections $\nabla_\lambda$, particularly in relation to the structures of a pre-Frobenius and Frobenius manifold.

\, 

Let $\nabla_\lambda$ denote the structure connection associated with the pre-Frobenius manifold 
$(\cM, g, A)$. The curvature of this connection satisfies a quadratic relation in the parameter $\lambda$, taking the form:
\[
\nabla_\lambda^2 = R_1 \lambda^2 + R_2 \lambda + R_3,
\]
where $R_1, R_2, R_3$ are curvature terms determined by the pre-Frobenius structure. A key observation is that the term 
$R_3 $ coincide with $\lambda_0^2$ which satisfies $\lambda_0^2 = 0$,  leading to the simplification:
\[
\nabla_\lambda^2 = R_2 \lambda^2 + R_1 \lambda.
\]

This relation gives rise to the following fundamental theorem, characterizing the Frobenius condition in terms of the vanishing of specific curvature terms.
\begin{theorem}
Let $\nabla_\lambda$ be the structure connection associated with the pre-Frobenius manifold $(\cM, g, A)$. Then:
\begin{itemize}
    \item  The vanishing of $R_1$, i.e. $R_1 = 0$ if and only if $(M, g, A)$ is equivalent to the potentiality of $(M, g, A)$;
    \item The vanishing of $R_2$, i.e. $R_2 = 0$ is equivalent to the associativity of $(\cM, g, A)$.
\end{itemize}
Thus, the manifold $(M, g, A)$ is Frobenius atisfies the full Frobenius condition if and only if the pencil of connections $\nabla_\lambda$ is flat.
\end{theorem}
\begin{ex}\label{Ex:Manin2}
    Write a proof. 
\end{ex}

\subsection{Semisimple Frobenius Manifolds and Webs}
\subsubsection{Webs}
In this subsection, we discuss the construction of flat $3$-webs via semisimple $3$-dimensional Frobenius manifolds and provide a geometric interpretation of the Chern connection associated with these webs. We show that these webs are biholomorphic to the characteristic webs on the solutions of the corresponding associativity equations. Furthermore, these webs are hexagonal and possess at least one infinitesimal symmetry at each singular point.

\, 

Consider local trivial fibrations of class $C^k$. This is given by a triple $(Y,X,\pi)=\lambda$, where $\pi:Y\to X$  is a projection of $Y$ onto $X$ and $Y$ and $X$ are differentiable (smooth) manifolds of respective dimensions $m$ and $n$ where $m > n$. 
\begin{enumerate}
    \item For each point $x \in X$ the set $\pi^{-1}(x)\subset Y$ is a submanifold of dimension $m-n$ which is diffeomorphic to a manifold $F$;
 \item For each point $x \in X$ there exists a neighbourhood $U_x\subset X$ such that $\pi^{-1}(U_x)$ is diffeomorphic to the product $U_x \times F$, and the diffeomorphism between $\pi^{-1}(U_x)$ and  $U_x \times F$  is compatible with $\pi$ and the projection $pr_{U_x} : U_x \times F \to U_x$.
\end{enumerate}

Given our interest  in the local structure of such manifolds we can often think of them as connected domains of a Euclidean space of the same dimension.

\, 

If $T_{p}(Y)$ is the tangent space of $Y$ at a point ${p}$. A fibre $F$ passing through ${p}$ determines in $T_{p}(Y)$ a subspace $T_{p}(F_x)$, of codimension $r$ say, tangent to $F_x$ at ${p}$. We provide $T_{p}(Y)$ with a local moving frame $\{e_i,\,  e_\alpha; i=1,...,r; \alpha =r+1,...,m\}$, where $dim F =m - r$ and $dim X =r$. It is natural then to obtain a co-frame $\{\omega^i, \omega^\alpha\}$ dual to $\{e_i, e_\alpha\}$, such that  $\omega^i(e_{\alpha})=0$. Fibers of the fibration are integral manifolds of the system of equations $$\omega^i=0.$$

Since there exists a unique fibre $F$ through a point ${p}\in Y$, the system above $\{\omega^i=0, i=1,\cdots, r\}$ is completely integrable. By the Frobenius theorem, the integrability condition is given by 
\[d\omega^i=\sum_{j=1}^r\omega^j\wedge \phi^i_j\]
where $\phi^i_j$ are differential forms.

\, 

Assume $Y$ is an $m$-dimensional manifold. A $k$-dimensional distribution $\theta$ on an $m$-dimensional manifold $Y$, $0\leq k\leq m$ is a smooth field of $k$-dimensional tangential directions.  To each point ${p} \in Y$ there is a function which assigns a linear $k$-dimensional subspace of the tangent space $T_{p}(Y)$ to ${p}$. These surfaces are called the leaves of the foliation. The numbers $k$ is the dimension of the foliation.

\, 

Let $Y=X$ be a differentiable manifold of dimension $nr$. We  say that a $d$-web $W(d,n,r)$ of codimension $r$ is given in an open domain $D\subset Y$  by a set of $d$ foliations of codimension $r$ which are in general position. 

\, 
\subsubsection{Some results}
We begin by stating the following two theorems.

\begin{theorem}\label{thm:flatweb1}
Let the implicit cubic ODE
\[
p^3 + a(x, y)p^2 + b(x, y)p + c(x, y) = 0
\]
have a flat web of solutions and satisfy the regularity condition at 
$m = (x_0, y_0, z_0) \in \mathcal{C} \subset \sfS$. Then there exists a local diffeomorphism around $\pi(m) = (x_0, y_0)$ that reduces the ODE to:
\begin{itemize}
    \item $p^3 + px - y = 0$, if $p_0$ is a triple root and $\mathcal{C}$ is Legendrian;
    \item $p^3 + 2xp + y = 0$, if $p_0$ is a triple root and $\mathcal{C}$ is not Legendrian.
\end{itemize}
\end{theorem}

\begin{theorem}\label{thm:flatweb2}
Let the implicit cubic ODE
\[
p^3 + a(x, y)p^2 + b(x, y)p + c(x, y) = 0
\]
have a flat web of solutions and satisfy the regularity condition at 
$m = (x_0, y_0, z_0) \in \mathcal{C} \subset \sfS$. Then there exists a local diffeomorphism around $\pi(m) = (x_0, y_0)$ that reduces the ODE to:
\begin{itemize}
    \item $p^3 + px - y = 0$, if $p_0$ is a triple root and $\mathcal{C}$ is Legendrian;
    \item $p^3 + 2px + y = 0$, if $p_0$ is a triple root and $\mathcal{C}$ is not Legendrian;
    \item $p^2 - y = 0$, if $p_0$ is a double root and $\mathcal{C}$ is Legendrian;
    \item $p^2 - x = 0$, if $p_0$ is a double root and $\mathcal{C}$ is not Legendrian.
\end{itemize}
\end{theorem}

\section{Flat Coordinates and Darboux--Egoroff Coordinates}
The concept of coordinates is foundational in geometry. In this section, we examine flat and canonical coordinates. Flat coordinates play a crucial role in the theory of Frobenius manifolds. 

The problem of finding flat coordinates is essential in studying the Gauss--Manin systems, which are deeply connected to the differential equations for the integrals of basic differential forms over vanishing cycles associated with a given singularity. These systems also relate closely to the theory of primitive forms. The Gauss--Manin connection can be understood as a way to differentiate cohomology classes with respect to parameters.

Let $(\mathcal{U}, x, g)$ be a triple, where $\mathcal{U}$ is a domain in $\mathbb{R}^n$, $x = (x^1, \cdots, x^n)$ are local coordinates, and $g$ is a metric. Let $y = (y^1, \cdots, y^n)$ be another system of coordinates. Since $g$ is a symmetric tensor of rank $2$, the metric can be written in both coordinate systems as:
\[
\tilde{g}(y) = \sum_{k=1}^n \sum_{l=1}^n g_{kl}(x) \frac{\partial x^k}{\partial y^i} \frac{\partial x^l}{\partial y^j}.
\]
We say that the coordinates are flat if the matrix associated with $g$ is constant.

Consider the metric $g$, and let $T_x \cM$ denote the tangent space of the manifold $\cM$ at $x$. Let $v \in T_x \cM$. Define the kernel of the metric $g$ as
\[
\ker(g) := \{ v \in T_x \cM \mid g(v, \cdot) = 0 \}.
\]
At any point $x$, there exist functions $\Gamma_{ijk}$ and $\Gamma_{ij}^k$, defined as:
\[
\Gamma_{ijk} := \frac{1}{2} \left( \frac{\partial g_{im}}{\partial x^j} + \frac{\partial g_{jm}}{\partial x^i} - \frac{\partial g_{ij}}{\partial x^m} \right),
\]
\[
\Gamma_{ij}^k := \frac{1}{2} g^{km} \left( \frac{\partial g_{im}}{\partial x^j} + \frac{\partial g_{jm}}{\partial x^i} - \frac{\partial g_{ij}}{\partial x^m} \right).
\]
Then, the following theorem holds.

\begin{theorem}\label{thm:flatmetric}
For any $ i, j, k $, the condition  
\begin{equation}\label{E:Gamma1}
\sum_{l=1}^n \left( \Gamma_{ij}^l g_{kl} + \Gamma_{kj}^l g_{il} \right) = \frac{\partial g_{ij}}{\partial x^k} 
\end{equation}
is equivalent to the condition:  
\begin{equation}\label{E:Gamma2}
\sum_{l=1}^n \Gamma_{ki}^l v^l = 0, 
\end{equation}
where $v \in \ker(g) $ and $\Gamma_{ki}^i = \Gamma_{ik}^i $.  

If the rank of $ g $ is constant and condition \eqref{E:Gamma2} holds, then there exists a smooth function $ \Gamma_{ki}^i(x) $ satisfying \eqref{E:Gamma1}.  
\end{theorem}

\begin{proof}
Fix a point $x = (x^1, \cdots, x^n)$. Consider the system of equations
\[
\sum_{l=1}^n (\Gamma_{ij}^l g_{kl} + \Gamma_{kj}^l g_{il}) = \frac{\partial g_{ij}}{\partial x^k}.
\]
This is a linear system in the unknowns $\Gamma_{ij}^l$, $\Gamma_{kj}^l$, with $g_{kl}$, $g_{il}$, and $\frac{\partial g_{ij}}{\partial x^k}$ as coefficients. Let $A$ be the coefficient matrix, $y$ the vector of unknowns, and $b$ the vector of constants. The system has a solution if and only if, for every vector $a$ such that $a^\top A = 0$, it holds that $a^\top b = 0$. This leads to:
\[
\sum_{l=1}^n g_{lk} \Gamma_{ij}^l = \Gamma_{ij, k}.
\]
From this, we deduce that:
\[
\sum_{l=1}^n \Gamma_{ki}^l v^l = 0.
\]
\end{proof}

Finally, we have the following theorem characterizing the existence of flat coordinates.

\begin{theorem}\label{thm:flatcoord}
Flat coordinates for the metric $g$ exist if and only if there exist smooth functions $\Gamma_{ij}^k(x)$, symmetric in $i$ and $j$, satisfying:
\begin{itemize}
    \item The compatibility condition:
    \[
    \sum_{l=1}^n (\Gamma_{jk}^l g_{il} + \Gamma_{ik}^l g_{jl}) = \frac{\partial g_{ij}}{\partial x^k};
    \]
    \item The vanishing of the Riemann curvature tensor:
    \[
    R_{ijkm} = 0,
    \]
    where
    \[
    R_{ijkm} := \sum_l g_{il} \left( \frac{\partial}{\partial x^k} \Gamma_{jm}^l - \frac{\partial}{\partial x^m} \Gamma_{jk}^l + 
    \sum_t (\Gamma_{kt}^l \Gamma_{mj}^t - \Gamma_{mt}^l \Gamma_{jk}^t) \right).
    \]
\end{itemize}
\end{theorem}

\section{Exercises}
Exercise \ref{Ex:preLie}. Using the relations in \eqref{E:1} and \eqref{E:2}, it is easy to see that one can obtain the structure of a pre-Lie algebra on the tangent sheaf.

\, 

Exercise \ref{Ex:Ma} is easily shown by using the paper of Calabi (1954)\cite{Ca54} and the definition of a potential pre-Frobenius manifold.

\,

Exercises \ref{Ex:ManinProof} and \ref{Ex:Manin2}. The proof proceeds in two stages.
\noindent\textbf{Part 1.}  

\, 

We calculate the coefficient of the $\lambda$ term, $R_1$, in the following expression:
\[
\big[\nabla_{0, \partial_a} + \lambda \partial_a \circ, \nabla_{0, \partial_b} + \lambda \partial_b \circ\big](\partial_c).
\]
It follows immediately that $R_1 = 0$ if and only if, for any $a, b, c, d$ and for a mixed $(1, 2)$ rank tensor $A_{bc}^e$,  
\[
\partial_a A_{bc}^e = (-1)^{ab} \partial_b A_{ac}^e,
\]
or equivalently, for a rank-3 covariant tensor:
\[
\partial_a A_{abc} = (-1)^{ab} \partial_b A_{aac}.
\]

When $A$ is a potential, the symmetry of the rank-3 tensor $A(X, Y, Z)$, written as $A(X, Y, Z) = (XYZ)\Phi$, ensures that the above condition holds.

Suppose that the relation $\partial_a A_{abc} = (-1)^{ab} \partial_b A_{aac}$ is true. Then, for all $c, d$, the form $\sigma_b \, dx^b \, A_{bcd}$ is closed. Locally, we can find functions $B_{cd}$ satisfying $B_{cd} = (-1)^{cd} B_{dc}$. Taking into account the symmetry of the rank-3 covariant tensor $A$, we obtain:
\[
A_{bcd} = \partial_b B_{cd} = (-1)^{bc} \partial_c B_{bd} = (-1)^{bc} A_{cbd}.
\]
Hence, for any $d$, the expression $\sigma_c \, dx^c \, B_{cd}$ is closed. Analogously, we find locally $B_{cd} = \partial_c C_d$. Since $C_d = \partial_d \Phi$, it follows that $A_{bcd} = \partial_b \partial_c \partial_d \Phi$.

This completes the first part of the proof.

\noindent\textbf{Part 2.}  
We compute the coefficient of the $\lambda^2$ term in $[\nabla_{\lambda, X}, \nabla_{\lambda, Y}](Z)$:
\[
R_{2, XY}(Z) = X \circ (Y \circ Z) - (-1)^{\overline{X}\, \overline{Y}} Y \circ (X \circ Z).
\]
If the multiplication $\circ$ is associative, then $R_2 = 0$ because $\circ$ is always commutative. 

Conversely, if $R_2 = 0$, then:
\[
X \circ (Y \circ Z) = (-1)^{\tilde{X} \tilde{Y}} Y \circ (X \circ Z) = (-1)^{\tilde{X} (\tilde{Y} + \tilde{Z})} (Y \circ X) \circ Z.
\]
Thus, associativity of $\circ$ follows. This completes the proof of the theorem.

\chapter{Unveiling the Hidden Geometry of Statistical Manifolds}
In this chapter, and the following one, we reveal the intricate (and often hidden!) geometric structures underlying statistical manifolds. In particular, we revisit certain geometric concepts that have long been overlooked, such as $m$-pairs, and demonstrate their relevance in this context. By bringing these ideas back into focus, we aim to provide a deeper understanding of the rich interplay between geometry and statistical structures.
\section{Projective geometry, $m$-pairs, Grassmannians}
\subsection{$m$-pairs}\label{S:4.3}
Let us consider an object $X_{d}$, a $d$-dimensional surface residing in an $n$-dimensional projective space $\bbP^{n}$, where the constraint $d \leq n$ holds. 

\,

\begin{definition}
We say that the surface $X_{d}$ is \textit{normalized} if, for every point $p \in X_{d}$, we associate two distinct hyperplanes:
\begin{enumerate}
\item Normal of first type, $P_{I}$, is of dimension $(n-d)$ and intersects the tangent $d$-plane $T_{p}X_{d}$ at a unique point $p$.
\item Normal of the second type, $P_{II}$, is of dimension $(d-1)$, situated  within the $d$-plane $T_{p}X_d$, and does not pass through the point $p$.
\end{enumerate}
\end{definition}

\begin{example}
    Consider the case where $d = 2$ and  $n = 3 $. This means we have a 2-dimensional surface $ X_2$  within the 3-dimensional projective space $\bbP^{3}$.

\, 

At a point $p$ on $X_2 $:

\begin{itemize}
    \item The tangent plane  $T_{p}X_{2}$ is a 2-dimensional plane in the ambient $\bbP^{3}$ space.
    \item The normal of first type   $P_{I} $ would be a line that intersects the tangent plane   $T_{p}X_{2}$  at the point   $p$.
    \item The normal of second type  $P_{II}$ is another line that lies completely within the tangent plane  $ T_{p}X_{2}$ but does not contain the point  $p$.
    \end{itemize}
    \end{example}

This embodies a duality, intrinsic to projective geometry. Notably, in the  case where $d = n$, the hyperplane $P_{I}$ identifies to the point $p$, and $P_{II}$ becomes the $(n-1)$-dimensional surface devoid of the point $p$. This situation reflects the classical notion of duality in projective spaces, leading to the identification of $X_{n}$ with the projective space $\bbP^{n}$.

\subsection{$m$-pairs}
\begin{definition}\label{D:mpairs}
We define an \textit{$m$-pair} as a pair constituted of an $m$-plane and an $(n-m-1)$-plane.
\end{definition}

More precisely, an \textbf{m-pair} is a pair consisting of:
\begin{enumerate}
    \item An \textbf{m-plane}: A linear subspace of dimension $m $ in an $n $-dimensional space.
    \item An \textbf{(n-m-1)-plane}: A linear subspace of dimension $n-m-1$ in the same $n$-dimensional space.
\end{enumerate}

These two planes are typically considered in the context of projective geometry or linear algebra, where they may satisfy certain geometric or algebraic relationships.
\begin{ex}
Draw an $m$-pair in the low dimensional cases. 
\end{ex}

\subsection{Observations and Examples}
\subsection*{3D Space ($n = 3 $)}
\begin{itemize}
    \item \textbf{0-pair}: A point (0-plane) and a plane (2-plane). This can be identified with the projective space $\bbP^3$
    \item \textbf{1-pair}: A line (1-plane) and another line (1-plane). Two lines in 3D space can either intersect at a point, be parallel (not intersecting but lying on a common plane), or be skew (not intersecting and not parallel).
\end{itemize}

\subsection*{4D Space ($n = 4 $)}
\begin{itemize}
    \item \textbf{1-pair}: A line (1-plane) and a plane (2-plane). In 4D space, a line and a plane can intersect at a point, not intersect at all, or intersect along a line if the line lies entirely within the plane.
    \item \textbf{2-pair}: Two planes (2-planes). In 4D space, two planes can intersect at a point, along a line, or not intersect at all.
\end{itemize}

\subsection*{Projective Geometry}
In projective geometry, an $m $-pair can be used to describe configurations of points, lines, and planes at infinity. 

\subsection*{Grassmannians}
The concept of $m $-pairs is closely related to Grassmannians, which are spaces that parameterize all $m $-dimensional subspaces of an $n$-dimensional space. The study of $m $-pairs can be seen as a way to explore the relationships between different Grassmannians.

\begin{ex}
Determine the relation between Grassmanianns and $m$-pairs. 
\end{ex}
\subsection*{Applications in Computer Vision}
In computer vision, $m$-pairs can be used to model the relationship between different views of a scene. For example, in structure from motion, the relationship between 2D image planes (2-planes) and 3D space (3-planes) can be analyzed using the concept of $m $-pairs.

\begin{ex}
Generate an example using your favourite software $m$-pairs modelisin the relationship between different views of a scene.
\end{ex}

 The examples above illustrate how $m $-pairs can be applied in various contexts, from projective geometry to computer vision.\,

\subsection{Some properties}

Normalized surfaces associated with an $m$-pair space possesses the following properties:
\medskip
\begin{lemma}\label{L:pairs}
\
\begin{enumerate}
\item The collection of $m$-pairs forms a projective differentiable manifold.
\item For any integer $m \geq 0$, a manifold of $m-$pairs contains two flat, affine, and symmetric connections.
\end{enumerate}
\end{lemma}

\section{Paracomplex numbers and modules}
Paracomplex numbers are a generalization of complex numbers, where instead of the imaginary unit   i    satisfying  $ i^2 = -1 $  , the paracomplex unit   $\epsilon $   satisfies  $ \epsilon^2 = 1 $  . The algebra of paracomplex numbers is defined as follows:

\subsection{Definition}
A \textbf{paracomplex number} is an element of the form:
\[
z = x + \epsilon y,
\]
where  $ x, y \in \mathbb{R} $   are real numbers, and  $ \epsilon  $   is the paracomplex unit satisfying:
\[
\epsilon ^2 = 1.
\]
The set of all paracomplex numbers is denoted by   $\mathbb{P}$   .

\subsection{Algebraic Structure}
The algebra of paracomplex numbers   $\mathfrak{P} $  is a two-dimensional commutative algebra over the real numbers   $\mathbb{R}$   . It is isomorphic to the direct sum  $ \mathbb{R} \oplus \mathbb{R} $  , and its multiplication rule is given by:
\[
(x_1 + \epsilon y_1)(x_2 + \epsilon y_2) = (x_1x_2 + y_1y_2) + \epsilon (x_1y_2 + x_2y_1).
\]

\subsection{Properties}
\begin{itemize}
    \item \textbf{Idempotent Basis}: Paracomplex numbers can be expressed in terms of idempotent elements. Define:
    \[
    e_+ = \frac{1 + \epsilon }{2}, \quad e_- = \frac{1 - \epsilon }{2}.
    \]
    These elements satisfy  $ e_+^2 = e_+ $  ,   $e_-^2 = e_- $  , and   $e_+ e_- = 0 $ . Any paracomplex number   $z = x + \epsilon y  $  can be written as:
    \[
    z = (x + y)e_+ + (x - y)e_-.
    \]

    \item \textbf{Conjugation}: The paracomplex conjugate of  $ z = x + \epsilon y $   is defined as:
    \[
    \overline{z} = x - \epsilon y.
    \]

    \item \textbf{Norm}: The norm of a paracomplex number   $z = x + \epsilon y $   is given by:
    \[
    \|z\| = z \overline{z} = x^2 - y^2.
    \]
    Note that this norm is not positive definite, as it can take negative values.
\end{itemize}

\begin{ex}
Show that $(a+a\epsilon)\cdot(b-b\epsilon)=0$.
\end{ex}
\section{Modules over Paracomplex Algebras}
A module over a paracomplex algebra generalizes the concept of a vector space, where the scalars are paracomplex numbers instead of real or complex numbers.

\subsection{Definition}
Let  $ \mathfrak{P}$ be the algebra of paracomplex numbers. A \textbf{module over  $ \mathfrak{P}$} is an abelian group   M    together with a scalar multiplication:
\[
\cdot :  \mathfrak{P} \times M \to M,
\]
satisfying the following properties for all   $z, z_1, z_2 \in  \mathfrak{P}$     and   $m, m_1, m_2 \in M $  :
\begin{enumerate}
    \item   $z \cdot (m_1 + m_2) = z \cdot m_1 + z \cdot m_2   $,
    \item  $ (z_1 + z_2) \cdot m = z_1 \cdot m + z_2 \cdot m $  ,
    \item   $(z_1 z_2) \cdot m = z_1 \cdot (z_2 \cdot m) $  ,
    \item  $ 1 \cdot m = m $  .
\end{enumerate}

\subsection{Examples}
\begin{itemize}
    \item \textbf{Paracomplex Vector Space}: The simplest example of a module over   $ \mathfrak{P}$     is  $ \mathfrak{P}^n$   , the set of   n   -tuples of paracomplex numbers. Scalar multiplication is defined component-wise:
    \[
    z \cdot (z_1, z_2, \dots, z_n) = (z z_1, z z_2, \dots, z z_n).
    \]

    \item \textbf{Decomposition into Real Submodules}: Using the idempotent basis   $e_+ $   and  $ e_- $  , any module  $ M$    over   $\mathfrak{P} $   can be decomposed into two real submodules:
    \[
    M = e_+ M \oplus e_- M.
    \]
    Here,  $ e_+ M $   and  $ e_- M  $  are real vector spaces, and the action of  $\mathfrak{P}$  on  $ M $   is determined by the actions of  $ e_+  $  and $  e_- $  .
\end{itemize}

\subsection{Applications}
Modules over paracomplex algebras appear in various areas of mathematics and physics, including:
\begin{itemize}
    \item \textbf{Geometry}: Paracomplex structures are used in the study of para-Hermitian and para-Kähler manifolds.
    \item \textbf{Physics}: Paracomplex numbers and modules are used in the study of supersymmetry and integrable systems.
\end{itemize}

\section{Applications to statistical manifolds}

The algebra of paracomplex numbers has a remarkable incidence on the manifold of probability distributions. We discuss this in the following propositions and statements. 
This leads us to the following salient proposition:
\begin{proposition}\label{P:isome}
The space of $0$-pairs within the projective space $\bbP^{n}$ is isometric to the hermitian projective space over the algebra of paracomplex numbers.
\end{proposition}
\begin{proof}
Refer to section 4.4.5 of~\cite{Ro97} for detailed proof.
\end{proof}
\medskip
\begin{proposition}\label{P:zero}
Let $(X, \mathcal{F})$ be a finite measurable set with dimension $n+1$, where measures vanish exclusively on an ideal $\mathcal{I}$. Define $\mathcal{H}_{n}$ as the space of probability distributions on $(X, \mathcal{F})$. It follows that the space $\mathcal{H}_{n}$ embodies a manifold of $0$-pairs.
\end{proposition}
\begin{proof}
The $n$-dimensional surface $\mathcal{H}_{n}$ arises as the intersection of the hyperplane constrained by $\mu(X) = 1$ and the cone $\mathcal{C}_{n+1}$ of strictly positive measures within the affine space $\mathcal{W}_{n+1}$ of signed bounded measures. It is interpreted as an $n$-dimensional surface within the projective space $\bbP^{n}$. The geometrical structure of this surface is thus inherited from projective geometry. By invoking the remark from the initial paragraph of section 0.4.3 in\cite{Ro97} alongside definition\ref{D:mpairs} of $0$-pairs, we conclude the correspondence with a manifold of $0$-pairs.
\end{proof}
\medskip
\begin{theorem}\label{Th:main}
Consider $(X, \mathcal{F})$ as a finite measurable set with dimension $n+1$, where measures vanish solely on an ideal $\mathcal{I}$. The space $\mathcal{H}_{n}$ of probability distributions on $(X,\mathcal{F})$ is isomorphic to the hermitian projective space over the cone $M_{+}(2,\fC)$.
\end{theorem}

\chapter{Statistical Frobenius manifold and Learning}

\section{Statistical Gromov--Witten invariants and learning}

\vskip-.2cm

We introduce Gromov–Witten invariants for statistical manifolds (denoted GWS), extending an analog of classical Gromov–Witten invariants to the realm of information geometry. Originally, these invariants are rational numbers that enumerate (pseudo-)holomorphic curves satisfying specific conditions in a symplectic manifold. In our generalization, GWs encode fundamental geometric structures of statistical manifolds, reflecting the intersection theory of (para-)holomorphic curves within this framework.

\,

Furthermore, this perspective reveals an intrinsic connection between the geometry of statistical learning and the dynamics of the learning process. The presence or obstruction of certain pseudo-holomorphic structures, as captured by GWs, provides a criterion for determining whether a learning system successfully acquires information or encounters fundamental limitations, thereby offering a novel geometric approach to the theory of learning.

\,

\subsection{Brief Recollections of Gromov–Witten Invariants}

Gromov–Witten invariants are fundamental numerical invariants in symplectic geometry and algebraic geometry, capturing intersection properties of (pseudo-)holomorphic curves in a given space.

Given a compact symplectic manifold $(\cM,\varpi)$, the Gromov–Witten invariant counts the number of (pseudo-)holomorphic maps
\[u:(\mathscr{S},j)\to (\cM,\varpi),\]
from a compact Riemann surface $\mathscr{S}$ (with complex structure 
$j$) into $\cM$, satisfying certain constraints on their homology class and intersection conditions with given cycles.

\begin{itemize}
    \item This notion appears in {\bf Enumerative Geometry}:
Gromov–Witten invariants count the number of holomorphic curves passing through prescribed points or satisfying intersection constraints.
\item {\bf Quantum Cohomology:}
They define a deformation of classical cohomology, giving rise to a quantum product that encodes curve counts in a ring structure.
\item {\bf Holomorphic Curve Moduli Space}:
The counts arise from integration over the moduli space of stable maps. We refer to \cite{Man99} for more information on this topic. \end{itemize}

The point of view that we adopt, here, is inspired from quantum cohomology, which is a formal Frobenius manifold. 

\,  

Let us consider the (formal) Frobenius manifold $(H,g)$. We denote by $k$ a field of characteristic 0 (such as $\bbC$ or $\bbR$). Let $H$ be a $k$-module of finite rank and \[g:H\otimes H\to k,\]  an even symmetric pairing (which is non degenerate). We denote $H^*$ the dual to $H$. 

\,

An important part of the {\bf Frobenius manifold structure} is encoded in the existence of a {\bf potential function} 
\[{\bf \Phi} \in k[[H^*]],\]
which governs the multiplication structure on the manifold. In local coordinates, under suitable conditions, this function can be expressed as: 
\[{\bf \Phi}=\sum_{n\geq 3}\frac{1}{n!}Y_n,\] 
where $Y_n\in (H^*)^{\otimes n}$ is a symmetric multilinear map 
\[Y_n\in (H^*)^{\otimes n},\quad Y_n: H^{\otimes n} \to k.\] 
This system of multilinear forms defines a system of {\it abstract correlation functions} on the pair $(H,g)$, where $H$ is a vector space equipped with a non-degenerate pairing $g$. These functions are symmetric and (in the context of Gromov–Witten theory) correspond to intersection numbers on the moduli space of stable maps. The Gromov--Witten invariants are generated from those multi-linear maps. 

\,

In the context of Gromov–Witten invariants, the potential function 
$\Phi$ serves as the generating function for the intersection numbers of moduli spaces of holomorphic curves. The symmetric multilinear maps $Y_n$ correspond to the correlation functions computed via topological field theory techniques.
\, 

 The function $\Phi$ satisfies the WDVV equations (associativity conditions on quantum cohomology), which govern the Frobenius manifold structure.

\,

The maps $Y_n$ define the higher-order correlation functions, whose values give the Gromov–Witten invariants.

\, 

This formulation provides a bridge between Frobenius manifolds, quantum cohomology, and Gromov–Witten theory, showing how the potential function encodes geometric intersection theory in an algebraic and formal power series framework. The abstract correlation functions $Y_n$	
  serve as the structural foundation from which the Gromov–Witten invariants emerge, linking the geometry of moduli spaces with the algebraic structure of Frobenius manifolds.

  \subsection{Links to statistics}
In the framework of statistical geometry, we return to the study of statistical manifolds, emphasizing the discrete case of the exponential family. The fundamental relation governing this structure is given by the expansion:
\begin{equation}\label{E:2}\sum_{\omega\in \Omega} \exp\{-\sum \beta^jX_j(\omega)\}=
\sum_{\omega\in \Omega}\sum_{m\geq 1}\frac{1}{m!}\left\{ -\sum_{j} \beta^jX_j(\omega)\right\}^{\otimes m}, \end{equation}
where:

\begin{itemize}
    \item The parameter $\beta$, given by $\beta=(\beta_0,....,\beta_n)\in \mathbb{R}^{n+1}$, provides an   affine canonical parametrisation.
    \item The objects $X_j(\omega)$ are directional co-vectors, 
    forming a finite set of necessary and sufficient statistics, denoted by $\mathcal{X}_n$. 
    \item The co-vectors, $X_1(\omega),\dots ,X_{n}(\omega)$ are linearly independent co-vectors and we impose the normalization $X_0(\omega)\equiv 1$. 
\end{itemize}
The family in (\ref{E:2}) defines an analytic $n$-dimensional hypersurface within the statistical manifold, which can be uniquely determined by $n+1$ points in general position. 

\subsection{Gromov--Witten Invariants for Statistical Manifolds}
We introduce the notion of Gromov–Witten invariants for statistical manifolds (GWS) as follows:
\begin{definition}
Let $k$ be the field of real numbers. Let $\sfS$ be the statistical manifold. The Gromov--Witten invariants for statistical manifolds (GWS) are defined from the family of multilinear maps:
 \[\tilde{Y}_n:\sfS^{\otimes n}\to k.\] 
\end{definition}
Equivalently, these invariants may be written in terms of the generating expansion:
\[\tilde{Y}_n\in \left(-\sum_{j}\beta^jX_j(\omega)\right)^{\otimes n}.\]

\subsection{Interpretation via the Relative Entropy Function}
These invariants naturally emerge as part of the potential function
$\tilde{\bf \Phi}$, which is identified with the Kullback--Liebler entropy function of the statistical system. The entropy function itself is expressed in the form: 

\begin{equation}\label{E:3}\tilde{\bf \Phi}= \ln \sum_{\omega\in\Omega} \exp{(-\sum_{j}\beta^jX_j(\omega))}.\end{equation}

This formulation suggests a deeper geometric and categorical interpretation of statistical learning, where the intersection theory of statistical structures plays a fundamental role. Within this perspective, the entropy function $\tilde{\bf \Phi}$ governs the geometry of statistical families, much like the potential function in Frobenius manifolds or Gromov–Witten theory encodes intersection numbers in moduli spaces of holomorphic curves.

Therefore, we state the following:

\begin{proposition}
The entropy function $\tilde{\bf \Phi}$ of the statistical manifold is intrinsically determined by the Gromov–Witten invariants for statistical manifolds (GWS).
\end{proposition}
More precisely, $\tilde{\bf \Phi}$ arises as a generating function whose coefficients encode the multilinear maps $\tilde{Y_n}$, which define the GWS structure.  These invariants characterize the underlying statistical geometry by capturing the intersection properties of statistical hypersurfaces.

\begin{proof}
Indeed, since $\tilde{\bf \Phi}$, in formula (\ref{E:3}) relies on the polylinear maps $\tilde{Y}_n\in \left(-\sum_{j}\beta^jX_j(\omega)\right)^{\otimes n}$,
 defining the (GWS), the statement follows. 
\end{proof}

\section{Learning}
We consider the tangent fiber bundle over the statistical manifold 
$\sfS$, where $\sfS$ is the space of probability distributions. This bundle structure encodes the infinitesimal geometry of the statistical space, allowing us to describe variations in probability distributions in terms of a Lie group action.
\subsection{Tangent Fiber Bundle Structure}
The tangent fiber bundle is denoted by the quintuple $(T\sfS,\sfS,\pi,G,F)$, where:
\begin{itemize}
    \item $T\sfS$ is the total space of the tangent bundle, consisting of all possible tangent vectors to points in $\sfS$.

    \item $\pi: T\sfS\to \sfS$ is a continuous surjective map that projects each tangent vector to its base point in $\sfS$.
    \item  $F$ is the fiber,  representing the space of allowable tangent vectors at each point of $\sfS$. 
    \item $G$ is a Lie group that acts on the fibers, encoding the parallel transport structure within the statistical manifold.
    
\end{itemize}
\subsection{Tangent Spaces and the Space of Measures}

For any point $\rho\in \sfS$, the tangent space at $\rho$, denoted 
$T_{\rho}\sfS$, is given by: 
\[T_{\rho}S \cong \{\text{bounded, signed measures vanishing on an ideal $I$ of the $\sigma-$algebra}\}.\]

This identification follows from the fact that infinitesimal perturbations of a probability distribution $\rho$ can be described by signed measures that respect the probabilistic constraints imposed by 
$\sfS$. The ideal 
$I$ of the 
$\sigma$-algebra corresponds to the subspace of measures that do not contribute to the variations in probability distributions, ensuring consistency with the underlying measure-theoretic structure.

\subsection{Lie Group Action on the Fibers}
The Lie group $G$  acts freely and transitively on each fiber of $T\sfS$, meaning that every element of the fiber can be transformed into any other through the group action. The action is given by:

\[f\overset{h}{\mapsto} f+h,\]
where:
\begin{itemize}
    \item $f$ is an element of the total space $T\sfS$ i.e., a tangent vector at some $\rho\in \sfS$.

    \item $h$ is a parallel transport within the statistical manifold, representing an infinitesimal displacement in the space of probability distributions.
\end{itemize}

This affine structure on the fibers implies that the action of $G$
acts as a translation group, ensuring a well-defined parallel transport mechanism in the space of probability measures. Such a structure is crucial for describing information geometry, as it encodes how probability distributions evolve under statistical transformations.


\begin{lemma}
Consider the fiber bundle $(T\sfS,\sfS,\pi,G,F)$ where: 

Let path $\gamma:\cI\to \sfS$ be a smooth geodesic path in $\sfS$, where $\cI\subset \bbR$. 
The fiber over $\gamma$ is denoted by $F_{\gamma}=\pi^{-1}(\gamma)$, which represents the space of tangent vectors along $\gamma$. 

Then, the fiber \[F_{\gamma}=\gamma^+\sqcup\gamma^{-1}\] consists of two disjoint connected components. Each component $\gamma^+$ and $\gamma^-$ is contained within a totally geodesic submanifold of $T\sfS$, denoted $E^+$ and $E^-$, respectively. 
\end{lemma}
\begin{proof}
Consider the fiber above $\gamma$. Since for any point of $\sfS$, its the tangent space is identified to  module over paracomplex numbers. This space is decomposed into a pair of subspaces (i.e. eigenspaces with eigenvalues $\pm \e$).
The geodesic curve in $\sfS$ is a path such that $\gamma=(\gamma^i(t)): t\in [0,1]\to \sfS$. In local coordinates, the fiber budle is given by $\{\gamma^{ia}e_{a}\}$, and $a\in \{1,2\}$. Therefore, the fiber over $\gamma$ has two components $(\gamma^+,\gamma^-)$. Taking the canonical basis for $\{e_1,e_2\}$, implies that $(\gamma^+,\gamma^-)$ lie respectively in the subspaces $E^+$ and $E^-$. These submanifolds are totally geodesic in virtue of  Lemma 3 in \cite{CoCoNen}. \end{proof}

\subsection{Learning Process via the Ackley--Hinton--Sejnowski method}

We define a learning process in terms of the Ackley--Hinton--Sejnowski method \cite{AHS}, which is based on minimizing the Kullback--Leibler divergence as a measure of distance between probability distributions. This process can be interpreted in a geometric framework as follows:

\begin{proposition}[Geometric Formulation of Learning Process] \label{P:CoNen}
The learning process consists of determining whether there exist intersections between the \emph{paraholomorphic curve} $\gamma^+ $ and the \emph{orthogonal projection} of $\gamma^- $ into the subspace $E^+ $.
\end{proposition}

More precisely, let $\sfS$ be a \emph{statistical manifold} equipped with a fiber bundle structure $(T\sfS, \sfS, \pi, G, F) $. Consider a geodesic path $\gamma: I \to \sfS $, parametrized by an interval $I \subset \mathbb{R}$. The fiber over $\gamma$, denoted by $F_{\gamma}$, is assumed to decompose into two connected components:
\[
    F_{\gamma} = \gamma^+ \sqcup \gamma^-.
\]
Each component $\gamma^+ $ and $\gamma^- $ is contained in a \emph{totally geodesic submanifold} of $TS $, denoted by $E^+ $ and $E^- $, respectively:
\[
    \gamma^+ \subset E^+, \quad \gamma^- \subset E^-.
\]

In particular, the learning process is considered successful whenever the distance between the geodesic $\gamma^+ $ and its orthogonal projection into $E^+ $ decreases towards zero. That is, the process converges if:
\[
    d(\gamma^+, \pi_{E^+}(\gamma^-)) \to 0, \quad \text{as the learning iterations progress}.
\]

This formulation provides a rigorous geometric criterion for assessing the success of learning, leveraging the underlying differential geometry of the statistical manifold.

In other words: 
\begin{proposition}\label{P:CoNen}
The learning process consists in determining if there exist intersections of the paraholomorphic curve $\gamma^+$ with the orthogonal projection of the curve $\gamma^-$ in the subspace $E^+$. 
\end{proposition} 

More formally, as was depicted in\, \cite{BCN99} (sec. 3) let us denote by $\Upsilon$ the set of (centered) random variables over $(\Omega,\mathcal{F},P_{\theta})$
which admit an expansion in terms of the scores under the following form:
\[\Upsilon_P= \{X\in \mathbb{R}^{\Omega}\, |\, X-\mathbb{E}_P[X]=g^{-1}(\mathbb{E}_P[Xd\ell]), d\ell \}. \]

By direct calculation, one finds that the log-likelihood $\ell= ln\rho$ of the usual (parametric) families of probability distributions belongs to 
$\Upsilon_p$  as well as the difference $\ell -\ell^*$  of log-likelihood of two probabilities of the same family. 

\, 

Being given a family of probability distributions such that $\ell \in  \Upsilon_P$ for any $P$, let $\mathcal{U}_P$, let us denote $P^*$ the set
 such that $\ell-\ell^*\in  \Upsilon_p$. Then, for any $P^*\in  \mathcal{U}_p$, we define $K(P,P^*)=\mathbb{E}_P[\ell - \ell^*]$. 

\begin{theorem}
Let $\sfS$ be  statistical manifold, equipped with a Riemannian metric and an affine connection. Then, the (GWS) determines the evolution of the learning process, through the associated geometric constraints. 
\end{theorem}
\begin{proof}
Whenever there is a successful learning, the distance between the curve $\gamma^+$ and the projection of  $\gamma^-$ on $E^+$ tends to be as small as possible. This implies that $K(P,P^*)=\mathbb{E}_P[\ell - \ell^*]$, so that $K(P,P^*)$ is minimized.

The learning process is by definition given by a {\it deformation} of a pair of geodesics, defined respectively in the pair of totally geodesic manifolds $E^+, E^-$. The (GWS), arise in the $\tilde{Y}_n$ in the potential function $\tilde{\bf \Phi}$, which is directly related to the relative entropy function $K(P,P^*)$. Therefore, it is easy to conclude that the (GWS) determine the learning process.
 \end{proof}

Similarly as in the classical (GW) case, the (GWS) count intersection numbers of the para-holomorphic curves generated by $\gamma^+$ and $\gamma^-$. In fact, we have the following statement:
\begin{corollary}
Let $(T\sfS,\sfS,\pi,G,F)$ be the fiber bundle above, where: 
\begin{itemize}
    \item $\pi:T\sfS\to \sfS$ is the projection map;
    \item $G$ is the Lie group;
    \item $F$ is a fiber.
\end{itemize}
Let $\gamma^-\subset T\sfS$ be a geodesic in the tangent bundle with respect to the affine connection and let $\gamma^+\subset E^+$ be a geodesic in the sub-bundle $E^+\subset T\sfS$. Then, the (GWS) determine the number of intersections of the projection of  $\gamma^-$ onto $E^+$, with the geodesic $\gamma^+$.
\end{corollary}

\appendix

\backmatter

\end{document}